\documentclass{amsart}

\usepackage{amssymb,amsmath,amscd,amsthm,xspace,color,tocvsec2, array, 
	tikz-cd, enumitem}
\usepackage[mathscr]{euscript}
\usepackage[breaklinks=true]{hyperref} %To get clickable links to papers
\usepackage[letterpaper]{geometry}
\usepackage[all, cmtip]{xy}
\usepackage[ampersand]{easylist}
\usepackage{stmaryrd}
\usepackage[myheadings]{fullpage}
\usepackage{csquotes}
\usepackage{mathtools}
\usepackage[skip=5pt, indent=10pt]{parskip}

\usepackage{tikz}
\usetikzlibrary{cd}

\newcommand{\sss}{\subsubsection{}}

\newcommand{\beq}{\begin{equation}}
\newcommand{\eeq}{\end{equation}}
\newcommand{\on}{\mathrm}

\numberwithin{equation}{section}
\theoremstyle{plain}
\newtheorem{Thm}{Theorem}[section]
\newtheorem{Prop}[Thm]{Proposition}
\newtheorem{Cor}[Thm]{Corollary}
\newtheorem{Lem}[Thm]{Lemma}
\theoremstyle{definition}

\newtheorem{Rem}[Thm]{Remark}

\DeclareMathOperator{\Rep}{Rep}

\DeclareMathOperator{\Hom}{Hom}
\DeclareMathOperator{\Funct}{Funct}
\DeclareMathOperator{\Coh}{Coh}

\DeclareMathOperator{\QCoh}{QCoh}
\DeclareMathOperator{\IndCoh}{IndCoh}

\DeclareMathOperator{\Ind}{Ind}
\DeclareMathOperator{\gr}{gr}

\DeclareMathOperator{\End}{End}
\DeclareMathOperator{\Ext}{Ext}

\DeclareMathOperator{\Spec}{Spec}

\DeclareMathOperator{\Mod}{Mod}
\DeclareMathOperator{\LMod}{LMod}
\DeclareMathOperator{\RMod}{RMod}

\DeclareMathOperator{\id}{id}

\newcommand{\Inj}{\mathop{\mathrm{Inj}}}
\newcommand{\noeth}{{\mathrm{noeth}}}
\newcommand{\coh}{{\mathrm{coh}}}
\newcommand{\Map}{\mathrm{Map}}

\newcommand{\OO}{\mathscr{O}}
\newcommand{\sC}{\mathscr{C}}
\newcommand{\sS}{\mathscr{S}}
\newcommand{\sD}{\mathscr{D}}
\newcommand{\sJ}{\mathscr{J}}
\newcommand{\sK}{\mathscr{K}}

\newcommand{\sA}{\mathscr{A}}
\newcommand{\sB}{\mathscr{B}}
\newcommand{\sM}{\mathscr{M}}

\newcommand{\bs}{\backslash}
\newcommand{\sN}{\mathscr{N}}

\newcommand{\Vect}{\on{Vect}}

\newcommand{\sP}{\mathscr{P}}
\newcommand{\sH}{\mathscr{H}}

\newcommand{\DGCat}{\on{DGCat}_{\mathrm{cont},k}}

\newcommand{\fh}{\mathfrak{h}}
\newcommand{\fk}{\mathfrak{k}}

\newcommand{\Dmod}{\on{D}\text{-}\on{mod}}

\newcommand{\sT}{\mathscr{T}}

\renewcommand{\mod}{\text{-}\on{mod}}
\newcommand{\modr}{\on{mod}\text{-}}
\newcommand{\hc}{\on{HCh}}

\renewcommand{\and}{\quad \on{and} \quad}

\newcommand{\h}{\mathfrak{h}}

\newcommand{\Z}{\mathbb{Z}}
\newcommand{\im}{\operatorname{im}}

\newcommand{\ind}{\mathrm{ind}}
\newcommand{\inv}{\mathrm{inv}}
\newcommand{\op}{\mathrm{op}}
\newcommand{\rev}{\mathrm{rev}}
\newcommand{\Ch}{\mathrm{Ch}}
\newcommand{\Chb}{\mathrm{Ch}^{\mathrm{b}}}
\newcommand{\Db}{D^{\mathrm{b}}}

\newcommand{\comp}{\mathrm{c}}
\newcommand{\simto}{\xrightarrow{\sim}}

\newcommand{\BB}{\mathrm{B}}
\newcommand{\fd}{\mathrm{fd}}

\newcommand{\PrL}{\mathrm{Pr}^{\mathrm{L}}}
\newcommand{\PrLSt}{\mathrm{Pr}^{\mathrm{L}}_{\mathrm{St}}}
\newcommand{\Catoo}{\mathrm{Cat}_{\infty}}
\newcommand{\res}{\mathrm{res}}
\newcommand{\coind}{\mathrm{coind}}

\newcommand{\DD}{\mathbb{D}}

\newcommand{\Frob}{\mathrm{F}}
\newcommand{\Tr}{\mathrm{Tr}}

\title[Categorical representations]{Categorical representations of algebraic groups in positive characteristic}

\author{Pramod N. Achar}
\address{Department of Mathematics\\
  Louisiana State University\\
  Baton Rouge, LA 70803\\
  U.S.A.}
\email{pramod.achar@math.lsu.edu}

\author{Gurbir Dhillon}
\address{UCLA Mathematics Department, Los Angeles, CA 90095-1555, USA.}
\email{gsd@math.ucla.edu}

\author{Simon Riche}
\address{Universit\'e Clermont Auvergne, CNRS, LMBP, F-63000 Clermont-Ferrand, France.}
\email{simon.riche@uca.fr}

\thanks{P.A. was supported by NSF Grant No.~DMS-2202012.  
G.D. was supported by an NSF Postdoctoral Fellowship under grant No.~DMS-2103387.
This project has received
funding from the European Research Council (ERC) under the European Union's Horizon 2020
research and innovation programme (S.R., grant agreement No.~101002592).
}

\begin{document}

\begin{abstract}
In this paper we present a formalism of categorical representation theory for affine group schemes of finite type over fields (both in the weak and strong settings) which applies to arbitrary base fields. This is based on the development of an appropriate $\infty$-categorical setting for the study of representations and Harish-Chandra (bi)modules of such group schemes, which constitute the building blocks and basic examples of such structures. 
We also study categorical traces in this context, and show in particular that the categorical trace of the identity morphism, resp.~Frobenius morphism, on the appropriate $\infty$-category of representations of a connected reductive algebraic group $G$ over an algebraically closed field $k$ of positive characteristic identifies (under suitable assumptions) with the $\infty$-category of Ind-coherent sheaves on the adjoint quotient $G/G$, resp.~of the fixed points $G^\Frob$ of the Frobenius.
\end{abstract}

\maketitle

%%%%%%%%%%%%%%%%%%%%%%%%%%
\section{Introduction}
%%%%%%%%%%%%%%%%%%%%%%%%%%

%-----------------------------------------
\subsection{Categorical representation theory in characteristic \texorpdfstring{$0$}{0}}
%-----------------------------------------

\sss
Categorical representation theory of affine group schemes over fields of characteristic $0$ has proven to be an invaluable tool in the geometric Langlands program; see in particular~\cite{beraldo, tao} for the basics of this theory, and~\cite{dhillon} for a survey of some of its applications.

\sss
\label{sss:intro-def-weak-rep}
Let us briefly recall the basic definitions used to set up this theory, for later comparison with the case of general fields. Let $k$ be a field of characteristic $0$, and let $H$ be an affine group scheme of finite type over $k$. Then the $\infty$-category $\QCoh(H)$ of quasi-coherent sheaves on $H$ (i.e., $\OO(H)$-modules) admits a canonical structure of an algebra object in the monoidal $\infty$-category $\DGCat$ of presentable, stable, $k$-linear $\infty$-categories, with associated binary operation given by convolution. (Here, as usual, the monoidal structure on $\DGCat$ is given by the Lurie tensor product relative to the $\infty$-category $\Vect_k$ of $k$-vector spaces.) A \emph{weak categorical representation} of $H$ is defined to be a $\QCoh(H)$-module in $\DGCat$. Typical examples of such structures arise from geometry: if $X$ is a $k$-scheme (or, more generally, prestack over $k$) endowed with an action of $H$, then the $\infty$-category $\QCoh(X)$ is a weak categorical representation of $H$ in a natural way.

\sss
\label{sss:intro-def-strong-rep}
\emph{Strong categorical representations} are defined in a similar way, replacing $\QCoh(H)$ by the $\infty$-category $\Dmod(H)$ of $D$-modules on $H$ (i.e., $\sD(H)$-modules). We have a natural morphism of algebra objects $\QCoh(H) \to \Dmod(H)$ given by induction, which provides natural ``restriction'' and ``induction'' functors relating weak and strong categorical representations of $H$.

\sss
\label{sss:intro-weak-reps-Rep-modules}
It is a basic but very important observation that these structures can be described in ``dual'' terms as follows. Consider the $\infty$-category $\Rep(H)$ of algebraic $H$-modules. Then we have a canonical identification
\[
\Rep(H) = \Funct_{\QCoh(H)}(\Vect_k, \Vect_k),
\]
where $\Vect_k$ is equipped with the trivial action of $\QCoh(H)$, and we consider the $\infty$-category of $\QCoh(H)$-linear endofunctors of $\Vect_k$. Via this equivalence, composition of functors corresponds to the natural monoidal structure on $\Rep(H)$, induced by tensor product of representations. This identification allows us to define two functors from the $\infty$-category of weak categorical representations of $H$ to that of $\Rep(H)$-modules in $\DGCat$, those of ``invariants'' and ``coinvariants,'' given by
\[
(-)^{H, \mathrm{w}} : \sC \mapsto \Funct_{\QCoh(H)}(\Vect_k, \sC) \quad \text{and} \quad (-)_{H, \mathrm{w}} : \sC \mapsto \Vect_k \otimes_{\QCoh(H)} \sC
\]
respectively.

Using the ``$1$-affineness theorem'' for the stack $\BB H$ proved by Gaitsgory--Lurie, Beraldo proved in~\cite{beraldo} that these functors are naturally isomorphic, and induce an equivalence of $\infty$-categories between weak categorical representations of $H$ and $\Rep(H)$-modules. Note that equivalences between $\infty$-categories of functors and tensor products are typical for \emph{rigid} monoidal $\infty$-categories, but that here $\QCoh(H)$ is \emph{not} rigid. On the other hand $\Rep(H)$ \emph{is} rigid, which can sometimes make computations easier on the other side of the equivalence.

\sss
\label{sss:intro-strong-reps-Rep-modules}
The results from~\S\ref{sss:intro-weak-reps-Rep-modules} have counterparts for strong representations, which are in fact consequences of their weak versions. Namely, let $\fh$ be the Lie algebra of $H$, and consider the $\infty$-categories $\fh\mod$ of $\fh$-modules and $\hc_H$ of Harish-Chandra bimodules for $H$. Here $\hc_H$ is an algebra object in $\DGCat$, with associated binary operation given by (derived) tensor product of $U(\fh)$-bimodules. Then $\fh\mod$ has a natural structure of strong categorical representation of $H$ (obtained by interpreting $U(\fh)$-modules as weakly $H$-equivariant $D$-modules on $H$), and we have an identification
\[
\hc_H \simeq \Funct_{\Dmod(H)}(\fh\mod, \fh\mod).
\]
Using this identification one can ``upgrade'' the functors $(-)^{H, \mathrm{w}}$ and $(-)_{H, \mathrm{w}}$ to functors from strong categorical representations of $H$ to $\hc_H$-modules in $\DGCat$, see~\cite[\S 2.3.9]{beraldo}. As explained in the course of the proof of~\cite[Theorem~2.3.12]{beraldo} (see also~\cite[Proposition~2.5.7]{tao}), these functors are again isomorphic, and induce an equivalence between strong categorical representations of $H$ and $\hc_H$-modules.

%-----------------------------------------
\subsection{General base fields}
%-----------------------------------------

\sss
The goal of this paper is to develop the foundations of the corresponding theory for affine group schemes of finite type over an arbitrary field $k$.  In particular, for such an affine group scheme $H$, we study the monoidal $\infty$-categories $\QCoh(H)$, $\Rep(H)$, $\Dmod(H)$, and $\hc_H$, as well as modules over them.
This theory diverges at various points from the characteristic-$0$ story, both in terms of the definitions of the objects we are studying, and in terms of the statements we prove. (As usual, the main new difficulty we have to face is that the abelian category of $H$-representations typically has infinite cohomological dimension.)
We should emphasize that, in characteristic $0$, a similar renormalization was found to be necessary for weak actions of groups of infinite type, in the work of Raskin; see~\cite{raskin-ws,raskin}. In the present situation, it appears in the categorical representation theory, weak and strong, already for finite-dimensional groups.

\sss
To start, we must modify the definition of $\Rep(H)$, because the ``naive''  derived $\infty$-categories of $H$-representations (namely, the $\infty$-category $\QCoh(\BB H)$ of quasi-coherent sheaves on the classifying stack $\BB H$ or the derived $\infty$-category of the abelian category of algebraic representations of $H$) turn out to have poor categorical properties; in particular, they need not be compactly generated. To remedy this, we define $\Rep(H)$ in this paper as the category of \emph{Ind-coherent sheaves} on $\BB H$.
Equivalently, $\Rep(H)$ is obtained from $\QCoh(\BB H)$ by a ``renormalization'' procedure, which guarantees that finite-dimensional $H$-representations are compact in $\Rep(H)$.  Similar care is required in defining the appropriate $\infty$-category $\hc_H$ of Harish-Chandra bimodules: our definition ensures that finitely-generated Harish-Chandra bimodules are compact in $\hc_H$. 

\sss
In both cases, the $\infty$-category we study admits a canonical t-structure, whose bounded below part coincides with that of its more naive counterpart(s) (see~\S\S\ref{sss:QCoh-BH}--\ref{sss:def-RepH} and~\S\ref{sss:HC-ren}). These $\infty$-categories possess natural structures of algebra objects in $\DGCat$, and as such are rigid. The algebra $\Rep(H)$ is commutative, and we show that $\hc_H$ is pivotal if the ``modular character'' $\fh \to k$ given by $x \mapsto \mathrm{tr}(\mathrm{ad}(x))$ vanishes, e.g.~if $H$ is reductive. (In general, the double left dual on compact objects is given by conjugation with an invertible object closely related to this character, see~\S\ref{sss:twist-pivotality}.)

\sss
\label{sss:intro-def-cat-rep-general}
We then define a \emph{weak categorical representation} of $H$ to be a $\Rep(H)$-module in $\DGCat$, and a \emph{strong categorical representation} of $H$ to be an $\hc_H$-module in $\DGCat$. Comparing with~\S\S\ref{sss:intro-def-weak-rep}--\ref{sss:intro-def-strong-rep}, we see that these are not the obvious generalizations of the definitions in characteristic $0$, but compare more directly with the ``dual'' point of view of~\S\S\ref{sss:intro-weak-reps-Rep-modules}--\ref{sss:intro-strong-reps-Rep-modules}. We show that these definitions are also related to more direct adaptations of the characteristic-$0$ definitions, although they are not equivalent. Namely, consider the $\infty$-categories $\QCoh(H)$ of quasi-coherent sheaves on $H$ and $\Dmod(H)$ of (crystalline, i.e.~without divided powers) $D$-modules on $H$. These $\infty$-categories again have natural structures of algebra objects in $\DGCat$, and we construct (under a mild assumption) \emph{fully faithful} functors from (right) $\QCoh(H)$-modules to weak categorical representations of $H$ (see Proposition~\ref{prop:QCoh-Rep-ff}), and from (right) $\Dmod(H)$-modules to strong categorical representations of $H$ (see Proposition~\ref{prop:Dmod-HC-ff}), which are \emph{not} equivalences.

\sss
One can produce examples of categorical representations coming from geometry using the fully faithful functors from~\S\ref{sss:intro-def-cat-rep-general}.\footnote{This comment applies literally for weak representations, because the $\QCoh$ theory is well developed over general bases. For strong representations these constructions are presently limited by the absence of a documented general theory of $\infty$-categories of crystalline $D$-modules on schemes or stacks.}
But the examples we are interested in rather come from algebra. For weak representations we study $\infty$-categories of representations of subgroups (or, more generally, groups endowed with a morphism to $H$) in~\S\ref{ss:reps-morphisms} and equivariant modules for algebras in~\S\ref{ss:equiv-modules}; we describe in particular intertwining functors and tensor products of such categorical representations, and study dualizability. For strong representations we study $\infty$-categories of Harish-Chandra \emph{modules} in~\S\S\ref{ss:hK-mod}--\ref{ss:hKmod-functorialities}. These examples are particularly relevant for us, and were the main motivation for our investigations, because they appear in equivalences (which will be constructed in forthcoming works) relating such $\infty$-categories for $H$ a reductive group or a Borel subgroup with semiinfinite sheaves on the affine flag variety or affine Grassmannian of the Langlands dual group; see~\cite[\S 1.4]{adr} for some details.

\sss
We also take this opportunity to study categorical traces. Recall (see e.g.~\cite[\S 7.3.1]{zhu}) that for any algebra object $\sA$ in $\DGCat$ and any $\sA$-bimodule $\sM$ we have an associated Hochschild homology (or trace) $\Tr(\sA,\sM)$, which is an $\infty$-category. When $\sA=\Rep(H)$ and $\varphi : H \to H$ is a morphism, we can consider the bimodule $\sM = {}^{\varphi^*} \hspace{-2pt} \Rep(H)$ defined by $\Rep(H)$ with the left action twisted by pullback under $\varphi$. In this case we set $\Tr(\Rep(H),\varphi) := \Tr(\Rep(H), {}^{\varphi^*} \hspace{-2pt} \Rep(H))$.

We study in detail two examples of this construction.
The first one is when $\varphi=\id$; for this case we prove in Proposition~\ref{prop:trace-Rep-id} that if $k$ is algebraically closed and $H$ is a possibly disconnected reductive algebraic group, under mild assumptions the trace identifies with the $\infty$-category of Ind-coherent sheaves on the adjoint quotient $H/H$. The proof uses in particular a result of Fargues--Scholze~\cite{fs} which ensures that this $\infty$-category is generated by the ``free sheaves'' associated with finite-dimensional $H$-modules. The other case we study is when $k$ is an algebraic closure of a finite field, $H$ is smooth and defined over that finite field, and $\phi$ is pullback under the Frobenius endomorphism $\Frob$ of $H$. Assuming that $H$ is connected and that any simple $k$-representation of the finite group $H^\Frob$ extends to an algebraic representation of $H$ (which is satisfied in particular if $H$ is reductive with simply connected derived subgroup), we show in Proposition~\ref{prop:trace-Rep-Frob} that the trace identifies with the $\infty$-category $\Rep(H^\Frob)$.

We also study counterparts of these computations for strong categorical representations. When $\varphi$ is as above we also have a natural $\hc_H$-bimodule attached to $\varphi$, whose Hochschild homology is denoted $\Tr(\hc_H,\varphi)$. We construct a canonical conservative functor $\Tr(\hc_H,\varphi) \to \Tr(\Rep(H),\varphi)$, and prove that it is an equivalence in the setting above when $\varphi$ is the Frobenius morphism.

%---------------------------------------
\subsection{Contents}
%-----------------------------------------

\sss
Let us now describe in more detail the contents of the paper. Sections~\ref{sec:preliminaries} and~\ref{sec:infty-cat-gp-scheme} collect some preparatory material that will be required for the later sections.
Section~\ref{sec:preliminaries} is devoted to generalities on $\infty$-categories.
In Section~\ref{sec:infty-cat-gp-scheme} we state some basic properties of the $\infty$-category $\QCoh(H)$ of quasi-coherent sheaves on an affine group scheme $H$ of finite type over a field $k$. We then explain what we believe is the ``correct'' definition of the $\infty$-category $\Rep(H)$ of representations of $H$, and study the relation with other possible definitions of this object.

\sss
In Section~\ref{sec:weak-actions} we study weak categorical representations of $H$, i.e.~$\Rep(H)$-modules. In particular, in~\S\ref{ss:def-properties-weak} we explain how to compute tensor products and $\infty$-categories of intertwining functors between such representations, and in~\S\ref{ss:equiv-modules} we introduce the example of equivariant modules over an algebra. 
In~\S\ref{ss:reps-morphisms} we study in more detail the example of $\Rep(K)$ where $K$ is a subgroup of $H$, and show that under certain technical conditions (satisfied in particular when $H$ is a split reductive group scheme and $K$ is a parabolic subgroup) these modules are smooth and proper, which constitutes an analogue in this setting of the main result of~\cite{bzgo}; see Corollary~\ref{cor:smoothness-properness-Rep} and Remark~\ref{rmk:smoothness-properness-Rep}.
We also explain in~\S\ref{ss:equiv-deequiv} the relation between weak categorical representations and $\QCoh(H)$-modules, and compute in~\S\ref{ss:categorical-trace-weak}, under suitable assumptions, the categorical trace on $\Rep(H)$ of the identity of $H$ and the Frobenius morphism (when $k$ is an algebraic closure of a finite field and $H$ is defined over that finite field).

\sss
In Section~\ref{sec:hc} we define and study the $\infty$-category $\hc_H$ of Harish-Chandra bimodules for $H$. Here again there is a ``naive'' choice of definition, from which the ``clever'' definition is obtained using renormalization; see~\S\ref{sss:HC-ren}. We explain the structure of algebra object on $\hc_H$, show that it is rigid, and study its pivotality (see~\S\ref{sss:twist-pivotality}).

\sss
In Section~\ref{sec:strong-actions} we define strong categorical representations, and study in detail the example of the $\infty$-category $(\fh,K)\mod$ of Harish-Chandra modules with respect to a subgroup $K \subset H$ (see~\S\S\ref{ss:hK-mod}--\ref{ss:hKmod-functorialities}). In particular we show that, under suitable assumptions satisfied e.g.~when $H$ is a split reductive group scheme and $K$ is a parabolic subgroup, the $\hc_H$-module $(\fh,K)\mod$ is proper; see~\S\ref{sss:hK-mod-proper}. We also explain in~\S\ref{sss:tens-funct-strong} how to describe tensor products and $\infty$-categories of intertwining functors between strong categorical representations, and in~\S\ref{ss:equiv-deequiv-strong} we relate this notion to modules over the $\infty$-category $\Dmod(H)$ of $D$-modules on $H$. Finally, in~\S\ref{sss:traces-strong} we explain how to make sense of the trace on $\hc_H$ of an endomorphism of $H$, we relate it to the corresponding trace on $\Rep(H)$, and study the cases of the identity and Frobenius morphisms.

%---------------------------------------
\subsection{Acknowledgements}
%-----------------------------------------

We thank Julien Bichon and Quan Situ for useful discussions around the subject of this paper, and Michel Brion and Sasha Kuznetsov for helpful correspondence in relation with the content of Remark~\ref{rmk:smoothness-Rep}. We acknowledge interactions with Microsoft Copilot that suggested ideas used in some proofs, and for editing purposes.

%%%%%%%%%%%%%%%%%%%%%%%%%%
\section{Preliminaries}
\label{sec:preliminaries}
%%%%%%%%%%%%%%%%%%%%%%%%%%

%----------------------------------------------------------
\subsection{Algebra objects and \texorpdfstring{$\infty$}{infinity}-categories of modules}
\label{sss:algebras-modules}
%----------------------------------------------------------

\sss
\label{sss:inf-cat-notation}
We will work in the framework of Lurie's theory of $\infty$-categories. (We will always retain the prefix ``$\infty$-,'' and reserve the unadorned word ``category'' for \emph{ordinary} categories.) First, consider the $\infty$-category $\Catoo$ of $\infty$-categories, with the cartesian symmetric monoidal structure.
We also let $\PrL$ be the $\infty$-category of presentable $\infty$-categories and continuous (i.e., colimit-preserving) functors between them, equipped with the Lurie tensor product $\otimes$ (see~\cite[\S 4.8]{lurie-ha}). Recall that $\PrL$ has all limits and colimits, and that the tensor product is continuous in both variables, see~\cite[Remark~4.8.1.24]{lurie-ha}. (Here the inclusion $\PrL \subset \Catoo$ preserves limits, but not colimits in general.)

The canonical ``forgetful'' functor $\PrL \to \Catoo$ is lax symmetric monoidal; as a consequence it sends algebra objects in $\PrL$ to algebra objects in $\Catoo$ (which are the same as monoidal $\infty$-categories). As explained in~\cite[Introduction to~\S 4.8]{lurie-ha}, the functor sending an algebra object in $\PrL$ to the underlying monoidal $\infty$-category 
identifies such algebra objects with
monoidal $\infty$-categories $\sC$ which are presentable and such that the binary product $\sC \times \sC \to \sC$ preserves colimits in each variable.

The unit object in $\PrL$ is the $\infty$-category $\sS$ of spaces (also called the $\infty$-category of anima). Another important (symmetric) algebra object in this monoidal $\infty$-category is the $\infty$-category $\mathrm{Sp}$ of spectra. There exists a canonical functor $\sS \to \mathrm{Sp}$ which exhibits $\mathrm{Sp}$ as an idempotent object of $\PrL$ in the sense of~\cite[Definition~4.8.2.1]{lurie-ha}.

\sss
\label{sss:modules-oocat}
Recall that if $\sC$ is a monoidal $\infty$-category, $\sD$ is an $\infty$-category which is left-tensored over $\sC$, and $A$ is an algebra object in $\sC$,
we can consider the $\infty$-category
$\LMod_A(\sD)$ of left $A$-modules in $\sD$.
We have a natural ``restriction'' (or ``forgetful'') functor
\[
\LMod_A(\sD) \to \sD
\]
which is conservative
and admits a left adjoint, the ``induction functor,'' see~\cite[Corollary~4.2.4.8]{lurie-ha}. 

Assuming that $\sD$ admits colimits and that for any $X \in \sC$ the functor $X \otimes (-) : \sD \to \sD$ preserves colimits, $\LMod_A(\sD)$ admits colimits and the restriction functor is continuous, see~\cite[Corollary~4.2.3.5]{lurie-ha}.
Assuming moreover that $\sD$ is presentable,
$\LMod_A(\sD)$ is presentable,
 see~\cite[Corollary~4.2.3.7]{lurie-ha}.
If $\sD$ is stable and the functor $A \otimes (-) : \sD \to \sD$ is exact, the $\infty$-category $\LMod_A(\sD)$ is stable and the restriction functor is exact, see~\cite[Proposition~7.1.1.4]{lurie-ha}.

For details about this construction, see~\cite[\S 4.2]{lurie-ha} or~\cite[Chap.~1, \S 3.5]{gr}. Of course, as a particular case one can take $\sD=\sC$. 
The $\infty$-category $\LMod_A(\sC)$ is naturally right-tensored over $\sC$.

\sss
If $\sC$, $A$ are as in~\S\ref{sss:modules-oocat}, and if $\sD$ is an $\infty$-category which is right-tensored over $\sC$, then one can consider the $\infty$-category $\RMod_A(\sD)$ of right $A$-modules in $\sD$. This construction has properties similar to those recalled in~\S\ref{sss:modules-oocat} for left modules. 

\sss
\label{sss:reversed-algebra}
If $\sC$ is a \emph{symmetric} monoidal $\infty$-category, given an algebra object $A$ one can consider the ``opposite'' algebra object, see~\cite[Remark~4.1.1.7]{lurie-ha} or~\cite[Chap.~1, \S 3.1.4]{gr}. This opposite algebra is denoted $A^{\rev}$ to avoid any confusion with opposite categories (since $A$ itself will sometimes be an $\infty$-category). In this case an $\infty$-category $\sD$ which is left-tensored over $\sC$ is also naturally right-tensored over $\sC$, and we have an identification $\LMod_A(\sD) \simeq \RMod_{A^\rev}(\sD)$, see~\cite[\S 4.6.3]{lurie-ha}.

In this setting, in case $A$ is a commutative algebra object, the $\infty$-categories $\LMod_A(\sC)$ and $\RMod_A(\sC)$ identify canonically, see~\cite[\S 4.5.1]{lurie-ha}; they will be denoted $\Mod_A(\sC)$.

\sss
\label{sss:relative-tensor-product}
Let $\sC$, $\sD$ be as in~\S\ref{sss:modules-oocat}. Assume 
that $\sC$ and $\sD$ admit geometric realizations of simplicial objects, and that the binary operations $\sC \times \sC \to \sC$ and $\sC \times \sD \to \sD$ preserve geometric realizations of simplicial objects on each side.
If $A$ is an algebra object in $\sC$, there exists a natural ``relative tensor product'' functor
\[
\RMod_{A}(\sC) \times \LMod_{A}(\sD) \to \sD,
\]
denoted $- \otimes_A -$,
see~\cite[Example~4.4.2.12]{lurie-ha}. If $\sC$ and $\sD$ admit colimits and the binary operations $\sC \times \sC \to \sC$ and $\sC \times \sD \to \sD$ are continuous in both variables, then this functor is continuous in both variables, see~\cite[Corollary~4.4.2.16]{lurie-ha}.

\sss
\label{sss:restriction-extension-scalars}
Let $\sC$, $\sD$ be as in~\S\ref{sss:modules-oocat}.
Given algebra objects $A$ and $B$ in $\sC$ and a morphism of algebra objects $A \to B$, we have an associated ``restriction of scalars'' functor
\[
\LMod_B(\sD) \to \LMod_A(\sD).
\]
If $\sD$ is presentable and the functor $X \otimes (-) : \sD \to \sD$ is continuous for any $X \in \sC$, then this functor is continuous,
see~\cite[Corollary~4.2.3.7]{lurie-ha}. 
If 
$\sC$ and $\sD$ admit geometric realizations of simplicial objects, and if the binary operations $\sC \times \sC \to \sC$ and $\sC \times \sD \to \sD$ preserve geometric realizations of simplicial objects on each side,
then this functor admits a left adjoint, given by the relative tensor product $B \otimes_A (-)$ (suitably defined for bimodules), see~\cite[Proposition~4.6.2.17]{lurie-ha}. One can think of this adjoint as an ``extension of scalars'' functor.

\sss
\label{sss:Amod-monoidal}
Consider the setting of~\S\ref{sss:relative-tensor-product} with $\sC=\sD$, assuming now that $A$ is a
commutative algebra object in $\sC$. In this case the relative tensor product is part of a canonical symmetric monoidal structure on the $\infty$-category $\Mod_A(\sC)$, see~\cite[Theorem~4.5.2.1]{lurie-ha}.

\sss
\label{sss:claim-relative-tensor-product}
Below we will use the following standard fact, 
see also~\cite[Example~7.66]{zhu}.
Let $\sC$ be a symmetric monoidal $\infty$-category which admits colimits and such that the binary operation $\sC \times \sC \to \sC$ is continuous in both variables. 
Let $B$ be a commutative algebra object in $\sC$, and consider an algebra object $A$ in the monoidal $\infty$-category $\Mod_B(\sC)$. 
Via the lax monoidal forgetful functor $\Mod_B(\sC) \to \sC$, $A$ also defines an algebra object in $\sC$, endowed with an algebra map $B \to A$. In this setting, the Barr--Beck--Lurie theorem~\cite[Chap.~1, Proposition~3.7.7]{gr} implies that we have an equivalence $\LMod_A(\sC) \simeq \LMod_A(\Mod_B(\sC))$. Similarly we have $\RMod_A(\sC) \simeq \RMod_A(\Mod_B(\sC))$, and the two possible meanings of relative tensor product over $A$ identify.

Given $X$ in $\RMod_{A}(\sC)$ and $Y$ in $\LMod_A(\sC)$, the object $X \otimes_B Y$ is naturally a module over $A^\rev \otimes_B A$, and we have a canonical identification
\begin{equation}
\label{eqn:relative-tensor-product}
X \otimes_A Y \simeq A \otimes_{A^\rev \otimes_B A} (X \otimes_B Y),
\end{equation}
where $A$ is seen as a right module over $A^\rev \otimes_B A$ in the natural way. In fact, 
consider the bar resolution which computes the $(A^\rev \otimes_B A)$-module $A \otimes_A A \simeq A$; its $n$-th term is $A^{\otimes_B (n+2)}$. By continuity of the relative tensor product (see~\S\ref{sss:relative-tensor-product}), $A \otimes_{A^\rev \otimes_B A} (X \otimes_B Y)$ is the colimit of the diagram obtained by tensoring (over $A^\rev \otimes_B A$) this simplicial object with $X \otimes_B Y$. The $n$-th term of this diagram identifies with $X \otimes_B A^{\otimes_B n} \otimes_B Y$, and in fact one obtains in this way the bar complex which computes $X \otimes_A Y$, which finishes the proof.

%----------------------------------------------------------
\subsection{Algebra in stable presentable \texorpdfstring{$\infty$}{infinity}-categories of modules}
\label{sss:conventions-cat}
%----------------------------------------------------------

\sss
\label{sss:PrLSt}
Let $\PrLSt \subset \PrL$ be the full subcategory of \emph{stable} presentable $\infty$-categories. This $\infty$-category has all limits and colimits, and the inclusion $\PrLSt \subset \PrL$ preserves limits and colimits, see~\cite[\S 7.1.1]{zhu}.

The subcategory $\PrLSt \subset \PrL$ has the following interpretation: associated with the idempotent object $\mathrm{Sp}$ in $\PrL$ (see~\S\ref{sss:inf-cat-notation}) we have the full subcategory consisting of objects $\sC$ such that the functor $\sC \simeq \mathcal{S} \otimes \sC \to \mathrm{Sp} \otimes \sC$ is an equivalence. This subcategory is closed under tensor product, and it inherits a monoidal structure from that of $\PrL$, with unit object $\mathrm{Sp}$; moreover this monoidal $\infty$-category identifies with the monoidal $\infty$-category $\Mod_{\mathrm{Sp}}(\PrL)$, see~\cite[Proposition~4.8.2.10]{lurie-ha}. As explained in~\cite[Proposition~4.8.2.18]{lurie-ha}, this full subcategory is $\PrLSt$. 

The discussion following~\cite[Proposition~4.8.2.7]{lurie-ha} shows that the canonical functor from algebra objects in $\PrLSt$ to algebra objects in $\PrL$ is fully faithful, and that its essential image consists of those algebra objects in $\PrL$ which are moreover stable.
For the same reason,
given an algebra object $\sC$ in $\PrLSt$, if we denote similarly the induced algebra object in $\PrL$, the canonical functor
\[
\LMod_{\sC}(\PrLSt) \to \LMod_{\sC}(\PrL)
\]
is an equivalence of $\infty$-categories. 

\sss
Given $\sC$ in $\PrLSt$, the full subcategory of compact objects in $\sC$ (see~\cite[Chap.~1, \S 7.1]{gr}) will be denoted $\sC^\comp$. This is a stable (but nonpresentable in general) $\infty$-category.

\sss
Now we fix a base field $k$, and consider the $\infty$-category $\Vect_k$ of complexes of $k$-vector spaces. This $\infty$-category is presentable and stable, and it admits a canonical structure of commutative algebra object in $\PrLSt$. The $\infty$-categories we will encounter below will mostly be objects of the $\infty$-category
\[
\DGCat := \Mod_{\Vect_k}(\PrLSt).
\]
Recall that this $\infty$-category admits all limits and colimits, is presentable, and that the ``forgetful'' functor $\DGCat \to \PrLSt$ preserves limits and colimits, see e.g.~\cite[\S 7.1.5]{zhu}.
By~\S\ref{sss:Amod-monoidal} the $\infty$-category $\DGCat$ is naturally symmetric monoidal under the Lurie tensor product relative to $\Vect_k$; we will denote the associated binary product of two objects $\sC$ and $\sD$ by $\sC \otimes_k \sD$.

\sss
\label{sss:def-Amod-DGCat}
Given an algebra object $\sA$ in $\DGCat$, one can consider the $\infty$-categories
\[
\sA\mod := \LMod_{\sA}(\DGCat), \quad \modr\sA := \RMod_{\sA}(\DGCat).
\]
These $\infty$-categories are naturally modules over $\DGCat$.
We have canonical functors
\[
\sA\mod \to \LMod_{\sA}(\PrLSt), \quad \modr\sA \to \RMod_{\sA}(\PrLSt), 
\]
which as in~\S\ref{sss:claim-relative-tensor-product} are equivalences.

\sss
\label{sss:Funct}
Given objects $\sC$ and $\sD$ in $\sA\mod$ we will write $\Funct_{\sA}(\sC,\sD)$ for the $\infty$-category of $\sA$-linear functors from $\sC$ to $\sD$. When $\sA=\Vect_k$, we will simplify the notation to $\Funct_k(\sC,\sD)$.
The construction of $\Funct_{\sA}(\sC,\sD)$ is a particular case of the construction of the ``relative inner Hom,'' see~\cite[Chap.~1, \S 8.2.1]{gr}. As explained in~\cite[Chap.~1, \S 8.2.3]{gr}, in case $\sC=\sD$ the object $\Funct_{\sA}(\sC,\sC)$ admits a canonical structure of algebra object in $\DGCat$.
Concretely, using the bar resolution of $\sC$, the $\infty$-category $\Funct_{\sA}(\sC,\sD)$ can be computed as the limit of the natural cosimplicial diagram whose $n$-th term is $\Funct_{k}(\sA^{\otimes_k n} \otimes_k \sC,\sD)$. 

\sss
\label{sss:Funct-2}
In the setting of~\S\ref{sss:Funct},
note that $\Funct_k(\sC,\sD)$ is naturally a module over $\sA \otimes_k \sA^{\rev}$, and that we have an identification
\[
\Funct_{\sA}(\sC,\sD) \simeq \Funct_{\sA \otimes_k \sA^{\rev}}(\sA, \Funct_k(\sC,\sD)).
\]
Indeed, to compute the right-hand side we can replace $\sA$ by its bar resolution as in~\S\ref{sss:claim-relative-tensor-product}; we find that the $\infty$-category $\Funct_{\sA \otimes_k \sA^{\rev}}(\sA, \Funct_k(\sC,\sD))$ is the limit of a cosimplicial diagram with $n$-th term $\Funct_k(\sA^{\otimes_k n}, \Funct_k(\sC,\sD)) \simeq \Funct_k(\sA^{\otimes_k n} \otimes_k \sC, \sD)$, which identifies with the cosimplicial diagram considered above and which computes the left-hand side.

\sss
If $\sA$ is a commutative algebra object in $\DGCat$ then $\sA\mod$ identifies with $\modr\sA$ and admits a canonical symmetric monoidal structure given by relative tensor product over $\sA$, see~\S\ref{sss:Amod-monoidal}.

\sss
\label{sss:Amod-comp-gen}
Let $\sA$ be an algebra object in $\DGCat$, $\sC$ be a left $\sA$-module, and let $A$ be an algebra object in $\sA$. Then it makes sense to consider the $\infty$-category $\LMod_A(\sC)$ of left $A$-modules in $\sC$, and we have an adjunction
\[
A \otimes (-) : \sC \rightleftarrows \LMod_A(\sC) : \res
\]
where $\on{res}$ is conservative, see~\S\ref{sss:modules-oocat}.
Assume that $\sC$ is a compactly generated $k$-linear $\infty$-category, i.e.~$\sC=\Ind(\sC_0)$ for $\sC_0$ a small $k$-linear stable $\infty$-category. If we denote by $\sP_0$ the full idempotent complete stable subcategory of $\LMod_A(\sC)$ generated by the image of $\sC_0$ under the functor above, then we have
\[
\LMod_A(\sC) \simeq \Ind(\sP_0).
\]
In fact, by standard results (see e.g.~\cite[Chap.~1, Lemma~7.2.4]{gr}) it suffices to prove that $\LMod_A(\sC)$ is compactly generated by $\sP_0$, which follows from~\cite[Chap.~1, Lemmas~5.4.3 and~7.1.5]{gr}.

%-----------------------------------------------------
\subsection{\texorpdfstring{$\infty$}{Infinity}-categories of complexes}
%-----------------------------------------------------

\sss
\label{sss:derived-cat}
Recall that if $\mathsf{A}$ is an additive category we have the $\infty$-category $\Ch(\mathsf{A})$ of chain complexes of objects in $\mathsf{A}$, obtained using the dg-nerve construction from the differential graded category of complexes, see~\cite[\S 1.3.2]{lurie-ha}.\footnote{Our $\infty$-category $\Ch(\mathsf{A})$ corresponds to $\mathrm{N}_{\mathrm{dg}}(\Ch(\mathsf{A}))$ in the notation of~\cite{lurie-ha}.} The homotopy category $\mathrm{Ho}(\Ch(\mathsf{A}))$ is the usual (unbounded) homotopy category of $\mathsf{A}$ in the sense of triangulated categories. This $\infty$-category is stable, see~\cite[Proposition~1.3.2.10]{lurie-ha}. We will also consider the full subcategory $\Chb(\mathsf{A})$ of $\Ch(\mathsf{A})$ spanned by bounded complexes. By the comments preceding~\cite[Corollary~1.3.2.18]{lurie-ha}, $\Chb(\mathsf{A})$ is a stable subcategory of $\Ch(\mathsf{A})$; in particular it is also stable.

If $\mathsf{A}$ is abelian, we will also consider the $\infty$-category $\Db(\mathsf{A})$ obtained from $\Chb(\mathsf{A})$ by localization at the class of quasi-isomorphisms. Here again $\Db(\mathsf{A})$ is stable, see e.g.~\cite[Remark~2.8.2]{jasso}, and its homotopy category is the bounded derived category of $\mathsf{A}$ in the sense of triangulated categories. If $\mathsf{A}$ is a Grothendieck abelian category, we can also consider the unbounded derived category $D(\mathsf{A})$ following~\cite[\S 1.3.5]{lurie-ha}. Here again, $D(\mathsf{A})$ is the localization of $\Ch(\mathsf{A})$ at the class of quasi-isomorphisms (see~\cite[Proposition~1.3.5.13]{lurie-ha}), it is stable by~\cite[Proposition~1.3.5.9]{lurie-ha}, and presentable by~\cite[Proposition~1.3.5.21]{lurie-ha}, and finally its homotopy category is the unbounded derived category of $\mathsf{A}$ in the sense of triangulated categories; see also the discussion in~\cite[\S 2.6]{jasso}. In this setting we have a canonical fully faithful functor $\Db(\mathsf{A}) \to D(\mathsf{A})$.

\sss
\label{sss:Verdier-quotient}
Given a small stable $\infty$-category $\sC$ and a stable full subcategory $\sD$, recall that the Verdier quotient $\sC/\sD$ is the cofiber of the inclusion $\sD \to \sC$ (in small stable $\infty$-categories), see~\cite[\S 5.1]{bgt}. This quotient can also be described as the localization of $\sC$ at the class of morphisms whose cofiber belongs to $\sD$, see~\cite[Definition-Proposition~2.1.24]{jasso} or~\cite[Theorem~1.3]{drew}. By~\cite[Proposition~5.14]{bgt} or~\cite[Remark~2.1.25]{jasso}, the homotopy category of $\sC/\sD$ is the Verdier quotient $\mathrm{Ho}(\sC)/\mathrm{Ho}(\sD)$ in the sense of triangulated categories.

Using the description of Verdier quotient as localization one sees (using~\cite[Proposition~4.1.7.4]{lurie-ha}) that in case $\sC$ is furthermore a monoidal $\infty$-category (with binary product exact on both sides) and $\sD$ is a $2$-sided ideal (in the sense that it is stable under left and right multiplication with any object of $\sC$), the quotient $\sC/\sD$ admits a canonical monoidal structure (again, with binary product exact on both sides) such that the natural functor $\sC \to \sC/\sD$ is monoidal, see~\cite[Corollary~1.4]{drew} or~\cite[Theorem~I.3.6]{ns}.

In the setting of~\S\ref{sss:derived-cat}, if $\mathsf{A}$ is abelian $\Db(\mathsf{A})$ is the Verdier quotient of $\Chb(\mathsf{A})$ by the full subcategory of acyclic complexes, and if $\mathsf{A}$ is a Grothendieck abelian category then $D(\mathsf{A})$ is the Verdier quotient of $\Ch(\mathsf{A})$ by the full subcategory of acyclic complexes (because a morphism is a quasi-isomorphism if and only if its cofiber is acyclic).

\sss
\label{sss:algebra-structure-Chb}
Below we will use the following construction of algebra objects in $\DGCat$. Let $\mathsf{A}$ be a monoidal $k$-linear additive category. (This includes the condition that the monoidal product be bilinear.) Consider the $\infty$-category $\Chb(\mathsf{A})$.
The monoidal structure on $\mathsf{A}$ defines a (strict) monoidal structure on the dg-category of bounded chain complexes of objects in $\mathsf{A}$, and hence a monoidal structure on the $\infty$-category $\Chb(\mathsf{A})$, see the discussions in~\cite[\S A.1.11]{chen-dhillon} or~\cite[\S 2]{drew}. This monoidal structure is basically given by functors
\[
\Chb(\mathsf{A})^{\times I} \to \Chb(\mathsf{A})^{\times J}
\]
for any totally ordered finite sets $I$, $J$ and nondecreasing function $I \to J$, see e.g.~\cite[Remark~4.1.1.13]{lurie-ha}. Here $\Chb(\mathsf{A})$ admits a natural action of $k$ in the sense of~\cite[Definition~D.1.1.1]{lurie-sag}, so that $\Chb(\mathsf{A})^{\times I}$ admits an action of $k^{\times I}$ for any totally ordered finite set $I$, and the functors above have the property that they are $k^{\times I}$-linear where the action on $\Chb(\mathsf{A})^{\times J}$ is obtained by pullback along the natural map $k^{\times I} \to k^{\times J}$.

By~\cite[Corollary~4.8.1.14]{lurie-ha} (where the word ``symmetric'' can be omitted), from the monoidal structure on $\Chb(\mathsf{A})$ we obtain a canonical monoidal structure on the presentable stable $\infty$-category $\Ind(\Chb(\mathsf{A}))$ of Ind-objects in $\Chb(\mathsf{A})$. (See~\cite[Proposition~1.1.3.6]{lurie-ha} for stability.) The binary operation for this monoidal structure is continuous on both sides, so $\Ind(\Chb(\mathsf{A}))$ defines an algebra object in $\PrLSt$, see~\S\S\ref{sss:inf-cat-notation}--\S\ref{sss:PrLSt}.

We claim that $\Ind(\Chb(\mathsf{A}))$ in fact automatically upgrades to an algebra object in $\DGCat$. Indeed, from the action of $k$ on $\Chb(\mathsf{A})$ we obtain an action on $\Ind(\Chb(\mathsf{A}))$, see~\cite[Remark~D.1.1.4]{lurie-sag}. In view of~\cite[\S D.1.5]{lurie-sag}, this means that $\Ind(\Chb(\mathsf{A}))$ is naturally an object of $\DGCat$. (For this, see also the discussion in~\cite[Chap.~1, \S 10.5]{gr}.) As above the structure of algebra object in $\PrLSt$ is basically given by functors
\[
\Ind(\Chb(\mathsf{A}))^{\otimes I} \to \Ind(\Chb(\mathsf{A}))^{\otimes J}
\]
for any totally ordered finite sets $I$, $J$ and nondecreasing function $I \to J$. Using (a variant with more factors of) the identification~\eqref{eqn:relative-tensor-product}
and the discussion of extension of scalars in~\S\ref{sss:restriction-extension-scalars}, we see that each such functor defines canonically a functor
\[
\Ind(\Chb(\mathsf{A}))^{\otimes_k I} \to \Ind(\Chb(\mathsf{A}))^{\otimes_k J},
\]
and these functors define the desired structure of algebra object in $\DGCat$.

By similar considerations, given a $k$-linear additive category $\mathsf{M}$ equipped with a (bilinear) action of $\mathsf{A}$, we obtain on $\Ind(\Chb(\mathsf{M}))$ a structure of $\Ind(\Chb(\mathsf{A}))$-module object in $\DGCat$.

\sss
\label{sss:algebra-structure-Verdier-quotient}
The construction of~\S\ref{sss:algebra-structure-Chb} can also be combined with Verdier quotients (see~\S\ref{sss:Verdier-quotient}).
Consider as above a monoidal $k$-linear additive category $\mathsf{A}$ and an ideal $\mathsf{J}$ in the monoidal category $\mathrm{Ho}(\Chb(\mathsf{A}))$. Let $\sJ$ be the full subcategory of $\Chb(\mathsf{A})$ spanned by the objects which belong to $\mathsf{J}$, and consider the Verdier quotient $\Chb(\mathsf{A})/\sJ$. By the comments above this small stable $\infty$-category admits a canonical monoidal structure. By considerations similar to those explained in~\S\ref{sss:algebra-structure-Chb}, the monoidal $\infty$-category $\Ind(\Chb(\mathsf{A})/\sJ)$ admits a canonical structure of algebra object in $\DGCat$.

\begin{Rem}
\begin{enumerate}
\item
The discussion around~\cite[Proposition~5.13]{bgt} shows that $\Ind(\Chb(\mathsf{A})/\sJ)$ is also the cofiber (in $\DGCat$) of the embedding $\Ind(\sJ) \to \Ind(\Chb(\mathsf{A}))$. By~\cite[Proposition~5.9]{bgt}, its homotopy category is also given by the Verdier quotient $\mathrm{Ho}(\Ind(\Chb(\mathsf{A}))) / \mathrm{Ho}(\Ind(\sJ))$ in the sense of triangulated categories.
\item
Since a small stable $\infty$-category and its idempotent-completion have the same $\infty$-category of Ind-objects, one can replace in this discussion $\Chb(\mathsf{A})/\sJ$ by its idempotent-completion. Moreover, this idempotent-completion identifies with the full subcategory $\Ind(\Chb(\mathsf{A})/\sJ)^\comp$ of compact objects in $\Ind(\Chb(\mathsf{A})/\sJ)$.
\end{enumerate}
\end{Rem}

As above again, given a $k$-linear additive category $\mathsf{M}$ equipped with a (bilinear) action of $\mathsf{A}$ and a full subcategory $\mathsf{K}$ of $\mathrm{Ho}(\Chb(\mathsf{M}))$ stable under the action of $\mathrm{Ho}(\Chb(\mathsf{A}))$ and such that the induced action of $\mathrm{Ho}(\Chb(\mathsf{A}))$ on $\mathrm{Ho}(\Chb(\mathsf{M}))/\mathsf{K}$ factors through an action of $\mathrm{Ho}(\Chb(\mathsf{A}))/\mathsf{J}$, one can consider the full subcategory $\sK$ of $\Chb(\mathsf{M})$ spanned by objects which belong to $\mathsf{K}$, and the $\infty$-category $\Ind(\Chb(\mathsf{M})/\sK)$ admits a canonical structure of module for $\Ind(\Chb(\mathsf{A})/\sJ)$ in $\DGCat$.

%----------------------------------------------
\subsection{t-structures}
%----------------------------------------------

\sss
Recall from~\cite[Definition~1.2.1.4]{lurie-ha}\footnote{Contrary to Lurie, we use cohomological conventions.} that a t-structure on a stable $\infty$-category $\sC$ is by definition a t-structure $(\mathrm{Ho}(\sC)^{\leq 0}, \mathrm{Ho}(\sC)^{\geq 0})$ on the associated homotopy category $\mathrm{Ho}(\sC)$ (which has a canonical structure of triangulated category). Given such a t-structure, for any $n \in \Z$ we denote by $\sC^{\leq n}$, resp.~$\sC^{\geq n}$, the full subcategory of $\sC$ spanned by the objects which belong to $\mathrm{Ho}(\sC)^{\leq n}$, resp.~$\mathrm{Ho}(\sC)^{\geq n}$.

\sss
Let $\sC$ be a commutative algebra object in $\PrLSt$, and $\sD$ be an object in $\Mod_{\sC}(\PrLSt)$. Let $A$ be an algebra object in $\sC$, and consider the $\infty$-category $\LMod_A(\sD)$. By the comments in~\S\ref{sss:modules-oocat}, $\LMod_A(\sD)$ is stable and presentable, and the ``forgetful'' functor $\LMod_A(\sD) \to \sD$ is continuous and admits a left adjoint (which is automatically also continuous). At the level of objects, this functor sends $M$ to $A \otimes M$, with the action of $A$ on the left factor.

The following statement is well known. (See e.g.~\cite[Proposition~7.1.1.13]{lurie-ha} for a similar statement.)

\begin{Lem}
\label{lem:t-structures}
Assume that $\sD$ is endowed with an accessible t-structure.

\begin{enumerate}
\item
\label{it:t-struct-1}
If the functor $A \otimes (-) : \sD \to \sD$ is right t-exact, then there exists a unique t-structure on $\LMod_A(\sD)$ such that the forgetful functor $\LMod_A(\sD) \to \sD$ is t-exact. Moreover this t-structure is accessible, and we have $\LMod_A(\sD)^{\leq 0} \simeq \LMod_A(\sD^{\leq 0})$, and hence $\LMod_A(\sD)^{-} \simeq \LMod_A(\sD^{-})$.
\item
\label{it:t-struct-2}
If the functor $A \otimes (-) : \sD \to \sD$ is t-exact, then the induction functor $\sD \to \LMod_A(\sD)$ is also t-exact, and we also have $\LMod_A(\sD)^{\geq 0} \simeq \LMod_A(\sD^{\geq 0})$, and hence $\LMod_A(\sD)^{+} \simeq \LMod_A(\sD^{+})$.
\end{enumerate}
\end{Lem}

\begin{proof}
\eqref{it:t-struct-1}
Uniqueness of the t-structure is easy since, by conservativity of the forgetful functor, $\LMod_A(\sD)^{\leq 0}$, resp.~$\LMod_A(\sD)^{\geq 0}$, has to be the full subcategory spanned by objects whose image in $\sD$ belongs to $\sD^{\leq 0}$, resp.~$\sD^{\geq 0}$. For existence, we consider the full subcategory $\LMod_A(\sD)^{\leq 0}$ of $\LMod_A(\sD)$ spanned by objects whose image in $\sD$ belongs to $\sD^{\leq 0}$. The fully faithful functor $\sD^{\leq 0} \to \sD$ induces a fully faithful functor $\LMod_A(\sD^{\leq 0}) \to \LMod_A(\sD)$, whose essential image is clearly the same full subcategory. Since $\LMod_A(\sD) \to \sD$ is continuous and $\sD^{\le 0} \subset \sD$ is closed under colimits and extensions, it follows that $\LMod_A(\sD^{\le 0})$ is closed under colimits and extensions. By~\cite[Proposition~1.4.4.11]{lurie-ha}, there exists a unique t-structure on $\LMod_A(\sD)$ whose ``$\leq 0$'' part is $\LMod_A(\sD)^{\leq 0}$, and this t-structure is accessible. The forgetful functor is right t-exact with respect to that t-structure by construction. 
Its left adjoint is also right t-exact by our assumption, so this functor is also left t-exact.

\eqref{it:t-struct-2}
Our assumption clearly implies that the induction functor is t-exact, and the other claims are clear.
\end{proof}

%----------------------------------------------
\subsection{Renormalization}
\label{ss:renormalization}
%----------------------------------------------

\sss
Suppose $\sC$ is a stable $\infty$-category equipped with a t-structure.  Under some assumptions on this t-structure,~\cite[Proposition~6.3.2]{bznp} describes a functorial ``renormalization'' of $\sC$. (For a slightly different treatment of this theory, see~\cite[\S 2.4]{campbell-raskin}.) In this section, we explain how to use results of Lurie in~\cite[Appendix C]{lurie-sag} to describe this renormalization explicitly in the case where $\sC$ is the derived $\infty$-category of a locally noetherian Grothendieck category. We note that the main results of this subsection are essentially equivalent to those of~\cite[\S 2]{krause} (which is written in the language of triangulated categories).

\sss
Recall that if $\sC$ is a presentable stable $\infty$-category equipped with a t-structure compatible with filtered colimits, an object $X \in \sC$ is called \emph{almost compact} if $\tau_{\ge n}(X)$ is a compact object of $\sC^{\ge n}$ for all $n \in \Z$, i.e.~if for all $n \in \Z$ the functor
\[
\mathrm{Map}_{\sC}(X,-) : \sC^{\geq n} \to \sS
\]
commutes with filtered colimits.
An object is called \emph{coherent} if it is almost compact and bounded below. 

We will denote by $\sC^\coh \subset \sC$ the full subcategory spanned by coherent objects. Assuming $\sC$ is right complete, almost compact objects are automatically bounded above, see~\cite[Lemma~6.2.1]{bznp} or~\cite[Comments above Lemma~2.4.1]{campbell-raskin}, so coherent objects are bounded.

\begin{Rem}
\label{rmk:almost-stable-lurie}
If $\sC$ is as above, then $\sC^{\leq 0}$ is a prestable $\infty$-category in the sense of~\cite[Definition~C.1.2.1]{lurie-sag}. Lurie defines in~\cite[Definition~C.6.4.1]{lurie-sag} the condition of being almost compact in this setting. Note that if $X \in \sC^{\leq 0}$, for any $Y \in \sC$ we have
\[
\mathrm{Map}_{\sC}(X,Y) \simeq \mathrm{Map}_{\sC^{\leq 0}}(X,\tau_{\leq 0}(Y)).
\]
Hence $X$ is almost compact in the sense above if and only if it is almost compact in the sense of~\cite[Definition~C.6.4.1]{lurie-sag}.
\end{Rem}

\sss
From now on, let $\sA$ be a Grothendieck abelian category; in particular, $\sA$ has enough injectives (see~\cite[Corollary~1.3.5.7]{lurie-ha}).  Let $\Inj \sA$ denote the additive category of injective objects in $\sA$, and consider the $\infty$-category $\Ch(\Inj \sA)$ of chain complexes of objects in $\Inj \sA$.  This $\infty$-category is also called the \emph{unseparated derived $\infty$-category} of $\sA$, see~\cite[Definition~C.5.8.2]{lurie-sag}. This is a presentable stable $\infty$-category, and it is equipped with a canonical t-structure which is right complete and compatible with filtered colimits, and whose heart identifies with $\sA$; see~\cite[Theorem~C.5.8.8]{lurie-sag}.

The usual derived $\infty$-category $D(\sA)$ can be described as a certain full stable subcategory of the $\infty$-category of chain complexes of injective objects of $\sA$, see~\cite[Remark~C.5.8.3]{lurie-sag}; we therefore have a canonical fully faithful exact functor $D(\sA) \to \Ch(\Inj \sA)$. Now, recall the canonical t-structure on $D(\sA)$, see~\cite[Proposition~1.3.5.21]{lurie-ha}. By definition of the t-structure on $\Ch(\Inj \sA)$, see~\cite[Notation~C.5.8.4]{lurie-sag}, the functor above restricts to an equivalence
\begin{equation}
\label{eqn:chinj-d-positive}
D(\sA)^{\ge 0} \simto \Ch(\Inj \sA)^{\ge 0}.
\end{equation}
On the other hand, since $D(\sA)$ is a localization of $\Ch(\sA)$ we have a canonical functor
$\Ch(\Inj \sA) \to D(\sA)$.
The composition $D(\sA) \to \Ch(\Inj \sA) \to D(\sA)$ of these functors is isomorphic to the identity functor. In particular, this functor also restricts to an equivalence
\begin{equation}
\label{eqn:chinj-d-positive-2}
 \Ch(\Inj \sA)^{\ge 0} \simto D(\sA)^{\ge 0}.
\end{equation}

\sss
\label{sss:DA-coh}
Recall that the natural t-structure on $D(\sA)$ is right complete and compatible with filtered colimits, see~\cite[Proposition~1.3.5.21]{lurie-ha}. One can therefore consider the full subcategory $D(\sA)^\coh$ of coherent objects,
and $D(\sA)^\coh$ consists of bounded objects.
In view of~\eqref{eqn:chinj-d-positive}, we may 
identify $D(\sA)^\coh$ with a full subcategory of $\Ch(\Inj \sA)$.

\sss
Recall (see~\cite[Definition~C.6.8.5]{lurie-sag}) that a Grothendieck abelian category is called \emph{locally noetherian} if any object is the colimit of its noetherian subobjects, or equivalently is a filtered colimit of noetherian objects. We assume that this condition is satisfied, and let
\[
\sA^\noeth \subset \sA
\]
be the full subcategory spanned by the noetherian objects. By~\cite[Proposition~C.6.8.2]{lurie-sag}, $\sA^\noeth$ is a Serre subcategory of $\sA$, and by~\cite[Corollary~C.6.8.9]{lurie-sag}
it consists of the compact objects of $\sA$.

\begin{Rem}
It follows from~\cite[Proposition~3.6]{krause} that, assuming $\sA$ is locally noetherian, the canonical functor $\Ch(\Inj \sA) \to D(\sA)$ induces an equivalence from the Verdier quotient of $\Ch(\Inj \sA)$ by the full subcategory of acyclic complexes to $D(\sA)$. (This paper is written in the language of triangulated categories, but this property can be checked at the level of homotopy categories, so that we indeed obtain the claimed statement.) By definition of the t-structure, this full subcategory identifies with $\bigcap_n \Ch(\Inj \sA)^{\leq n}$. 
\end{Rem}

\sss
The following statement is a special case of~\cite[Lemma~6.2.5]{bznp}. We explain the proof, since it is omitted in~\cite{bznp} and can be obtained by a combination of results from~\cite{lurie-sag}.

\begin{Lem}
\label{lem:locnoeth-coh}
Assume that $\sA$ is locally noetherian, and let $X$ be an object in $D(\sA)$.  The following conditions are equivalent:
\begin{enumerate}
\item
\label{it:lnc-coh}
$X$ is coherent;
\item 
\label{it:lnc-ab}
there are only finitely many $i$ such that $H^i(X) \ne 0$, and these $H^i(X)$ belong to $\sA^\noeth$.
\end{enumerate}
\end{Lem}

\begin{proof}
As explained above the coherent objects in $D(\sA)$ are bounded. Since both conditions considered in this lemma are stable under cohomological shifts, we can therefore assume that $X \in D(\sA)^{\leq 0}$. Recall from Remark~\ref{rmk:almost-stable-lurie} that for such an object, almost compactness in our sense is equivalent to almost compactness in the sense considered in~\cite[\S C.6]{lurie-sag}.

By~\cite[Corollary~C.6.8.8]{lurie-sag}, we may identify $\sA \simeq \Ind(\sA^\noeth)$.  Then, by~\cite[Corollary~C.6.5.9]{lurie-sag}, $D(\sA)^{\le 0}$ is a ``coherent Grothendieck prestable $\infty$-category,'' and hence also a ``locally noetherian prestable $\infty$-category'' in the sense of~\cite[Definition~C.6.9.1]{lurie-sag}. Therefore,~\cite[Proposition~C.6.9.3]{lurie-sag} is applicable to $D(\sA)^{\le 0}$: it says that 
an object of $D(\sA)^{\le 0}$ is almost compact if and only if it has all its cohomology objects in $\sA^\noeth$.  
The claim follows.
\end{proof}

\begin{Rem}
\label{rmk:DAcoh-DbAnoeth}
Since $\sA^\noeth$ is a full abelian subcategory of $\sA$,
there exists a canonical functor $\Db(\sA^\noeth) \to D(\sA)$. By standard arguments (see e.g.~\cite[\href{https://stacks.math.columbia.edu/tag/0FCL}{Tag 0FCL}]{stacks-project}), this functor is fully faithful, and its essential image is the subcategory considered in Lemma~\ref{lem:locnoeth-coh}.
\end{Rem}

\sss
As explained in~\S\ref{sss:DA-coh}, we can consider $D(\sA)^\coh$ as a full subcategory of $\Ch(\Inj \sA)$. The following lemma says that this subcategory consists of compact objects.

\begin{Lem}
\label{lem:coh-compact}
Assume that $\sA$ is locally noetherian.  Every object in $D(\sA)^\coh$ is compact as an object in $\Ch(\Inj \sA)$.
\end{Lem}

\begin{proof}
For any objects $X$ and $Y$ in $\Ch(\sA)$ we can consider the ``internal Hom'' $\underline{\Hom}_{\sA}(X,Y)$, a chain complex of abelian groups describing morphisms from $X$ to $Y$ in the differential graded category of chain complexes of objects in $\sA$, and for any $n \geq 0$ we have
\[
\pi_n\Map_{\Ch(\sA)}(X,Y) \simeq H^{-n} \underline{\Hom}_{\sA}(X,Y),
\]
see~\cite[Remark~1.3.2.3]{lurie-ha}. In particular, when $X$ and $Y$ are complexes of injective objects, this describes homotopy groups of mapping spaces in the $\infty$-category $\Ch(\Inj \sA)$. Let us note the following: if $X_1$, $X_2$ are bounded below complexes of objects in $\sA$ and $f : X_1 \to X_2$ is a quasi-isomorphism of complexes, for any object $Y$ in $\Ch(\Inj \sA)$ the morphism
\[
\underline{\Hom}_{\sA}(X_2,Y) \to \underline{\Hom}_{\sA}(X_1,Y)
\]
induced by $f$ is a quasi-isomorphism. Indeed, if $X_1$ and $X_2$ are zero in degrees $<N$, then for any $n \in \Z$ we have
\[
\underline{\Hom}_{\sA}(X_2,Y)^n = \underline{\Hom}_{\sA}(X_2,Y_{\geq n+N})^n, \quad \underline{\Hom}_{\sA}(X_1,Y)^n = \underline{\Hom}_{\sA}(X_1,Y_{\geq n+N})^n
\]
where $Y_{\geq n+N}$ is obtained from $Y$ by replacing components of degree $<n+N$ by $0$, so we can assume that $Y$ is bounded below, in which case the claim is well known.

Now we come to the proof of the lemma.
In view of Lemma~\ref{lem:locnoeth-coh},
it is enough to show that any noetherian object $C$ in $\sA$ is compact as an object of $\Ch(\Inj \sA)$, i.e.~that the functor $\pi_0 \Map_{\Ch(\Inj\sA)}(C,-)$ commutes with direct sums (see~\cite[Proposition~1.4.4.1]{lurie-ha}).
If $C \to X^\bullet$ is an injective resolution, then $C$ seen as an object in $\Ch(\Inj \sA)$ is $X^\bullet$. Let $Y^\bullet$ be another object in $\Ch(\Inj \sA)$. Then by the comments above, $\underline{\Hom}_\sA(X^\bullet, Y^\bullet) \to \underline{\Hom}_{\sA}(C,Y^\bullet)$ is a quasi-isomorphism, so for any $n \geq 0$ we have
 \begin{multline*}
 \pi_n\Map_{\Ch(\Inj\sA)}(C,Y^\bullet) = H^{-n} \underline{\Hom}_{\sA}(X^\bullet, Y^\bullet) \\
 \simeq H^{-n} \underline{\Hom}_{\sA}(C, Y^\bullet) = \frac{\ker \bigl( \Hom_{\sA}(C,Y^{-n}) \xrightarrow{d_Y^{-n} \circ (-)} \Hom_{\sA}(C,Y^{-n+1}) \bigr)}
{\im \bigl( \Hom_{\sA}(C,Y^{-n-1}) \xrightarrow{d^{-n-1}_Y \circ (-)} \Hom_{\sA}(C,Y^{-n}) \bigr)}.
 \end{multline*}
Since $C$ is compact in the abelian category $\sA$, $\Hom_{\sA}(C,{-})$ commutes with direct sums.  And since $\sA$ is a Grothendieck category, direct sums are exact: they commute with operations like kernel, image, and cokernel.  The lemma follows.
\end{proof}

\sss
The following lemma says that the full subcategory $D(\sA)^\coh \subset \Ch(\Inj \sA)$ generates this $\infty$-category.

\begin{Lem}
\label{lem:coh-generate}
Assume that $\sA$ is locally noetherian.  For every nonzero object $X$ in $\Ch(\Inj \sA)$, there exists a nonzero morphism in $\mathrm{Ho}(\Ch(\Inj \sA))$ from an object of $D(\sA)^\coh$.
\end{Lem}

\begin{proof}
As in the proof of Lemma~\ref{lem:coh-compact}, it suffices to construct a bounded complex $Y$ of noetherian objects of $\sA$ and a morphism of complexes $Y \to X$ which is not null-homotopic.
The proof is different depending on whether $X$ is acyclic or not.

Suppose first that $X = (X^\bullet, d^\bullet)$ is \emph{not} acyclic: say $H^n(X) \ne 0$.  Consider the (nonzero) surjective map $\ker (d_X^n) \twoheadrightarrow H^n(X)$.  Because $\ker (d_X^n)$ is the union of its noetherian subobjects, there exists a noetherian subobject $C \subset \ker (d_X^n)$ such that the composition
\[
C \to \ker (d_X^n) \twoheadrightarrow H^n(X)
\]
is nonzero.  The inclusion $C \hookrightarrow \ker (d_X^n)$ defines a chain map $C[-n] \to X$, which is not null-homotopic because it induces a nonzero map in cohomology.

Now suppose that $X$ is acyclic.  If $\ker (d_X^n)$ were injective for all $n$, it would follow that $X$ is homotopy-equivalent to $0$; so our assumption implies that there is some $n$ such that $\ker (d_X^n)$ is not injective. Since injectivity of an object can be detected on noetherian objects (see~\cite[Proposition~C.5.6.12]{lurie-sag}), there is a monomorphism of noetherian objects $i: C \to D$ such that the induced map
\begin{equation}
\label{eqn:inj-detect}
\Hom_{\sA}(D,\ker (d_X^n)) \to \Hom_{\sA}(C, \ker (d_X^n))
\end{equation}
is not surjective.

Let $f: C \to \ker (d_X^n)$ be a map that is not in the image of~\eqref{eqn:inj-detect}.  Then the solid arrows in the diagram
\begin{equation}
\label{eqn:coh-generate}
\begin{tikzcd}
\cdots \ar[r] & 0 \ar[r] \ar[d] & C \ar[r, "i"] \ar[d, "f"] \ar[dl, dashed, "h"'] & D \ar[d, "0"] \ar[r] \ar[dl, dashed, "k"'] & 0 \ar[r] \ar[d] & \cdots \\
\cdots \ar[r] & X^{n-1} \ar[r, "d_X^{n-1}"] & X^n \ar[r, "d_X^n"] & X^{n+1}  \ar[r, "d_X^{n+1}"] & X^{n+2} \ar[r] & \cdots 
\end{tikzcd}
\end{equation}
depict a chain map from an object of $D(\sA)^\coh$ to $X$.  If it were null-homotopic, there would be maps $h$ and $k$ as shown above, with
\[
f = d_X^{n-1} \circ h + k \circ i,\qquad d_X^n \circ k = 0.
\]
The latter equation says that $k$ can be treated as a map $k: D \to \ker (d_X^n)$.  And of course $d_X^{n-1}$ can also be treated as a map $d_X^{n-1}: X^{n-1} \to \ker (d_X^n)$.
Since $X^{n-1}$ is injective, there is a map $\tilde h: D \to X^{n-1}$ satisfying $\tilde h \circ i = h$.  Then
\[
f = d_X^{n-1}\circ \tilde h \circ i + k \circ i = (d_X^{n-1} \circ \tilde h + k) \circ i.
\]
Regarding $d_X^{n-1} \circ \tilde h + k$ as an element of $\Hom_{\sA}(D, \ker (d_X^n))$, this equation contradicts the assumption that $f$ is not in the image of~\eqref{eqn:inj-detect}.  We conclude that the chain map~\eqref{eqn:coh-generate} is not null-homotopic.
\end{proof}

\sss
We conclude this subsection with the following proposition, which says that $\Ch(\Inj \sA)$ is the renormalization of $D(\sA)$ equipped with its canonical t-structure.

\begin{Prop}
\label{prop:renormalization-DA}
Let $\sA$ be a locally noetherian Grothendieck abelian category.  There is a canonical equivalence
\[
\Ind(D(\sA)^\coh) \simto \Ch(\Inj \sA).
\]
\end{Prop}

\begin{proof}
The full stable subcategory $D(\sA)^\coh \subset \Ch(\Inj \sA)$ consists of compact objects (see Lemma~\ref{lem:coh-compact}), and generates $\Ch(\Inj \sA)$ (see Lemma~\ref{lem:coh-generate}). The claim therefore follows from standard results on compact generation, see~\cite[Chap.~1, Lemma~7.2.4(3)]{gr}.
\end{proof}

\begin{Rem}
\phantomsection
\label{rmk:renormalization-DA}
\begin{enumerate}
\item
\label{it:renormalization-DA-IndDbA}
In view of Remark~\ref{rmk:DAcoh-DbAnoeth}, in the setting of Proposition~\ref{prop:renormalization-DA} we also have an equivalence of $\infty$-categories
\begin{equation}
\label{eqn:renormalization-DA-IndDbA}
\Ind(\Db(\sA^\noeth)) \simto \Ch(\Inj \sA).
\end{equation}
\item
\label{it:renormalization-bounded-below-parts}
A basic result in the theory of renormalization is that (under suitable assumptions) the canonical functor from the renormalized $\infty$-category to the original $\infty$-category restricts to a t-exact equivalence between the bounded below parts of the t-structures, see~\cite[Proposition~6.3.2]{bznp} or~\cite[Proposition~2.4.6]{campbell-raskin}. In the case considered above, this says that the canonical functor $\Ch(\Inj \sA) \to D(\sA)$ restricts to a t-exact equivalence $\Ch(\Inj \sA)^+ \to D(\sA)^+$, which directly follows from~\eqref{eqn:chinj-d-positive-2}.
\end{enumerate}
\end{Rem}

%----------------------------------------------------------
\subsection{Duality}
\label{ss:duality}
%----------------------------------------------------------

\sss
\label{sss:dual-bimodules}
Recall\footnote{There seem to be inconsistencies in the literature regarding the notation and terminology for left/right duals. Here we follow the conventions of~\cite{gr} for terminology, which is opposite to that used in~\cite[\S 4.6.1]{lurie-ha} or~\cite[Definition~7.15]{zhu}.} that if $\sC$ is a monoidal $\infty$-category and $A$ and $B$ are algebra objects in $\sC$, one can speak of $(A,B)$-bimodules in $\sC$. Assume that $\sC$ admits geometric realizations of simplicial objects, and that the binary product $\sC \times \sC \to \sC$ preserves geometric realizations of simplicial objects. (This assumption ensures that the relative tensor products considered below exist.) If $M$ is an $(A,B)$-bimodule in $\sC$, a right dual $M^\vee$, resp.~left dual ${}^\vee M$, of $M$, is a $(B,A)$-bimodule in $\sC$ endowed with morphisms
\[
e_M : M \otimes_B M^\vee \to A \quad \text{and} \quad u_M : B \to M^\vee \otimes_A M
\]
of $A$-bimodules and of $B$-bimodules respectively,
resp.
\[
e_M : {}^\vee M \otimes_A M \to B \quad \text{and} \quad u_M : A \to M \otimes_B {}^\vee M
\]
of $B$-bimodules and of $A$-bimodules respectively, which satisfy the appropriate zigzag relations at the level of homotopy categories. We will say that $M$ is left, resp.~right, dualizable, if it admits a left, resp.~right, dual. Note that duals are unique up to a contractible space of choices when they exist, so that it makes sense to use a special notation for them.

\sss
The morphisms $e_M$ and $u_M$ are called the evaluation and coevaluation morphisms respectively. If $M$ is right dualizable, the functor
\[
M^\vee \otimes_A (-) : \LMod_A(\sC) \to \LMod_B(\sC), \quad \text{resp.} \quad (-) \otimes_B M^\vee : \RMod_B(\sC) \to \RMod_A(\sC),
\]
is right, resp.~left, adjoint to the functor $M \otimes_B (-)$, resp.~$(-) \otimes_A M$. Similarly, if $M$ is left dualizable, the functor
${}^\vee M \otimes_A (-)$, resp.~$(-) \otimes_B {}^\vee M$, is left, resp.~right, adjoint to the functor $M \otimes_B (-)$, resp.~$(-) \otimes_A M$.

\sss
In the special case when $A=B$ is the unit object in $\sC$, one does not need to impose any assumption on $\sC$, and one obtains the notions of left and right dualizability in $\sC$. In case $\sC$ is a \emph{symmetric} monoidal $\infty$-category, these two notions coincide, and we can simply speak of dualizability. When $\sC$ is an ordinary (monoidal) category, one recovers the usual notions of dualizability discussed e.g.~in~\cite{egno}.

\sss
\label{sss:dualizability-modules-DGCat}
We will in particular consider the notions above with $\sC=\DGCat$, and either $A$ or $B$ is $\Vect_k$ (so that we consider left or right $A$-modules).
Note that the ``dualizability'' of left and right $A$-modules in the sense of~\cite[Chap.~1, \S 4.3.1]{gr} corresponds to \emph{right} dualizability of left $A$-modules and \emph{left} dualizability of right $A$-modules in the sense above.

\sss
\label{sss:smooth-proper-modules}
Consider the setting of~\S\ref{sss:dual-bimodules} with $\sC=\DGCat$. Let $\sA$, $\sB$ be algebra objects in $\DGCat$, and let $\sM$ be a right dualizable $(\sA,\sB)$-bimodule, with right dual $\sM^\vee$, unit of the duality $u_\sM : \sB \to \sM^\vee \otimes_\sA \sM$, and counit of the duality $e_\sM : \sM \otimes_\sB \sM^\vee \to \sA$. Recall (see e.g.~\cite[Definition~7.86]{zhu}) that $\sM$ is called \emph{smooth} if $u_\sM$ admits a right adjoint $u_\sM^R$ as a functor of $\sB$-bimodules, and \emph{proper} if $e_\sM$ admits a right adjoint $e_\sM^R$ as a functor of $\sA$-bimodules.

See~\cite[\S 7]{zhu} for some applications of these notions in the context of computation of categorical traces.
Up to some differences of conventions, smooth and proper bimodules are the same as the ``properly dualizable'' bimodules considered in~\cite{bzgo}; more specifically, if $\sM$ is as above, $\sM$ is smooth and proper if and only if the (left dualizable) $(\sB,\sA)$-bimodule $\sM^\vee$ is properly dualizable in the sense of~\cite[Definition~2.8]{bzgo}. In this context we therefore have the ``Morita equivalence'' given by~\cite[Theorem~1.4]{bzgo}.

In practice, we will consider this notion in the case where $\sA$ and $\sB$ are rigid $k$-linear monoidal $\infty$-categories (in the sense of~\cite[Example~7.88]{zhu}). In this case, linearity of the adjoints follows from continuity (see~\cite[Chap.~1, Lemma~3.5.3 and Lemma~9.3.6]{gr}).

%---------------------------------------------
\subsection{Rigidity}
\label{ss:rigidity}
%---------------------------------------------

\sss
\label{sss:rigidity-duals}
Recall the notion of rigid $k$-linear monoidal $\infty$-category from~\cite[Example~7.88]{zhu}. (In particular, a rigid $k$-linear monoidal $\infty$-category is a rigid monoidal $\infty$-category in the sense of~\cite[Chap.~1, \S 9]{gr}.) If $\sA$ is a rigid $k$-linear monoidal $\infty$-category, recall that a left $\sA$-module $\sC$ is right dualizable if and only if it is dualizable in $\DGCat$, or in $\PrLSt$. Moreover, in this case, the right dual of $\sC$ is the right $\sA$-module given by $\sC^\vee$ (the dual of $\sC$ in $\DGCat$, or in $\PrLSt$), endowed with the twist of the natural right action of $\sA$ by the autoequivalence $(\varphi_{\sA})^{-1}$ where $\varphi_\sA$ is as in~\cite[Chap.~1, \S 9.2.7]{gr}. (In case $\sA$ is compactly generated, in view of~\cite[Chap.~1, Lemma~9.2.8]{gr}, a compact object $a$ therefore acts on $\sC$ by the natural action of $a^{\vee\vee}$.) For details, see~\cite[Chap.~1, \S 9.5]{gr} or~\cite[Proposition~7.105]{zhu}.

\sss
\label{sss:rigidity}
Let $\sA$ be a rigid $k$-linear monoidal $\infty$-category. By definition, the binary product $m_\sA : \sA \otimes_k \sA \to \sA$ admits an $(\sA \otimes_k \sA^\rev)$-linear right adjoint $m_\sA^R : \sA \to \sA \otimes_k \sA$. Consider the monad
\[
m_\sA : \sA \otimes_k \sA \rightleftarrows \sA : m_\sA^R.
\]
The essential image of $m_\sA$ manifestly generates the target $\infty$-category, so that $(m_\sA)^R$ is conservative by~\cite[Chap.~1, Lemma~5.4.3]{gr}. We can therefore apply the Barr--Beck--Lurie theorem~\cite[Chap.~1, Proposition~3.7.7]{gr}. If we consider $\sA \otimes_k \sA$ and $\sA$ as right modules over $\sA^\rev \otimes_k \sA$, then both $m_\sA$ and $m_\sA^R$ are maps of right $(\sA^\rev \otimes_k \sA)$-modules, so that the corresponding monad is given by the algebra object $\mathfrak{a} := m_\sA^R( 1_\sA)$ in $\sA^\rev \otimes_k \sA$, and we obtain an equivalence
\begin{equation}
\label{eqn:rigidity-A-Lmod}
\sA \simeq \LMod_{\mathfrak{a}}(\sA^\rev \otimes_k \sA).
\end{equation}
Under this equivalence, $m_\sA^R$ corresponds to the obvious ``restriction'' functor, and $m_\sA$ corresponds to its left adjoint (the ``induction'' functor).

\sss
\label{sss:rigidity-tens-Hom}
Continue with the setting of~\S\ref{sss:rigidity}. Let also $\sM$, resp.~$\sN$, be a right, resp.~left, $\sA$-module, so that we can consider the tensor product $\sM \otimes_\sA \sN$. By~\eqref{eqn:relative-tensor-product} we have
\[
\sM \otimes_\sA \sN \simeq \sA \otimes_{\sA^\rev \otimes_k \sA} (\sM \otimes_k \sN).
\]
Combining this with~\eqref{eqn:rigidity-A-Lmod} and~\cite[Chap.~1, Corollary~8.5.7]{gr} we deduce an identification
\[
\sM \otimes_\sA \sN \simeq \LMod_{\mathfrak{a}}(\sM \otimes_k \sN).
\]
Under this identification, the insertion functor $\sM \otimes_k \sN \to \sM \otimes_\sA \sN$ corresponds to the induction functor $\sM \otimes_k \sN \to \LMod_{\mathfrak{a}}(\sM \otimes_k \sN)$; this functor admits a continuous right adjoint, corresponding to the restriction functor $\LMod_{\mathfrak{a}}(\sM \otimes_k \sN) \to \sM \otimes_k \sN$.

Similarly, if $\sM$, $\sN$ are left $\sA$-modules, then by the discussion in~\S\ref{sss:Funct-2} $\Funct_k(\sM, \sN)$ is naturally a right module for $\sA^\rev \otimes_k \sA$, and we have an identification
\[
\Funct_\sA(\sM, \sN) \simeq \RMod_{\mathfrak{a}}(\Funct_k(\sM, \sN)).
\]

%%%%%%%%%%%%%%%%%%%%%%%%%%%%%%%
\section{Some \texorpdfstring{$\infty$}{infinity}-categories of sheaves attached to an affine group scheme}
\label{sec:infty-cat-gp-scheme}
%%%%%%%%%%%%%%%%%%%%%%%%%%%%%%%

We continue with our fixed base field $k$.

%----------------------------------------------------------
\subsection{Sheaf theories}
%----------------------------------------------------------

\sss
The geometric objects we will work with are prestacks over $k$, as defined e.g.~in~\cite[\S 9.1]{zhu}. Among these objects we have (derived\footnote{Following standard practice, the word ``derived'' will be omitted in relation with stacks, algebraic spaces or schemes. So, what we call a stack (algebraic space, scheme) should formally be called a \emph{derived} stack (algebraic space, scheme). We will use the word ``classical'' when referring to objects of standard (nonderived) algebraic geometry.}) stacks, algebraic stacks, algebraic spaces and schemes, see again~\cite[\S 9.1]{zhu}. Below we will also consider the property for a prestack of being almost of finite presentation; for this notion, see~\cite[\S 9.1.3]{zhu}. We will write $\times_k$ for $\times_{\Spec(k)}$.

\sss
\label{sss:QCoh-def}
Recall that for any prestack $Y$ over $k$ we have an $\infty$-category $\QCoh(Y)$ of quasi-coherent sheaves on $Y$, which is naturally an object of $\DGCat$; see~\cite[Chap.~3, \S 1.1]{gr} or~\cite[\S 9.2]{zhu} for a review of the construction. This $\infty$-category admits a natural t-structure, which is left complete. In case $Y$ is an algebraic stack, this t-structure is also right complete and compatible with filtered colimits, see~\cite[Chap.~3, Corollary~1.5.7]{gr} or~\cite[Lemma~9.8]{zhu}. Moreover its heart $\QCoh(Y)^\heartsuit$ identifies with the usual abelian category $\QCoh^\heartsuit(Y_{\mathrm{cl}})$ of quasi-coherent sheaves on the classical stack $Y_{\mathrm{cl}}$ attached to $Y$, and is a Grothendieck abelian category, see~\cite[Comments preceding Lemma~9.8]{zhu}. By the universal property of the derived $\infty$-category (see~\cite[Theorem~2.6.8]{jasso}), in this case there exists a t-exact continuous functor
\[
D(\QCoh^\heartsuit(Y_{\mathrm{cl}})) \to \QCoh(Y)
\]
whose restriction to the hearts is the identity functor. If $Y$ is a quasi-compact and quasi-separated classical algebraic stack with affine diagonal, then this functor restricts to an equivalence of categories
\[
D(\QCoh^\heartsuit(Y))^+ \simto \QCoh(Y)^+,
\]
see~\cite[Lemma~9.8]{zhu} or~\cite[Chap.~3, Proposition~2.4.3]{gr}.  As a consequence, in this case $\QCoh(Y)$ is the left completion of $D(\QCoh^\heartsuit(Y))$, see the discussion in~\cite[Chap.~3, Remark~2.4.4]{gr}. 

\sss
\label{sss:QCoh-sheaf-theory}
The assignment $Y \mapsto \QCoh(Y)$ can be extended to a sheaf theory defined on an appropriate $\infty$-category of correspondences with values in $\DGCat$, see~\cite[Eqn.~(9.5)]{zhu}. This means in particular that for any morphism of prestacks $f : Y \to Z$ we have an associated pullback functor $f^* : \QCoh(Z) \to \QCoh(Y)$ in $\DGCat$ and that for a certain class of morphisms we also have a pushforward functor $f_* : \QCoh(Y) \to \QCoh(Z)$ in $\DGCat$, these functors being compatible in a very strong sense. 
(By the adjoint functor theorem, $f^*$ always admits a right adjoint $f_*$, which however might not be continuous, and thus might not be a morphism in $\DGCat$. When $f$ is allowed as a horizontal morphism in the $\infty$-category of correspondences as above, then this adjoint $f_*$ \emph{is} continuous, and coincides with the functor considered above.)
The class of morphisms for which pushforward is defined contains all morphisms representable by qcqs algebraic spaces (see the comments following~\cite[Eqn.~(9.5)]{zhu}), and also all morphisms that are representable by concentrated stacks (see the comments preceding~\cite[Lemma~9.11]{zhu}).

\sss
\label{sss:QCoh-algebra-object}
By a general principle explained e.g.~in~\cite[\S 8.2.1]{zhu}, this implies in particular that for any prestack $Y$ over $k$ the $\infty$-category $\QCoh(Y)$ has a 
canonical structure of commutative algebra object in $\DGCat$ induced by tensor product over the structure sheaf. For another discussion of this structure, see~\cite[Chap.~3, \S 3.2.1]{gr}.

\sss
\label{sss:def-IndCoh}
If $Y$ is an algebraic stack which is almost of finite presentation over $k$, one can consider the full subcategory $\Coh(Y) \subset \QCoh(Y)$ whose objects are those which are bounded with respect to the t-structure considered above, and such that all cohomology objects are coherent sheaves (on $Y_{\mathrm{cl}}$). Following~\cite[Definition~9.19]{zhu}, the $\infty$-category $\IndCoh(Y)$ of Ind-coherent sheaves on $Y$ is defined as the Ind-completion (in the sense discussed in~\cite[Chap.~1, \S 7.2]{gr}) of $\Coh(Y)$. Here, since $\Coh(Y)$ is idempotent complete, one can recover this $\infty$-category as the full subcategory $\IndCoh(Y)^\comp$ of compact objects in $\IndCoh(Y)$, see~\cite[Chap.~1, Lemma~7.2.4]{gr}.

The t-structure on $\QCoh(Y)$ restricts to a t-structure on $\Coh(Y)$, with heart the abelian category $\Coh^\heartsuit(Y_{\mathrm{cl}})$ of coherent sheaves on $Y_{\mathrm{cl}}$. Taking the Ind-completions of the connective and coconnective parts of this t-structure (see~\cite[Chap.~4, Lemma~1.2.4]{gr}) we obtain a t-structure on $\IndCoh(Y)$, which is compatible with filtered colimits. 

If $Y$ is moreover a classical algebraic stack with affine diagonal, recall that we have an identification $D(\QCoh^\heartsuit(Y))^+ \simeq \QCoh(Y)^+$, see~\S\ref{sss:QCoh-def}. (Note that the condition of being almost of finite presentation includes the property of being quasi-compact and quasi-separated.) In this case, $\Coh(Y)$ therefore identifies with a full subcategory of $D(\QCoh^\heartsuit(Y))$; as such, it is equivalent to $\Db \Coh^\heartsuit(Y)$, see Remark~\ref{rmk:DAcoh-DbAnoeth}. Proposition~\ref{prop:renormalization-DA} therefore implies that $\IndCoh(Y) \simeq \Ch(\Inj \QCoh^\heartsuit(Y))$.

\sss
\label{sss:Psi}
By the universal property of the Ind-completion (see e.g.~\cite[Chap.~1, Lemma~7.2.4]{gr}) there exists a canonical functor
\[
\Psi_Y : \IndCoh(Y) \to \QCoh(Y),
\]
which is t-exact.
In case $Y$ is a (quasi-compact and quasi-separated) smooth classical scheme over $k$, it is a standard fact that $\QCoh(Y)$ is compactly generated by $\Coh(Y)$, so that in this case the functor $\Psi_Y$ is an equivalence. This is not the case for more general algebraic stacks.

\begin{Rem}
The construction of $\IndCoh(Y)$ we consider here follows~\cite{zhu}. Even when $k$ has characteristic $0$, this construction is \emph{not} the same as that studied in~\cite{gr}, and the resulting $\infty$-categories can be different. See~\cite[Remark~9.41(2)]{zhu} for a discussion of this topic.
\end{Rem}

\sss
\label{sss:t-structure-IndCoh}
The following statement is~\cite[Lemma~9.21]{zhu}.

\begin{Lem}
\label{lem:QCoh-IndCoh-tstr}
For any $n \in \Z$, the functor $\Psi_Y$ restricts to an equivalence
$\IndCoh(Y)^{\geqslant n} \simto \QCoh(Y)^{\geqslant n}$.
As a consequence, this functor restricts to a t-exact equivalence
$\IndCoh(Y)^+ \simto \QCoh(Y)^+$.
\end{Lem}

\begin{Rem}
\begin{enumerate}
\item
As noted in~\cite[Remark~9.23]{zhu}, Lemma~\ref{lem:QCoh-IndCoh-tstr} implies that $\IndCoh(Y)$ is obtained from $\QCoh(Y)$ and its canonical t-structure via
the renormalization procedure of~\cite[\S 6]{bznp}.
\item
The full subcategory $\IndCoh(Y)^{\leqslant 0} \subset \IndCoh(Y)$ is spanned by the objects $\xi$ such that $\Psi_Y(\xi)$ belongs to $\QCoh(Y)^{\leqslant 0}$. In fact, by t-exactness, if $\xi$ is in $\IndCoh(Y)^{\leqslant 0}$ then $\Psi_Y(\xi)$ is in $\QCoh(Y)^{\leqslant 0}$. Conversely, if $\Psi_Y(\xi)$ is in $\QCoh(Y)^{\leqslant 0}$ we consider the truncation fiber sequence $\tau_{\leqslant 0}(\xi) \to \xi \to \tau_{> 0}(\xi)$. By t-exactness, our assumption implies that $\Psi_Y(\tau_{>0}(\xi))=0$. Since $\Psi_Y$ restricts to an equivalence on $\IndCoh(Y)^{\geqslant 0}$ it follows that $\tau_{>0}(\xi)=0$, so that $\xi$ is in $\IndCoh(Y)^{\leqslant 0}$.
\item
We have constructed $\IndCoh(Y)$ above starting from $\QCoh(Y)$ and its canonical t-structure. A posteriori, one can also reconstruct $\QCoh(Y)$ from $\IndCoh(Y)$ and its t-structure (when it is defined): in fact, since the t-structure on $\QCoh(Y)$ is left complete, in view of Lemma~\ref{lem:QCoh-IndCoh-tstr} this $\infty$-category is the left completion of $\IndCoh(Y)$.
\end{enumerate}
\end{Rem}

\sss
The assignment $Y \mapsto \IndCoh(Y)$ also extends to a sheaf theory with values in $\DGCat$ on an appropriate $\infty$-category of correspondences for algebraic stacks\footnote{In~\cite{zhu} the author even considers ind-algebraic stacks; but this extra generality will not be required here.} almost of finite presentation over $k$. In fact there are two such extensions; we first consider the $*$-version, see~\cite[Theorem~9.32]{zhu}. In this case we have a pushforward functor $f^{\IndCoh}_* : \IndCoh(Y) \to \IndCoh(Z)$ for any morphism $f : Y \to Z$, which is obtained by the universal property of the Ind-completion from the composition
\[
\Coh(Y) \subset \QCoh(Y)^+ \xrightarrow{f_*} \QCoh(Z)^+ \xleftarrow[\sim]{\Psi_Z} \IndCoh(Z)^+ \subset \IndCoh(Z).
\]
(Here $f_*$ is the not necessarily continuous right adjoint of $f^*$, which is defined for any $f$.) 
In particular, in case the functor $f_*$ is continuous, by the universal property of the Ind-completion
we have $\Psi_Z \circ f^{\IndCoh}_* \simeq f_* \circ \Psi_Y$.

For $f : Y \to Z$ as above, we have a pullback functor $f^*_{\IndCoh} : \IndCoh(Z) \to \IndCoh(Y)$ in case $f$ is of finite tor amplitude. (See the comments above~\cite[Definition~9.7]{zhu} for the definition of this condition.) In fact, in this case the functor $f^*$ sends $\QCoh(Z)^+$ into $\QCoh(Y)^+$, and $f^*_{\IndCoh}$ is obtained as above from this functor. In particular, here we always have $\Psi_Y \circ f^*_{\IndCoh} \simeq f^* \circ \Psi_Z$.

There is also a $!$-version of the sheaf theory $\IndCoh$, see~\cite[Theorem~9.39]{zhu}. In particular, for any morphism $f : Y \to Z$ we have a continuous functor $f_{\IndCoh}^! : \IndCoh(Z) \to \IndCoh(Y)$. In case $f$ is representable by algebraic spaces and proper, $f^!_{\IndCoh}$ is the right adjoint of the functor $f_*^{\IndCoh}$.

\sss
\label{sss:IndCoh-algebra-object}
Let $Y$ be an algebraic stack which is almost of finite presentation over $k$. Assume that the structure map $Y \to \Spec(k)$ and the diagonal map $Y \to Y \times_k Y$ have finite tor amplitude. Then, again by the principle explained in~\cite[Definition~8.16(2)]{zhu}, $\IndCoh(Y)$ acquires a canonical structure of commutative algebra object in $\DGCat$. In this case, the functor $\Psi_Y$ is a morphism of commutative algebra objects.

More explicitly, the assumptions on $Y$ imply that the symmetric monoidal structure on $\QCoh(Y)$ restricts to a symmetric monoidal structure on $\Coh(Y)$. Using~\cite[Corollary~4.8.1.14]{lurie-ha} we deduce a symmetric monoidal structure on $\IndCoh(Y)$, which is the structure induced by the algebra structure constructed above.

%----------------------------------------------------------
\subsection{Quasi-coherent sheaves on an affine group scheme}
\label{ss:QCoh-gp-scheme}
%----------------------------------------------------------

In this subsection we fix a classical affine group scheme $H$ over $k$. 

\sss
\label{sss:QCohH-algebra-object}
Let us consider the $\infty$-category $ \QCoh(H)$. This $\infty$-category admits a structure of commutative algebra object in $\DGCat$ by the general considerations in~\S\ref{sss:QCoh-algebra-object}, but this structure is not very interesting since it does not involve the group structure on $H$ in any way. On the other hand, $H$  defines a Segal object in the $\infty$-category of correspondences considered in~\S\ref{sss:QCoh-sheaf-theory}, see~\cite[Example~8.8]{zhu}, and hence an algebra object by~\cite[Corollary~8.11]{zhu}, so that $\QCoh(H)$ has an associated structure of algebra object in $\DGCat$. The associated binary product is given by convolution, i.e.~the operation $(M,N) \mapsto m_*(M \boxtimes N)$ for $m : H \times_k H \to H$ the multiplication map, where we use the identification $\QCoh(H \times_k H) = \QCoh(H) \otimes_k \QCoh(H)$.

\sss
\label{sss:QCohH-algebra-object-2}
This structure can equivalently be obtained using ``only'' the monoidal structure on $\QCoh$,
following the discussion in~\cite[\S 2.1]{beraldo}. Namely, we have the symmetric monoidal functor $A \mapsto \LMod_A(\Vect_k)$ on animated $k$-algebras. (Here we still denote by $A$ the algebra object in $\Vect_k$ associated with $A$. Note that $\LMod_A(\Vect_k)$ coincides with the $\infty$-category of $A$-modules in any reasonable sense one might consider.) Our affine group scheme $H$ corresponds to a Hopf algebra $\OO(H)$ over $k$, which defines a coalgebra object in animated $k$-algebras. Therefore $\QCoh(H) = \LMod_{\OO(H)}(\Vect_k)$ has a canonical structure of coalgebra object in $\DGCat$. (Here the comultiplication is induced by the functor $m^*$.) Now $\QCoh(H)$ is dualizable as a $\Vect_k$-module, and in fact self-dual (e.g.~by~\cite[Chap.~1, Corollary~8.6.3]{gr}), so that by duality we deduce an algebra structure on $\QCoh(H)$.

\sss
\label{sss:action-QCoh}
For similar reasons, if $X$ is a prestack
endowed with an action of $H$,\footnote{See~\cite[\S 4.3]{khan} for a discussion of this notion. A basic example is a classical $k$-scheme endowed with an action of $H$ in the classical sense.} then the $\infty$-category $\QCoh(X)$ admits a canonical structure of $\QCoh(H)$-module in $\DGCat$. On the other hand one can consider the quotient stack $X/H$, which by definition is the \'etale sheafification (see~\cite[Chap.~2, \S 2.3.4]{gr}) of the quotient \emph{pre}stack of $X$ by $H$, i.e.~the colimit of the simplicial diagram defining the action of $H$ on $X$.
(As explained in~\cite[Remark~4.3.4]{khan}, the $n$-th term of this simplicial object is $H^{\times_k n} \times_k X$.)

The following statement is standard, see e.g.~\cite[Lemma~2.1.4]{tao}. Here we consider $\Vect_k = \QCoh(\Spec(k))$ as a $\QCoh(H)$-module using the previous construction for the (trivial) action of $H$ on $\Spec(k)$. (Concretely, this action is given by tensoring with global sections.)

\begin{Lem}
\label{lem:fixed-pts-QCoh-quotient}
We have a canonical identification
\[
\Funct_{\QCoh(H)}(\Vect_k, \QCoh(X)) \simeq \QCoh(X/H).
\]
\end{Lem}

\begin{proof}
Recall that the construction of $\QCoh$ is insensitive to \'etale sheafification (see~\cite[Chap.~3, Corollary~1.3.8]{gr}) and that, being defined as a right Kan extension, this functor sends colimits of prestacks to limits (in $\DGCat$). Hence $\QCoh(X/H)$ identifies with the limit of the cosimplicial diagram obtained by applying $\QCoh$ to the simplicial diagram defining the action of $H$ on $X$. The $n$-th term of this diagram is $\QCoh(H^{\times_k n} \times_k X)$,
and using~\cite[Chap.~3, Proposition~3.1.7]{gr} one sees that for any $n \in \Z_{\geq 1}$ we have
\[
\QCoh(H^{\times_k n} \times_k X) \simeq \QCoh(H)^{\otimes_k n} \otimes_k \QCoh(X).
\]

On the other hand, following the discussion in~\S\ref{sss:Funct}, the object $\Funct_{\QCoh(H)}(\Vect_k, \QCoh(X))$ can be computed using the bar resolution as the limit of the natural cosimplicial diagram with $n$-th term $\Funct_{\Vect_k}(\QCoh(H)^{\otimes_k n}, \QCoh(X))$.
By (self-)duality, for any $n$ we have
\[
\Funct_{\Vect_k}(\QCoh(H)^{\otimes_k n}, \QCoh(X)) \simeq \QCoh(H)^{\otimes_k n} \otimes_k \QCoh(X),
\]
so that this diagram identifies with the one considered above, which finishes the proof.
\end{proof}

\begin{Rem}
\label{rmk:sheafification}
We have defined the quotient stack $X/H$ using \'etale sheafification, which is the ``minimal'' sheafification which ensures that this prestack is indeed a stack. However, it is sometimes interesting to also consider fpqc or fppf sheafifications of the quotient prestack, see~\S\ref{sss:BH} below. This does not affect Lemma~\ref{lem:fixed-pts-QCoh-quotient}, or any consideration regarding quasi-coherent sheaves, since in fact $\QCoh$ satisfies flat descent, see~\cite[Chap.~3, Corollary~1.3.7]{gr}.
\end{Rem}

%----------------------------------------------------------
\subsection{Representations of an affine group scheme}
\label{ss:representations-affine-gp-scheme}
%----------------------------------------------------------

We continue as in~\S\ref{ss:QCoh-gp-scheme} with a classical affine group scheme $H$ over $k$, which we now assume to be of finite type.

\sss
\label{sss:Rep-abelian}
We want to discuss $\infty$-categories of (algebraic) representations of $H$, following~\cite[Example~9.13]{zhu}. At the level of abelian categories there is only one reasonable choice of such a category, namely the category $\Rep^\heartsuit(H)$ of algebraic $H$-modules, i.e.~$\OO(H)$-comodules. This category is a locally noetherian Grothendieck abelian category (see~\cite[Corollary~2.2.8]{dnr}),
it admits a canonical monoidal structure given by tensor product over $k$, which is exact on both sides, and it contains the full subcategory $\Rep^{\heartsuit,\fd}(H)$ of finite-dimensional representations.  Note that $\Rep^{\heartsuit,\fd}(H)$ is also the category of noetherian (or equivalently, compact) objects in $\Rep^\heartsuit(H)$.

\sss
\label{sss:DRep}
For ``derived'' versions the situation is more complicated. The first $\infty$-category one can consider is the derived $\infty$-category $D(\Rep^\heartsuit(H))$, see~\S\ref{sss:derived-cat}.
This presentable stable $\infty$-category is equipped with a canonical t-structure, which is right complete and compatible with filtered colimits, see~\S\ref{sss:DA-coh}.
(It can fail to be left complete: see~\cite{neeman}.)  The bounded below part of this t-structure identifies with the bounded below derived $\infty$-category $D^+(\Rep^\heartsuit(H))$ studied in~\cite[\S 1.3.2]{lurie-ha}, see~\cite[Remark~1.3.5.10]{lurie-ha}.

\sss
\label{sss:BH}
Now let $\BB H$ be the stack of $H$-torsors in the fpqc topology (denoted $\mathbb{B}_{\mathrm{fpqc}} H$ in~\cite{zhu}), i.e.~the fpqc sheafification of the quotient prestack of $\Spec(k)$ by $H$. Since by assumption $H$ is of finite type over $k$, $\BB H$ is also the stack of $H$-torsors in the fppf topology (i.e., the fppf sheafification of the quotient prestack of $\Spec(k)$ by $H$), and is an algebraic stack (almost of finite presentation). If $H$ is smooth then $\BB H$ is also the stack of $H$-torsors in the \'etale topology (i.e., the quotient stack $\Spec(k)/H$ with our conventions, see~\S\ref{sss:action-QCoh}). A second ``natural'' $\infty$-category of representations of $H$ is $\QCoh(\BB H)$. By the general considerations in~\S\ref{sss:QCoh-algebra-object} this $\infty$-category admits a canonical structure of commutative algebra object in $\DGCat$, which informally is given by tensor product of representations.

\sss
\label{sss:QCoh-BH}
By the generalities in~\S\ref{sss:QCoh-def}, $\QCoh(\BB H)$ admits a canonical t-structure whose heart identifies with the abelian category $\Rep^\heartsuit(H)$. The algebra structure considered above is compatible with the t-structure in the obvious way, and the induced monoidal structure on $\Rep^\heartsuit(H)$ is the natural monoidal structure considered in~\S\ref{sss:Rep-abelian}.

By the same considerations the t-structure on $\QCoh(\BB H)$ is right complete (in addition to being left complete). 
There exists a t-exact functor $D(\Rep^\heartsuit(H)) \to \QCoh(\BB H)$, whose restriction to the hearts is the identity, and which restricts to an equivalence $D(\Rep^\heartsuit(H))^+ \simto \QCoh(\BB H)^+$.
In particular, $\QCoh(\BB H)$ is the left completion of $D(\Rep^\heartsuit(H))$. 

\sss
\label{sss:def-RepH}
Now, we consider
the $\infty$-category
\[
\Rep(H) := \IndCoh(\BB H).
\]
From our point of view this $\infty$-category is the most satisfactory version of a ``derived $\infty$-category of representations'' of $H$.
By the general considerations in~\S\ref{sss:IndCoh-algebra-object}, this $\infty$-category admits a canonical structure of commutative algebra object in $\DGCat$, and $\Psi_{\BB H} : \Rep(H) \to \QCoh(\BB H)$ is a morphism of algebra objects.

By the general considerations in~\S\ref{sss:t-structure-IndCoh}, we also have a canonical t-structure on $\Rep(H)$. By Lemma~\ref{lem:QCoh-IndCoh-tstr} the functor
$\Psi_{\BB H}$
is t-exact and restricts to an equivalence
\begin{equation}
\label{eqn:equiv-Rep+-QCoh+}
\Rep(H)^+ \simto \QCoh(\BB H)^+;
\end{equation}
in particular, the heart of the t-structure on $\Rep(H)$ identifies with $\Rep^\heartsuit(H)$. Here again the algebra structure considered above is compatible with the t-structure in the obvious way, and the induced monoidal structure on $\Rep^\heartsuit(H)$ is the natural one.

\sss 
\label{sss:IndCoh-finite-cohomological-dim}
If we assume in addition that $H$ has finite cohomological dimension (e.g.~that $\mathrm{char}(k)=0$, or that $H$ is linearly reductive, see~\cite{hr}), then the discussion in~\cite[Example~9.13]{zhu} (based on~\cite[Lemma~9.14]{zhu}) shows that the $\infty$-categories $D(\Rep^\heartsuit(H))$, $\QCoh(\BB H)$ and $\Rep(H)$ are canonically equivalent.
In general, however, the $\infty$-category $\QCoh(\BB H)$ is not compactly generated (see e.g.~\cite{hr} for closely related results in the realm of triangulated categories) and therefore $\Psi_{\BB H}$ is not an equivalence.

\sss
\label{sss:Repc-DbRepfg}
By definition the subcategory $\Rep(H)^\comp$ of compact objects in $\Rep(H)$ identifies with $\Coh(\BB H)$.
By the discussion in~\S\ref{sss:def-IndCoh}, we have a monoidal equivalence $\Rep(H)^\comp \simeq \Db (\Rep^{\heartsuit,\fd}(H))$, and then an equivalence of monoidal $\infty$-categories
\[
\Rep(H) = \Ind(\Db (\Rep^{\heartsuit,\fd}(H))),
\]
where the monoidal structure on the right-hand side is constructed as in~\S\ref{sss:algebra-structure-Verdier-quotient}.
We also have an equivalence
$\Rep(H) \simeq \Ch(\Inj (\Rep^\heartsuit(H)))$.

\sss
\label{sss:Repc-tilting}
In case 
$H$ is a split reductive group scheme over $k$, one can consider the full subcategory $\mathrm{Tilt}(H) \subset \Rep^{\heartsuit,\fd}(H)$ of (finite-dimensional) tilting $H$-modules, see~\cite[Chap.~II.E]{jantzen}. In this case it is a standard fact that the canonical functor
\[
\Chb (\mathrm{Tilt}(H)) \to \Db (\Rep^{\heartsuit,\fd}(H))
\]
is an equivalence.
In this case we therefore have an equivalence of monoidal $\infty$-categories
\[
\Rep(H) = \Ind(\Chb (\mathrm{Tilt}(H))).
\]

%----------------------------------------------------------
\subsection{Properties of the \texorpdfstring{$\infty$}{infinity}-category \texorpdfstring{$\Rep(H)$}{Rep(H)}}
%----------------------------------------------------------

We continue as in~\S\ref{ss:representations-affine-gp-scheme} with
a classical affine group scheme $H$ of finite type over $k$.

\sss
\label{sss:rigidity-Rep}
In the symmetric monoidal category $\Rep^{\heartsuit,\mathrm{fd}}(H)$, each object $V$ is dualizable, with dual given by the contragredient representation $V^*$, and evaluation/coevaluation morphisms induced by the natural morphisms $V \otimes_k V^* = V^* \otimes_k V \to k$ and $k \to V^* \otimes_k V = V \otimes_k V^*$. As a consequence, each object in $\Rep(H)^\comp$ is dualizable.
By~\cite[Chap.~1, Lemma~9.1.5]{gr}, $\Rep(H)$ is rigid; in fact $\Rep(H)$ is a rigid $k$-linear monoidal $\infty$-category in the sense of~\cite[Example~7.88]{zhu}.
By contrast, note that $\QCoh(\BB H)$ is typically not rigid, e.g., the trivial representation $k$ can be noncompact if $\mathrm{char}(k)>0$.

\sss 
We now study the compatibility of the construction above with respect to products of groups. 

\begin{Lem}
\label{lem:Rep-product}
Given affine group schemes $H_1$, $H_2$ of finite type over $k$, exterior product induces a symmetric monoidal fully faithful functor
\[
\Rep(H_1) \otimes_k \Rep(H_2) \to \Rep(H_1 \times_k H_2).
\]
Assuming that one of the following conditions is satisfied:
\begin{enumerate}
\item
\label{it:Rep-product-assumption-1}
$H_1$ and $H_2$ have finite cohomological dimension;
\item
\label{it:Rep-product-assumption-2}
 for any simple $H_1$-module $V$ we have $\End_{H_1}(V)=k$,
\item
\label{it:Rep-product-assumption-3}
 for any simple $H_2$-module $V$ we have $\End_{H_2}(V)=k$,
\end{enumerate} 
 this functor is an equivalence of $\infty$-categories. 
\end{Lem}

\begin{proof}
The existence and full faithfulness of the functor is given by~\cite[Proposition~9.31(1)]{zhu}. This functor is clearly monoidal. 

If the assumption in~\eqref{it:Rep-product-assumption-1} is satisfied, then $H_1 \times_k H_2$ also has finite cohomological dimension, so that as explained in~\S\ref{sss:IndCoh-finite-cohomological-dim} one can replace each ``$\IndCoh$'' by ``$\QCoh$,'' and the functor is an equivalence by~\cite[Chap.~3, Proposition~3.4.2]{gr} or~\cite[Lemma~9.11]{zhu}.
If the assumption in~\eqref{it:Rep-product-assumption-2} or~\eqref{it:Rep-product-assumption-3} is satisfied, by~\cite[Proposition~4.21]{milne} any simple representation of $H_1 \times_k H_2$ is a tensor product of (simple) representations of $H_1$ and $H_2$, and hence belongs to the essential image of our functor, which is therefore an equivalence.
\end{proof}

\begin{Rem}
\label{rmk:assumptions-Rep-product}
The assumption in Lemma~\ref{lem:Rep-product}\eqref{it:Rep-product-assumption-1} is satisfied if $\mathrm{char}(k)=0$ or if $H_1$, $H_2$ are linearly reductive, see~\S\ref{sss:IndCoh-finite-cohomological-dim}. The assumption in Lemma~\ref{lem:Rep-product}\eqref{it:Rep-product-assumption-2} is satisfied in case $k$ is algebraically closed by Schur's lemma. It is satisfied also if $H_1$ is a diagonalizable group scheme (associated with a finitely generated abelian group) or is split reductive (see~\cite[Proposition~II.2.8]{jantzen}). We do not know an example of failure of essential surjectivity of the functor considered in this lemma.
\end{Rem}

\sss
\label{sss:Res-Coind}
We conclude this subsection with a discussion of the functoriality of the construction of the $\infty$-category $\Rep$ with respect to change of groups. Consider affine group schemes $H$, $K$ of finite type over $k$, and a morphism of $k$-group schemes $H \to K$. Then we have an induced morphism of stacks $\BB H \to \BB K$, which is of finite tor amplitude. The associated $*$-pullback functor
\[
\res^K_H : \Rep(K) \to \Rep(H)
\]
is t-exact, and restricts on the hearts of the canonical t-structures to restriction of representations. The $*$-pushforward functor
\[
\coind_H^K : \Rep(H) \to \Rep(K)
\]
is its continuous right adjoint.\footnote{The functor $\coind_H^K$ is sometimes called ``induction'' in the algebraic groups literature, see e.g.~\cite{jantzen}; here we will reserve this word for a \emph{left} adjoint to a restriction functor. Often this functor is considered only when $H$ is a closed subgroup of $K$, but its standard properties easily generalize to arbitrary morphisms.}
Note that $\res^K_H$ is t-exact, and hence that $\coind_H^K$ is right t-exact.

%%%%%%%%%%%%%%%%%%%%%%%%%%
\section{Weak actions}
\label{sec:weak-actions}
%%%%%%%%%%%%%%%%%%%%%%%%%%

%------------------------------------------------------
\subsection{Definition and first properties}
\label{ss:def-properties-weak}
%------------------------------------------------------

We continue to consider a classical affine group scheme $H$ of finite type over $k$.

\sss 
We define a weak categorical representation of $H$ to be a $\Rep(H)$-module in $\DGCat$. We therefore have an $\infty$-category
\[
\Rep(H)\mod
\]
of weak categorical representations of $H$. In case $\mathrm{char}(k)=0$ the notion of categorical representation has been studied extensively already; in this setting our definition is not the standard one, but is equivalent to the usual definition by well-known results; see~\S\ref{ss:equiv-deequiv} below for details.

\sss
One can consider the dualizability of weak categorical representations of $H$, see~\S\ref{sss:dualizability-modules-DGCat}. (Since $\Rep(H)$ is commutative, there is no need to distinguish between left and right dualizability or left and right modules.) In particular, by the discussion in~\S\ref{sss:rigidity-duals}, a weak categorical representation $\sC$ is dualizable if and only if it is dualizable in $\DGCat$, or in $\PrLSt$. Moreover, in this case the dual of $\sC$ is the weak categorical representation given by $\sC^\vee$ (the dual in $\DGCat$, or in $\PrLSt$), endowed with the natural action of $\Rep(H)$.

The following statement is an immediate consequence of the rigidity of $\Rep(H)$ (discussed in~\S\ref{sss:rigidity-Rep}), see~\cite[Chap.~1, \S 9.2]{gr} or~\cite[Proposition~7.105]{zhu}.

\begin{Lem}
\label{lem:Rep-self-dual}
$\Rep(H)$ is canonically self-dual, i.e., one has an equivalence of $\Rep(H)$-modules 
\[
\Rep(H) \simeq \Rep(H)^\vee,
\]
given via the pairing 
\[
\langle - , - \rangle_H: \Rep(H) \otimes_k \Rep(H) \to \Rep(H) \xrightarrow{\on{inv}_H} \Vect_k,
\]
where the first arrow is the binary product and $\on{inv}_H := \Hom_{\Rep(H)}(k, -)$ denotes the functor of invariants. 
\end{Lem}

\sss
Note that the functor of Lemma~\ref{lem:Rep-product}, when it is an equivalence, is compatible with self-dualities in the obvious way.

In the setting of~\S\ref{sss:Res-Coind}, the functors $\res^K_H$ and $\coind^K_H$
are exchanged under duality in the sense that one has equivalences
\[
(\res^K_H)^\vee \simeq \coind_H^K, \quad \quad (\coind_H^K)^\vee \simeq \res_H^K.
\]
Indeed, for compact objects $V$ and $W$ of $\Rep(K)$ and $\Rep(H)$ respectively, using the notation of Lemma~\ref{lem:Rep-self-dual} we have 
\begin{multline*} 
\langle \res^K_H(V), W \rangle_H \simeq \Hom_{\Rep(H)}(k, \res^K_H(V) \otimes_k W) \simeq \Hom_{\Rep(H)}( \res^K_H (V^*), W) \\ 
\simeq \Hom_{\Rep(K)}( V^*, \coind_H^K(W)) \simeq \Hom_{\Rep(K)}(k, V \otimes_k \coind_H^K(W))
\simeq \langle V, \coind_H^K(W) \rangle_K, 
\end{multline*} 
whence the result follows by taking filtered colimits.     

\sss
The following properties are also consequences of the rigidity of $\Rep(H)$.

\begin{Cor}
\phantomsection
\label{cor:rigidity-RepH}
\begin{enumerate}
\item
\label{it:adjoints-RepH-linear}
Given $\sC$, $\sD$ in $\Rep(H)\mod$ and a morphism $\phi: \sC \rightarrow \sD$ in $\Rep(H)\mod$,
if $\phi$ admits a left adjoint, resp.~a continuous right adjoint,
then this functor is again $\Rep(H)$-linear. 
\item
\label{it:Funct-Rep-tensor}
Given $\sC$, $\sD$ in $\Rep(H)\mod$, one has a canonical equivalence
\[
\Funct_{\Rep(H)}(\sC, \sD) \simeq \Rep(H) \otimes_{\Rep(H) \otimes_k \Rep(H)} \Funct_k(\sC,\sD).
\]
In particular, if $\sC$ is dualizable 
(e.g.~compactly generated), one has a canonical equivalence
\[
\Funct_{\Rep(H)}(\sC, \sD) \simeq \sC^\vee \otimes_{\Rep(H)} \sD.
\]
\end{enumerate}
\end{Cor}

\begin{proof}
\eqref{it:adjoints-RepH-linear}
This follows from~\cite[Chap.~1, Lemma~3.5.3 and Lemma~9.3.6]{gr}.

\eqref{it:Funct-Rep-tensor}
By the discussion in~\S\ref{sss:Funct-2}, we have
\[
\Funct_{\Rep(H)}(\sC, \sD) \simeq \Funct_{\Rep(H) \otimes_k \Rep(H)}(\Rep(H), \Funct_k(\sC,\sD)).
\]
By rigidity and self-duality of $\Rep(H)$ (see Lemma~\ref{lem:Rep-self-dual}), we deduce the desired equivalence.
In case $\sC$ is dualizable, we have $\Funct_k(\sC,\sD) \simeq \sC^\vee \otimes_k \sD$, and the equivalence follows from the discussion in~\S\ref{sss:claim-relative-tensor-product}.
\end{proof}

\sss 
It will be convenient in what follows to have alternative descriptions of relative tensor products and intertwining operators between $\Rep(H)$-modules, following the general discussion in~\S\ref{ss:rigidity}.
This description will be based on the following preliminary result, for which
we need to assume that the functor
\begin{equation}
\label{eqn:Rep-product-assumption}
\Rep(H) \otimes_k \Rep(H) \to \Rep(H \times_k H)
\end{equation}
of Lemma~\ref{lem:Rep-product} for $H_1=H_2=H$ is an equivalence (e.g.~that $\mathrm{char}(k)=0$, or $k$ is algebraically closed, or $H$ is split reductive, see Remark~\ref{rmk:assumptions-Rep-product}). 
Consider the regular bimodule 
\[
\OO((H \times_k H) / \Delta H) \simeq \OO(H),
\]
viewed as a commutative algebra object in $\Rep(H \times_k H)$. (Here, $\Delta H$ denotes the diagonal copy of $H$ in $H \times_k H$, and in the right-hand side $H \times_k H$ acts on $H$ via the left and right regular actions.) 

\begin{Prop}
\label{prop:Rep-H2-modules-OH}
Assume that the functor~\eqref{eqn:Rep-product-assumption} is an equivalence.
The functor $\coind_{\Delta H}^{H \times_k H}$ factors through an equivalence of $\infty$-categories
\[
\Rep(H) \simto \Mod_{\OO(H)}(\Rep(H \times_k H)).
\]
\end{Prop}

\begin{proof}
The multiplication map $\Rep(H) \otimes_k \Rep(H) \to \Rep(H)$ identifies via the equivalence~\eqref{eqn:Rep-product-assumption} with $\res^{H \times_k H}_{\Delta H}$. Its right adjoint therefore identifies with $\coind_{\Delta H}^{H \times_k H}$, so that the claim
follows from the discussion in~\S\ref{sss:rigidity}
and the identification of the algebra object $\coind_{\Delta H}^{H \times_k H}(k)$ with $\OO(H)$, endowed with its usual algebra structure.
\end{proof}

\sss
Let us now spell out the consequences of Proposition~\ref{prop:Rep-H2-modules-OH} obtained by specializing the discussion of~\S\ref{sss:rigidity-tens-Hom} to our particular setting.

\begin{Cor}
\label{cor:tens-Hom-weak-rep}
Assume that the functor~\eqref{eqn:Rep-product-assumption} is an equivalence, and
let $\sC$ and $\sD$ be weak categorical representations of $H$.
\begin{enumerate}
\item
\label{it:tens-Rep-modules-OH}
We have a canonical equivalence 
\[
\sC  \otimes_{\Rep(H)} \sD \simeq \Mod_{\OO(H)}(\sC \otimes_k \sD).
\]
\item
\label{it:Funct-weak}
We have a canonical equivalence
\[
\Funct_{\Rep(H)}(\sC, \sD) \simeq \Mod_{\OO(H)}(\Funct_k(\sC, \sD)).
\]
In case $\sC$ is dualizable, there is a canonical equivalence 
\[
\Funct_{\Rep(H)}(\sC, \sD) \simeq \Mod_{\OO(H)}(\sC^\vee \otimes_k \sD).
\]
\end{enumerate}
\end{Cor}

%---------------------------------------------
\subsection{Equivariant modules}
\label{ss:equiv-modules}
%---------------------------------------------

\sss 
\label{sss:def-AmodH}
Some basic examples of objects of $\Rep(H)\mod$ are provided by $\infty$-categories of equivariant modules. Namely, 
suppose $A$ is an algebra object in $\Rep(H)$. By the general results recalled in~\S\ref{sss:modules-oocat}, the $\infty$-category
\[
A\mod_H := \LMod_A(\Rep(H))
\]
is presentable and stable. It admits a canonical action of $\Rep(H)$, and thus
is a weak categorical representation of $H$.

Of course basic examples of algebra objects in $\Rep(H)$ are given by classical $H$-equivariant $k$-algebras, i.e.~classical $k$-algebras 
equipped with a compatible structure of $H$-module.

\sss 

We first determine the dual of $A\mod_H$.

\begin{Lem}
\label{lem:Amod-dualizable}
The $\infty$-category
$A\mod_H$ is dualizable, and there is a canonical equivalence of $\Rep(H)$-modules
\[
(A\mod_H)^\vee \simeq A^{\rev}\mod_H,
\]
where the pairing of $M \in A^{\rev}\mod_H$ and $N \in A\mod_H$ is given by 
\[
\langle M, N \rangle_{A\mod_H} \simeq \Hom_{\Rep(H)}(k, M \otimes_A N).
\]
\end{Lem}

\begin{proof} 
The statement is a special case of~\cite[Chap.~1, Corollary~8.6.3 and Proposition~9.4.4]{gr} since $\Rep(H)$ is rigid (see~\S\ref{sss:rigidity-Rep}).
\end{proof}

\sss 
We now describe tensor products with and the linear maps out of the $\infty$-category $A\mod_H$. 

\begin{Lem} 
\label{lem:Funct-Amod}
For any weak categorical representation $\sC$ of $H$,
there are canonical equivalences 
\[
A\mod_H \otimes_{\Rep(H)} \sC \simeq \LMod_{A}(\sC), \quad
\Funct_{\Rep(H)}(A\mod_H, \sC) \simeq \RMod_{A}(\sC).
\]
In particular, one has a monoidal equivalence
\[
\Funct_{\Rep(H)}(A\mod_H, A\mod_H) \simeq (A \otimes_k A^{\rev})\mod_H,
\]
i.e., the $\Rep(H)$-linear endomorphisms of $A\mod_H$ are given by $A$-bimodules in $\Rep(H)$. 
\end{Lem}

\begin{proof} 
The first equivalence follows from~\cite[Chap.~1, Corollary~8.5.7]{gr}, and the second one
from~\cite[Chap.~1, Corollary~8.6.4]{gr}. We then deduce the description of $\Funct_{\Rep(H)}(A\mod_H, A\mod_H)$ using~\cite[Chap.~1, Proposition~8.5.4]{gr}.
\end{proof}

\sss 
\label{sss:equiv-modules-Ind}
We conclude this section with an alternative description of the $\infty$-category $A\mod_H$. Consider as above an algebra object $A$ in $\Rep(H)$,
and denote by $\sP_A$ the full idempotent complete stable subcategory of $A\mod_H$ generated by the objects of the form $A \otimes_k V$, for $V$ in $\Rep(H)^\comp$. The following statement follows from the discussion in~\S\ref{sss:Amod-comp-gen}.

\begin{Lem}
\label{lem:equiv-modules-Ind}
With notation as above, we have a canonical equivalence
\[
A\mod_H \simeq \Ind(\sP_A).
\]
\end{Lem}

\begin{Rem}
\begin{enumerate}
\item
Given an algebra object $A$ in $\QCoh(\BB H)^+$, the image of $A$ in $\Rep(H)^+$ (see~\eqref{eqn:equiv-Rep+-QCoh+}) defines an algebra object in $\Rep(H)$, which we will still denote by $A$. In this case, in the procedure above, the full subcategory $\sP_A$ is contained in the full subcategory $\LMod_A(\Rep(\BB H)^+)$, which identifies with $\LMod_A(\QCoh(\BB H)^+)$.
\item
In the setting of Lemma~\ref{lem:equiv-modules-Ind}, in view of Lemma~\ref{lem:Funct-Amod}, we also have $\Funct_{\Rep(H)}(A\mod_H, A\mod_H) \simeq \Ind(\sP_{A \otimes_k A^{\rev}})$. Under this identification, the monoidal structure on $\Funct_{\Rep(H)}(A\mod_H, A\mod_H)$ is induced by the monoidal structure on $\sP_{A \otimes_k A^{\rev}}$ given by tensor products of bimodules.
\end{enumerate}
\end{Rem}

%---------------------------------------------
\subsection{Representations and group homomorphisms}
\label{ss:reps-morphisms}
%---------------------------------------------

\sss 
\label{sss:Rep-weak-rep}
Another basic class of examples of weak categorical representations which is worth studying is the following. Consider a morphism of $k$-group schemes of finite type $K \rightarrow H$.
Then the associated monoidal functor
\[
\res^H_K : \Rep(H) \rightarrow \Rep(K)
\]
(see~\S\ref{sss:Res-Coind})
lets us view $\Rep(K)$ as a weak categorical representation of $H$. In particular, when $K=\Spec(k)$ is the trivial group scheme, we obtain the trivial weak categorical representation $\Vect_k$ of $H$.

\sss
\label{sss:linearity-res-Rep}
Given affine $k$-group schemes of finite type $K_1$, $K_2$ and morphisms $K_1 \to K_2 \to H$, the restriction functor $\res^{K_2}_{K_1} : \Rep(K_2) \to \Rep(K_1)$ is clearly $\Rep(H)$-linear. By rigidity (see Corollary~\ref{cor:rigidity-RepH}\eqref{it:adjoints-RepH-linear}), its right adjoint $\coind_{K_1}^{K_2}$ is also $\Rep(H)$-linear. (This property is a counterpart in this setting of the ``generalized tensor identity,'' see~\cite[Part I, Proposition~4.8]{jantzen}.)

\sss
The tensor products 
of such categorical representations can be described as follows. 

\begin{Lem}
\label{lem:weakhecke}
Consider a pair of morphisms of $k$-group schemes of finite type $J \rightarrow H \leftarrow K$, and assume that the functors
\[
\Rep(H) \otimes_k \Rep(H) \to \Rep(H \times_k H), \qquad \Rep(J) \otimes_k \Rep(K) \to \Rep(J \times_k K)
\]
of Lemma~\ref{lem:Rep-product} are equivalences. Then
there is a canonical equivalence
\[
\Rep(J) \otimes_{\Rep(H)} \Rep(K) \simeq \OO(H)\mod_{J \times_k K}.
\]
\end{Lem}

\begin{proof}
This follows from Corollary~\ref{cor:tens-Hom-weak-rep}\eqref{it:tens-Rep-modules-OH} and our definitions and assumptions.
\end{proof}

\sss
\label{sss:tensor-Funct-Rep-morphisms}
Assume that the conditions in Lemma~\ref{lem:weakhecke} are satisfied, and 
consider the classical algebraic stack $J \bs H / K$. We have an affine morphism $J \bs H / K \to \BB(J \times_k K)$, which induces an equivalence
\[
\QCoh(J \bs H / K) \simeq \Mod_{\OO(H)}(\QCoh(\BB(J \times_k K))),
\]
see~\cite[Chap.~3, Proposition~3.3.3]{gr}. By Lemma~\ref{lem:t-structures}, we have a canonical t-structure on the right-hand side induced by the t-structure on $\QCoh(\BB(J \times_k K))$, and this equivalence is clearly t-exact. It therefore restricts to an equivalence
\[
\QCoh(J \bs H / K)^+ \simeq \Mod_{\OO(H)}(\QCoh(\BB(J \times_k K))^+),
\]
which by~\eqref{eqn:equiv-Rep+-QCoh+} can be interpreted as an equivalence
\begin{equation}
\label{eqn:equiv-QCoh-JHK-OH-modules}
\QCoh(J \bs H / K)^+ \simeq \Mod_{\OO(H)}(\Rep(J \times_k K)^+).
\end{equation}
The $\infty$-category $\sP_{\OO(H)}$ considered in~\S\ref{sss:equiv-modules-Ind} (where $\OO(H)$ is seen as an algebra object in $\Rep(J \times_k K)$) is clearly contained in the right-hand side.
By Lemma~\ref{lem:equiv-modules-Ind},
$\OO(H)\mod_{J \times_k K}$ is obtained from the stable $\infty$-category
\[
\QCoh(J \bs H / K)
\]
by Ind-completing the full idempotent complete stable subcategory generated by the image under the pullback functor $\QCoh(\BB(J \times_k K)) \to \QCoh(J \bs H / K)$ of $\Coh(\BB(J \times_k K)) = \Rep(J \times_k K)^\comp$. This full subcategory is contained in $\Coh(J \bs H / K)$, but this inclusion might be strict.

\begin{Rem}
\phantomsection
\label{rmk:smoothness-Rep}
\begin{enumerate}
\item
\label{it:smoothness-Rep-char0}
If $\mathrm{char}(k)=0$, the full idempotent complete stable subcategory of $\QCoh(J \bs H / K)$ generated by the image of $\Rep(J \times_k K)^\comp$ is $\Coh(J \bs H / K)$. In fact, any object of $\Coh(J \bs H / K)^\heartsuit$ is a quotient of the pullback of an object in $\Rep^{\heartsuit,\mathrm{fd}}(J \times_k K)$. Now since $J$ and $K$ have finite cohomological dimension, the $\infty$-category $\Coh(J \bs H / K)^\heartsuit$ has finite cohomological dimension: there exists an integer $N$ such that $\Ext^i(\mathcal{F},\mathcal{G})=0$ for any $\mathcal{F}$, $\mathcal{G}$ in $\Coh(J \bs H / K)^\heartsuit$ and $i>N$. Then given $\mathcal{F} \in \Coh(J \bs H / K)^\heartsuit$ there exists a complex $\mathcal{G}^\bullet$ which is $0$ except in degrees between $-N$ and $0$, whose nonzero terms are pullbacks of objects in $\Rep^{\heartsuit,\mathrm{fd}}(J \times_k K)$, whose cohomology is $0$ except in degrees $-N$ and $0$, and a morphism $\mathcal{G}^\bullet \to \mathcal{F}$ which induces an isomorphism $\mathcal{H}^0(\mathcal{G}^\bullet) \simto \mathcal{F}$. If $\mathcal{K}$ is the kernel of the differential $d^{-N} : \mathcal{G}^{-N} \to \mathcal{G}^{-N+1}$, we then have a fiber sequence
\[
\mathcal{K}[N] \to \mathcal{G}^\bullet \to \mathcal{F}.
\]
The connecting morphism $\mathcal{F} \to \mathcal{K}[N+1]$ in this fiber sequence vanishes by our choice of $N$, so that $\mathcal{F}$ is isomorphic to a direct summand of $\mathcal{G}^\bullet$, which finishes the proof of our claim.
\item
\label{it:smoothness-Rep-flag}
Another case where the full idempotent complete stable subcategory of $\QCoh(J \bs H / K)$ generated by the image of $\Rep(J \times_k K)^\comp$ is $\Coh(J \bs H / K)$ is when $H=G$ is a split reductive group scheme over $k$ and $K=P$ is a parabolic subgroup. In fact, in this case it is known that there exists a full exceptional collection in $\Coh(G/P)$ consisting of objects obtained by pullback from $\Coh(\BB P)$, see~\cite{svdk}. Then the claim follows by applying~\cite[Theorem~2.6]{elagin}.
\end{enumerate}
\end{Rem}

\sss
\label{sss:self-duality-RepK}
Let $H$ and $K$ be as in~\S\ref{sss:Rep-weak-rep}. It is not difficult to see that $\Rep(K)$ is dualizable and self-dual as a $\Rep(H)$-module. But one can also construct the corresponding duality datum ``by hand'' as follows. For this we assume that the canonical functors $\Rep(H) \otimes_k \Rep(H) \to \Rep(H \times_k H)$ and $\Rep(K) \otimes_k \Rep(K) \to \Rep(K \times_k K)$ are equivalences. (By~\cite[Lemma~8.20]{zhu}, this implies similar properties for the products $K \times_k K \times_k K$ and $H \times_k K$.)
Consider the composition
\begin{equation}
\label{eqn:duality-Rep-1}
\Rep(K) \otimes_k \Rep(K) \simeq \Rep(K \times_k K) \xrightarrow{\res^{K \times_k K}_{\Delta K}} \Rep(K) \xrightarrow{\coind_K^H} \Rep(H)
\end{equation}
(where $\Delta K$ is the diagonal copy of $K$ in $K \times_k K$).
On the other hand, the (pushforward of the) structure sheaf $\OO_K$ determines an object in $\QCoh(K \backslash H/K)^+$, and hence via~\eqref{eqn:equiv-QCoh-JHK-OH-modules} an object in $\OO(H)\mod_{K \times_k K}$. Using Lemma~\ref{lem:weakhecke}, this object determines a functor
\begin{equation}
\label{eqn:duality-Rep-2}
\Vect_k \to \Rep(K) \otimes_{\Rep(H)} \Rep(K) \simeq \OO(H)\mod_{K \times_k K}.
\end{equation}

\begin{Lem}
\label{lem:duality-datum-Rep}
Assume that the conditions in~\S\ref{sss:self-duality-RepK} are satisfied.
The maps~\eqref{eqn:duality-Rep-1} and~\eqref{eqn:duality-Rep-2} exhibit $\Rep(K)$ as its own dual as a $\Rep(H)$-module.
\end{Lem}

\begin{proof}
We have to prove that two endomorphisms of $\Rep(K)$ are homotopic to the identity. We consider one of them; the other one can be analyzed similarly. This endomorphism is the composition
\[
\Rep(K) \to \Rep(K) \otimes_k \Rep(K) \otimes_{\Rep(H)} \Rep(K) \to \Rep(H) \otimes_{\Rep(H)} \Rep(K) \simeq \Rep(K)
\]
where the first map is induced by~\eqref{eqn:duality-Rep-2} and the second one by~\eqref{eqn:duality-Rep-1}. Here the second $\infty$-category identifies with $\Mod_{\OO(H)}(\Rep(K \times_k K \times_k K))$, and the third one with $\Mod_{\OO(H)}(\Rep(H \times_k K))$. Moreover, via these identifications the first functor is given by $M \mapsto M \otimes_k \OO(K)$, the second one is induced by the composition
\[
\Rep(K \times_k K \times_k K) \xrightarrow{\res^{K \times_k K \times_k K}_{\Delta K \times_k K}} \Rep( K \times_k K) \xrightarrow{\coind_{K \times_k K}^{H \times_k K}} \Rep(H \times_k K),
\]
and the identity of $\Rep(K)$ corresponds to the functor $M \mapsto \OO(H) \otimes_k M$ where the action of $\OO(H)$ is by multiplication on the first factor, the action of $H$ is via the action on $\OO(H)$, and the action of $K$ is diagonal. The desired identification is given by the canonical equivalences $\coind_K^H(M \otimes_k \OO(K)) = \coind_K^H \circ \coind_{\{1\}}^K(M) \simeq \coind_{\{1\}}^H(M) \simeq M \otimes_k \OO(H)$ (see the discussion of the tensor identity in~\S\ref{sss:linearity-res-Rep}).
\end{proof}

\begin{Rem}
The self-duality of $\Rep(K)$ implies that, in the setting of Lemma~\ref{lem:weakhecke}, we also have 
\[
\Funct_{\Rep(H)}(\Rep(J), \Rep(K)) \simeq \OO(H)\mod_{J \times_k K}.
\]
\end{Rem}

\sss
The following statement uses the notions discussed in~\S\ref{sss:smooth-proper-modules}.

\begin{Cor}
\label{cor:smoothness-properness-Rep}
Assume that 
the conditions in~\S\ref{sss:self-duality-RepK} are satisfied.
\begin{enumerate}
\item
\label{it:duality-Rep-smooth}
If the structure sheaf $\OO_K$ belongs to the full idempotent-complete stable subcategory of $\QCoh(K \backslash H/K)$ generated by the image under the pullback functor
\[
\QCoh(\BB(K \times_k K)) \to \QCoh(K \backslash H/K)
\]
of $\Rep(K \times_k K)^\comp$, then $\Rep(K)$ is smooth as a $\Rep(H)$-module.
\item
\label{it:duality-Rep-proper}
If $K$ is a subgroup scheme of $H$ and the quotient $H/K$ is projective, then $\Rep(K)$ is proper as a $\Rep(H)$-module.
\end{enumerate}
\end{Cor}

\begin{proof}
We have constructed the duality datum for $\Rep(K)$ in Lemma~\ref{lem:duality-datum-Rep}; what remains to be discussed is when these functors admit linear right adjoints, i.e.~when they admit continuous right adjoints, i.e.~when they send compact objects to compact objects (see~\cite[Chap.~1, Lemma~7.1.5]{gr}). For~\eqref{it:duality-Rep-smooth} it suffices to consider the image of $k$, i.e.~$\OO_K$, and our assumption exactly says that this object is compact, see~\S\ref{sss:tensor-Funct-Rep-morphisms}. For~\eqref{it:duality-Rep-proper} we have to study when $\coind_K^H$ sends compact objects to compact objects, and the claim follows from~\cite[Part~I, Proposition~5.12]{jantzen}.
\end{proof}

\begin{Rem}
\label{rmk:smoothness-properness-Rep}
By Remark~\ref{rmk:smoothness-Rep}, the assumption in Corollary~\ref{cor:smoothness-properness-Rep}\eqref{it:duality-Rep-smooth} is automatically satisfied if $\mathrm{char}(k)=0$, or if $H$ is a split reductive group and $K$ is a parabolic subgroup. It is not clear to us when this assumption is satisfied beyond these cases. 
\end{Rem}

%------------------------------------------------------
\subsection{Equivariantization and deequivariantization}
\label{ss:equiv-deequiv}
%------------------------------------------------------

\sss 
We will now apply the discussion of~\S\ref{ss:reps-morphisms} in the framework of (de)equivariantization. We assume that the functor
$\Rep(H) \otimes_k \Rep(H) \to \Rep(H \times_k H)$
of Lemma~\ref{lem:Rep-product} is an equivalence.

Consider the trivial weak categorical representation $\Vect_k$ of $H$, see~\S\ref{sss:Rep-weak-rep}. In this case, by Lemma~\ref{lem:weakhecke} we have
\begin{equation}
\label{eqn:End-Vect-Rep}
\Funct_{\Rep(H)}(\Vect_k, \Vect_k) \simeq \Mod_{\OO(H)}(\Vect_k) \simeq \QCoh(H),
\end{equation}
and this equivalence is an equivalence of algebra objects for the structure on the left-hand side obtained from the general considerations of~\S\ref{sss:def-Amod-DGCat} and the structure on the right-hand side discussed in~\S\S\ref{sss:QCohH-algebra-object}--\ref{sss:QCohH-algebra-object-2}.

\begin{Rem}
In the setting above, the first equivalence in Lemma~\ref{lem:weakhecke} provides an equivalence
\begin{equation}
\label{eqn:Tens-Vect-Rep}
\Vect_k \otimes_{\Rep(H)} \Vect_k \simeq \QCoh(H).
\end{equation}
See~\cite[Corollary~2.1.5(i) and Corollary~2.3.2(i)]{tao} for these equivalences in case $\mathrm{char}(k)=0$.
\end{Rem}

\sss
In particular,~\eqref{eqn:End-Vect-Rep} means that $\Vect_k$ is equipped with commuting actions of $\Rep(H)$ and $\QCoh(H)$. (Concretely, the action of $\QCoh(H)$ is the one considered in~\S\ref{sss:action-QCoh}.) 
Using this structure we obtain a continuous adjunction of $\DGCat$-linear $\infty$-categories
\[
(-) \otimes_{\QCoh(H)} \Vect_k : \modr\QCoh(H) \rightleftarrows \Rep(H)\mod: \Funct_{\Rep(H)}(\Vect_k, -).
\]
(Here, continuity and $\DGCat$-linearity of the functor $\Funct_{\Rep(H)}(\Vect_k, -)$ follow from its identification with $\Vect_k \otimes_{\Rep(H)}  -$, which is a consequence of rigidity of $\Rep(H)$, see Corollary~\ref{cor:rigidity-RepH}\eqref{it:Funct-Rep-tensor}. Note that $\QCoh(H)$ is \emph{not} rigid since it is compactly generated but the binary product does not preserve compactness.)

\begin{Prop}
\label{prop:QCoh-Rep-ff}
Assume that the functor~\eqref{eqn:Rep-product-assumption} is an equivalence.
The functor
\[
(-) \otimes_{\QCoh(H)} \Vect_k : \modr\QCoh(H) \to \Rep(H)\mod
\]
is fully faithful; in other words it
exhibits $\modr\QCoh(H)$ as a full subcategory of the $\infty$-category of weak categorical representations of $H$.
\end{Prop}  

\begin{proof}
Since we are considering a left adjoint,
to show full faithfulness we need to show that the unit natural transformation for the adjunction is an equivalence. This follows from the following identifications for $\sC$ in $\modr\QCoh(H)$:
\begin{align*}
\Funct_{\Rep(H)}(\Vect_k, \sC \otimes_{\QCoh(H)} \Vect_k) & \simeq \Vect_k \otimes_{\Rep(H)} (\sC \otimes_{\QCoh(H)} \Vect_k) \\
 & \simeq \sC \otimes_{\QCoh(H)} (\Vect_k \otimes_{\Rep(H)} \Vect_k) \\ 
 & \simeq \sC \otimes_{\QCoh(H)} \QCoh(H) \\ 
 &\simeq \sC. 
\end{align*}
Here the first equivalence uses Corollary~\ref{cor:rigidity-RepH}\eqref{it:Funct-Rep-tensor}, the second and fourth ones are straightforward, and the third one follows from~\eqref{eqn:Tens-Vect-Rep}.
\end{proof}

\begin{Rem}
We have stated Proposition~\ref{prop:QCoh-Rep-ff} with right $\QCoh(H)$-modules because this is the structure that naturally emerges from its proof. But the map $h \mapsto h^{-1}$ identifies $H$ with the opposite group, which provides an equivalence of algebra objects $\QCoh(H) \simeq \QCoh(H)^\rev$ and hence allows us to pass back and forth between left and right $\QCoh(H)$-modules, see~\S\ref{sss:reversed-algebra}.
\end{Rem}

\sss
In the case when $\mathrm{char}(k)=0$, it is a standard result (initially due to Gaitsgory and Lurie) that the fully faithful embedding of
Proposition~\ref{prop:QCoh-Rep-ff} is an equivalence of $\infty$-categories, see~\cite[Theorem~2.3.7]{beraldo} (see also~\cite{tao,dhillon} for other discussions). A closely related property is that for any $\QCoh(H)$-module $\sC$ we have an equivalence
\begin{equation}
\label{eqn:inv-coinv-weak}
\Vect_k \otimes_{\QCoh(H)} \sC \simto \Funct_{\QCoh(H)}(\Vect_k,\sC)
\end{equation}
(although $\QCoh(H)$ is not rigid, as noted above), see~\cite[Corollary~2.3.8]{beraldo}; in other words, invariants and coinvariants coincide for $\QCoh(H)$-modules.

For a general field $k$, the functor of Proposition~\ref{prop:QCoh-Rep-ff}
is typically not an equivalence, i.e.~the counit map of the adjunction is not an equivalence. Namely, applying this counit map to $\Rep(H)^\vee \simeq \Rep(H)$
we obtain a functor
\[
\Vect_k \otimes_{\QCoh(H)} \Vect_k \to \Rep(H).
\]
Applying $\Funct_{\Vect_k}({-},\Vect_k)$,
 we obtain a functor 
\[
\Rep(H) \rightarrow \Funct_{\QCoh(H)}(\Vect_k, \Vect_k).
\]
Now, by Lemma~\ref{lem:fixed-pts-QCoh-quotient}
the right-hand side identifies with
$\QCoh(\Spec(k)/H) \simeq \QCoh(\BB H)$ (see Remark~\ref{rmk:sheafification}),
and via this identification the above functor is the functor $\Psi_{\BB H}$ of~\S\ref{sss:Psi},
which in general is not an equivalence (see the discussion in~\S\ref{sss:IndCoh-finite-cohomological-dim}). Said differently, we have genuinely renormalized the naive category of weak categorical representations of $H$, and as a consequence the functors of equivariantization and de-equivariantization are no longer inverse equivalences. 

On a related note, we expect that the natural functor as in~\eqref{eqn:inv-coinv-weak} is \emph{not} an equivalence for a general field $k$ and a general $\QCoh(H)$-module $\sC$. We have however not been able to construct an explicit counterexample. (The difficulty here is that the $\infty$-category $\Vect_k \otimes_{\QCoh(H)} \sC$ is difficult to compute in nontrivial cases.)

%---------------------------------------------------
\subsection{Categorical traces}
\label{ss:categorical-trace-weak}
%---------------------------------------------------

\sss
\label{sss:cat-trace-Rep}
We again consider a group scheme $H$ of finite type over $k$, and assume that the functor
$\Rep(H) \otimes_k \Rep(H) \to \Rep(H \times_k H)$
of Lemma~\ref{lem:Rep-product} is an equivalence.

Consider the (commutative) algebra object $\Rep(H)$ in $\DGCat$. Given any $\Rep(H)$-bimodule $M$ one can consider its trace, or Hochschild homology,
\[
\Tr(\Rep(H), M) = \Rep(H) \otimes_{\Rep(H) \otimes_k \Rep(H)} M,
\]
see e.g.~\cite[\S 7.3.1]{zhu}. This $\infty$-category is equipped with a canonical functor $[-]_M : M \to \Tr(\Rep(H), M)$.

Due to our assumption, one can consider $M$ as a $\Rep(H \times_k H)$-module, and we have
\begin{equation}
\label{eqn:trace-tensor-product}
\Tr(\Rep(H), M) = \Rep(H) \otimes_{\Rep(H \times_k H)} M
\end{equation}
where the functor $\Rep(H \times_k H) \to \Rep(H)$ is restriction to the diagonal copy of $H$. As in the proof of Corollary~\ref{cor:tens-Hom-weak-rep}\eqref{it:tens-Rep-modules-OH}, we deduce an identification
\begin{equation}
\label{eqn:trace-modules-OH}
\Tr(\Rep(H), M) = \Mod_{\OO(H)}(M)
\end{equation}
such that the canonical functor $[-]_M$ corresponds to the induction functor $M \to \Mod_{\OO(H)}(M)$, see~\S\ref{sss:modules-oocat}.

\begin{Rem}
As explained in~\cite[Remark~7.65]{zhu}, $\Tr(\Rep(H), M)$ can be interpreted as a trace in a usual sense (a certain endomorphism of a unit object), applied to the Morita $(\infty,2)$-category associated with the symmetric monoidal $\infty$-category $\DGCat$.
\end{Rem}

\sss
In particular, following~\cite[Example~7.67]{zhu}, given an algebra endomorphism $\phi : \Rep(H) \to \Rep(H)$ we can consider the $\Rep(H)$-bimodule ${}^\phi \hspace{-1pt} \Rep(H)$ which is $\Rep(H)$ with the right action given by the natural action, and the left action given by the twist by $\phi$ of the natural action. We set
\[
\Tr(\Rep(H), \phi) := \Tr(\Rep(H), {}^\phi \hspace{-1pt} \Rep(H)), \quad [-]_\phi := [-]_{{}^\phi \hspace{-1pt} \Rep(H)} : \Rep(H) \to \Tr(\Rep(H), \phi).
\]

\begin{Rem}
\label{rmk:object-trace}
Consider a dualizable $\Rep(H)$-module $M$, and denote by ${}^\phi M$ the $\Rep(H)$-module obtained by twisting the $\Rep(H)$-action by $\phi$. Following~\cite[\S 7.3.2]{zhu}, given a morphism $\alpha : M \to {}^\phi M$ we obtain via functoriality of vertical traces a morphism $[M,\alpha]_\phi : \Vect_k \to \Tr(\Rep(H), \phi)$ (explicitly described in~\cite[Eqn.~(7.61)]{zhu}), i.e.~an object of $\Tr(\Rep(H), \phi)$.
\end{Rem}

\sss
\label{sss:trace-endomorphism}
Specializing even more, consider a group homomorphism $\varphi : H \to H$. Then we have the associated pullback functor $\phi = \varphi^*$ which is an algebra morphism, see~\S\ref{sss:Res-Coind}. In this case, we still use the notation
\[
\Tr(\Rep(H), \varphi) := \Tr(\Rep(H), \varphi^*), \quad [-]_\varphi := [-]_{\varphi^*}.
\]
In view of~\eqref{eqn:trace-tensor-product} we have
\[
\Tr(\Rep(H), \varphi) = \Rep(H) \otimes_{\Rep(H \times_k H)} \Rep(H)
\]
where on the right-hand side $\Rep(H)$ is seen as a $\Rep(H \times_k H)$-module via the monoidal functor $\Rep(H \times_k H) \to \Rep(H)$ given by restriction along the morphism $(\varphi, \id) : H \to H \times_k H$. Applying~\eqref{eqn:trace-modules-OH} we also have
\[
\Tr(\Rep(H), \varphi) = \Mod_{\OO(H)}(\Rep(H)) = \OO(H)\mod_H,
\]
where $\OO(H)$ is seen as an algebra object in $\Rep(H)$ for the natural multiplication and the action induced by the action of $H$ on itself given by $g \cdot h = \varphi(g)hg^{-1}$ for $g,h \in H$. Note that, by the 
same considerations as in~\S\ref{sss:tensor-Funct-Rep-morphisms}, 
$\OO(H)\mod_H$ identifies with the Ind-completion of the full stable idempotent complete subcategory of $\Coh(H/H)$ generated by the essential image of the functor $\Rep(H)^\comp = \Coh(\BB H) \to \Coh(H/H)$ given by pullback under the projection $H/H \to \BB H$, where $H$ acts on itself as above. Under this identification, the functor $[-]_\phi$ is the Ind-completion of this pullback functor.

\sss
\label{sss:trace-Rep-id}
Now we assume that $k$ is algebraically closed (which automatically implies that the assumption of~\S\ref{sss:cat-trace-Rep} is satisfied, see Remark~\ref{rmk:assumptions-Rep-product}), and that $H$ is smooth. We assume furthermore that the following conditions are satisfied:
\begin{enumerate}
\item
the identity component $H^\circ$ is a connected reductive algebraic group;
\item
the order of the group of components $H/H^\circ$ is invertible in $k$;
\item
the order of the torsion part of the fundamental group of $H^\circ$ is invertible in $k$.
\end{enumerate}
In this setting we consider the construction of~\S\ref{sss:trace-endomorphism} in the case when $\varphi=\id$.

\begin{Prop}
\label{prop:trace-Rep-id}
Under the assumptions above, we have
\[
\Tr(\Rep(H), \id) = \IndCoh(H / H),
\]
where $H$ acts on itself via the adjoint action, and the functor $[-]_{\id} : \Rep(H) \to \IndCoh(H / H)$ is induced by pullback along the projection morphism $H / H \to \BB H$.
\end{Prop}

\begin{proof}
This follows from the discussion in~\S\ref{sss:trace-endomorphism} together with the fact, proved in~\cite[Proposition~VIII.5.11]{fs}, that under our assumptions $\Coh(H / H)$ is generated, as an idempotent complete stable $\infty$-category, by the essential image of $\Rep(H)^\comp$. (The statement in~\cite{fs} is written in terms of perfect complexes, but $\Coh(H / H)$ coincides with the subcategory of perfect complexes in $\QCoh(H / H)$ since $H$ is smooth.)
\end{proof}

\begin{Rem}
Consider the construction of Remark~\ref{rmk:object-trace} with $M=\Rep(K)$ where $K$ is a subgroup scheme of $H$ and $\alpha=\id$. Under the assumptions of Proposition~\ref{prop:trace-Rep-id} and those of~\S\ref{sss:self-duality-RepK}, using the explicit construction of the duality datum in Lemma~\ref{lem:duality-datum-Rep} one can check that the object of $\IndCoh(H / H)$ produced in this way is the pushforward of the structure sheaf under the $H$-equivariant morphism $H \times^K K \to H$ (where $K$ acts on itself via conjugation, and $H$ acts by multiplication on the left factor) sending $[h:k]$ to $hkh^{-1}$. (In case $H$ is a connected reductive algebraic group and $K$ is a Borel subgroup, this object is the ``coherent Grothendieck--Springer sheaf.'')
\end{Rem}

\sss
\label{sss:trace-Rep-Frob}
Let us drop the assumptions of the preceding proposition. Instead we consider a finite field $k_\circ$ and a smooth affine group scheme $H_\circ$ of finite type over $k_\circ$. Let $k$ be an algebraic closure of $k_\circ$ and $H$ be the base change of $H_\circ$ to $k$. (As above, $H$ automatically satisfies the assumption of~\S\ref{sss:cat-trace-Rep}.) Then we have the geometric Frobenius endomorphism $\Frob$ of $H$, and the finite group scheme $H^\Frob$ over $k$ of $\Frob$-fixed points of $H$, which identifies with the constant group scheme associated with the finite group $H_\circ(k_\circ)$. (For this, see e.g.~\cite[Chap.~3]{dm}.) We will assume that $H$ is connected; by the Lang--Steinberg theorem (see~\cite[Theorem~3.10]{dm}), this implies that the map $h \mapsto \Frob(h) \cdot h^{-1}$ induces an isomorphism
\begin{equation}
\label{eqn:isom-Lang}
H / H^\Frob \simto H,
\end{equation}
where the natural action of $H$ on the left-hand side corresponds to the action on the right-hand side from~\S\ref{sss:trace-endomorphism}, for $\varphi=\Frob$.

Let us consider the $\infty$-categories $\Rep(H)$ and $\Rep(H^\Frob)$. Here $\Rep(H^\Frob)$ is the Ind-completion of the bounded derived $\infty$-category of the abelian category $\Rep^{\heartsuit,\fd}_k(H_\circ(k_\circ))$ of finite-dimensional representations of the finite group $H_\circ(k_\circ)$ over $k$, see~\S\ref{sss:Repc-DbRepfg}. We have a canonical (exact) restriction functor
\begin{equation}
\label{eqn:restriction-Rep-Frob}
\Rep^{\heartsuit,\fd}(H) \to \Rep^{\heartsuit,\fd}_k(H_\circ(k_\circ)).
\end{equation}
A second condition that we will consider below is that any simple object in $\Rep^{\heartsuit,\fd}_k(H_\circ(k_\circ))$ is the image of an object of $\Rep^{\heartsuit,\fd}(H)$, i.e.~any simple representation of $H_\circ(k_\circ)$ over $k$ extends to an algebraic representation of $H$. This condition is satisfied e.g.~if $H$ is a reductive group scheme with simply connected derived subgroup, see~\cite[Remark~3.6]{brunat-luebeck}. (In this case, the datum of $H_\circ$, or equivalently of $\Frob$, can be described in terms of root data, see~\cite[Theorem~3.17]{dm} or~\cite[\S 2.3]{brunat-luebeck}. The case when $H$ is semisimple---and simply connected---is classical, and due to Steinberg, see~\cite[\S 3.2]{brunat-luebeck}.)

\begin{Prop}
\label{prop:trace-Rep-Frob}
Assume that $H$ is connected. Then $\Tr(\Rep(H), \Frob)$ is the Ind-completion of the full stable subcategory of $\Rep(H^\Frob)^\comp$ generated by the essential image of~\eqref{eqn:restriction-Rep-Frob}, and the functor $[-]_{\Frob} : \Rep(H) \to \Tr(\Rep(H), \Frob^*)$ identifies with the restriction functor $\res^H_{H^\Frob} : \Rep(H) \to \Rep(H^\Frob)$.

In particular, if any simple $k$-representation of $H_\circ(k_\circ)$ is the restriction of a simple $H$-module, then we have
\[
\Tr(\Rep(H), \Frob) = \Rep(H^\Frob)
\]
and the functor $[-]_{\Frob}$ identifies with the restriction functor $\res^H_{H^\Frob}$.
\end{Prop}

\begin{proof}
Recall the discussion in~\S\ref{sss:trace-endomorphism}. In the present setting we need to consider the algebraic stack $H / H$ where $H$ acts on itself by $g \cdot h = \Frob(g) h g^{-1}$. By~\eqref{eqn:isom-Lang} this action identifies $H$ with $H/H^\Frob$. We therefore have $H / H = \BB H^\Frob$. Under this identification, the map $H / H \to \BB H$ identifies with the canonical map $\BB H^\Frob \to \BB H$, so what we need to consider is the full stable idempotent complete subcategory of $\Coh(\BB H^\Frob) = \Db \Rep^{\heartsuit,\fd}_k(H_\circ(k_\circ))$ generated by the image of the restriction functor~\eqref{eqn:restriction-Rep-Frob}, as claimed. 
\end{proof}

\begin{Rem}
Assume that $H$ is connected, and
consider the construction of Remark~\ref{rmk:object-trace} with $M=\Rep(K)$ where $K$ is a connected subgroup scheme of $H$ obtained by base change from a subgroup scheme $K_\circ$ of $H_\circ$, and with $\alpha$ given by pullback under the Frobenius morphism of $K$. 
Using the explicit construction of the duality datum in Lemma~\ref{lem:duality-datum-Rep} one can check that the object of $\OO(H/H^\Frob)\mod_H$
produced in this way is $\OO(H/K^\Frob)$. Under the second assumption in Proposition~\ref{prop:trace-Rep-Frob}, the corresponding representation of $H^\Frob$ is $\coind_{K^\Frob}^{H^\Frob}(k)$.
\end{Rem}

%%%%%%%%%%%%%%%%%%%%%%%%%%%%%%
\section{The monoidal \texorpdfstring{$\infty$}{infinity}-category of Harish-Chandra bimodules}
\label{sec:hc}
%%%%%%%%%%%%%%%%%%%%%%%%%%%%%%

We fix a $k$-group scheme of finite type $H$, and denote its Lie algebra by $\fh$. Our goal in this section is to define and study an algebra object $\hc_H$ in $\DGCat$ whose underlying $\infty$-category is an appropriate derived $\infty$-category of Harish-Chandra bimodules for $H$. This construction will follow the pattern of~\S\ref{sss:algebra-structure-Verdier-quotient}, so we start by considering abelian and additive versions.

%---------------------------------------------------
\subsection{The abelian category of Harish-Chandra bimodules}
%---------------------------------------------------

\sss 
\label{sss:hc-bimodules-def}
Recall that a \emph{Harish-Chandra bimodule} for $H$ is an $H$-equivariant $U(\mathfrak{h})$-bimodule such that the differential of the action of $H$ coincides with the restriction of the action of $U(\mathfrak{h}) \otimes_k U(\mathfrak{h})^{\rev}$ along the morphism $\Delta : U(\fh) \to U(\fh) \otimes_k U(\fh)^{\rev}$ sending $x \in \fh$ to $x \otimes 1 - 1 \otimes x$. (These objects are a classical object of study when $k=\mathbb{C}$. For recent occurrences over fields of positive characteristic, see~\cite{br-Hecke,br,losev}.)
We will denote by $\hc_{H}^\heartsuit$ the abelian category of Harish-Chandra bimodules. This category admits a natural monoidal structure, induced by tensor product $\otimes_{U(\fh)}$ of bimodules. (Note that the monoidal product is \emph{not} exact.)

\sss
\label{sss:hc-fg}
It is easily seen from local finiteness of $H$-modules that an object of $\hc_H^\heartsuit$ is finitely generated as a left $U(\fh)$-module if and only if it is finitely generated as a right $U(\fh)$-module, and if and only if it is finitely generated as a $U(\fh)$-bimodule. (Such Harish-Chandra bimodules will simply be called finitely generated.) We will denote by
$\hc_H^{\heartsuit, \mathrm{fg}}$ the full subcategory of $\hc_H^\heartsuit$ spanned by objects which satisfy this property. Then $\hc_H^{\heartsuit, \mathrm{fg}}$ is a monoidal Serre subcategory of $\hc_H^{\heartsuit}$.

\sss
\label{sss:diagonally-induced-hc-bimodules}
Recall from~\S\ref{sss:Rep-abelian} the abelian category $\Rep^{\heartsuit}(H)$ of 
$H$-modules. We have a natural exact monoidal functor 
\begin{equation}
\label{eqn:induction-Rep-hc}
\ind_{\hc}^H:
\Rep^{\heartsuit}(H) \rightarrow \hc_{H}^\heartsuit, \qquad V \mapsto \ind_{\hc}^H(V) = \bigl( U(\fh) \otimes_k U(\fh)^{\rev} \bigr) \otimes_{U(\fh)} V,
\end{equation}
sending a representation to its induction along $\Delta$, which is 
left adjoint to the functor $\hc_{H}^\heartsuit \to \Rep^{\heartsuit}(H)$ of forgetting the actions of $U(\fh)$. 
This functor restricts to a functor
\[
\Rep^{\heartsuit,\fd}(H) \rightarrow \hc_{H}^{\heartsuit, \mathrm{fg}}.
\]

For any $V \in \Rep^{\heartsuit}(H)$ we have an identification
\[
\ind_{\hc}^H(V) \simeq V \otimes_k U(\fh)
\]
where on the right-hand side $H$ acts diagonally, the left action of $U(\fh)$ is diagonal using left multiplication on the right factor, and the right action of $U(\fh)$ is given by right multiplication on the right factor. In other words, using Sweedler's notation we have
\[
a \cdot (v \otimes x) \cdot b = (a_{(1)} \cdot v) \otimes a_{(2)} x b 
\]
for $a,b,x \in U(\fh)$ and $v \in V$.

We also have an identification
\[
\ind_{\hc}^H(V) \simeq U(\fh) \otimes_k V
\]
where on the right-hand side $H$ acts diagonally, the left action of $U(\fh)$ is given by left multiplication on the left factor, and the right action of $U(\fh)$ is diagonal using right multiplication on the left factor; in other words the actions of $U(\fh)$ are now defined by
\[
a \cdot (x \otimes v) \cdot b = ax b_{(1)} \otimes (S(b_{(2)}) \cdot v)
\]
for $a,b,x \in U(\fh)$ and $v \in V$, where $S$ is the antipode.
For details, see e.g.~\cite[\S 3.4]{br-Hecke}.\footnote{There is a typo in the description of the isomorphism $V \otimes U(\fh) \simto U(\fh) \otimes V$ in~\cite{br-Hecke}, which should read $v \otimes x \mapsto x_{(1)} \otimes (S(x_{(2)}) \cdot v)$.}

\sss
Permutation of the two actions of $U(\mathfrak{h})$ (using the anti-automorphism of $U(\mathfrak{h})$ given by the antipode) yields a canonical involutive auto-equivalence $\inv_*$ of $\hc_H^\heartsuit$, which satisfies
\[
\inv_*(M \otimes_{U(\fh)} N) \simeq (\inv_* N) \otimes_{U(\fh)} (\inv_* M)
\]
canonically for $M,N \in \hc_H^\heartsuit$. It is clear that $\inv_*$ stabilizes $\hc_H^{\heartsuit, \mathrm{fg}}$, and that we have a canonical isomorphism
\[
\inv_* \circ \ind_{\hc}^H \simeq \ind_{\hc}^H.
\]

\sss
\label{sss:HC-equiv-modules-abelian}
The algebra $U(\fh)$, endowed with the adjoint $H$-action, is naturally an algebra object in the abelian category $\Rep^\heartsuit(H)$.
One can therefore consider the abelian category $\LMod_{U(\fh)}(\Rep^{\heartsuit}(H))$ of left $U(\fh)$-modules in $\Rep^\heartsuit(H)$, i.e.~$H$-equivariant $U(\fh)$-modules. It is a standard (and easy) fact that forgetting the right action of $U(\fh)$ induces an equivalence of abelian categories
\begin{equation}
\label{eqn:HC-equiv-modules-abelian}
\hc_H^\heartsuit \simeq \LMod_{U(\fh)}(\Rep^{\heartsuit}(H)),
\end{equation}
see e.g.~\cite[Equation~(3.5)]{br-Hecke}. Note that $\LMod_{U(\fh)}(\Rep^{\heartsuit}(H))$ is a Grothendieck abelian category. (Here a generator can be obtained from a generator $M$ of $\Rep^{\heartsuit}(H)$ by considering the $H$-equivariant $U(\fh)$-module $U(\fh) \otimes_k M$.) Hence $\hc_H^\heartsuit$ is also a Grothendieck abelian category.

 \sss 
We will denote by $\sH_H^\circ$ the full additive monoidal 
subcategory of $\hc_H^{\heartsuit, \mathrm{fg}}$ spanned by objects of the form $\ind_{\hc}^H(V)$ with $V \in \Rep^{\heartsuit,\mathrm{fd}}(H)$. 
We will also denote by $\sH_H$ the full additive monoidal 
subcategory of $\hc_H^{\heartsuit, \mathrm{fg}}$ spanned by objects which are finite extensions of objects in $\sH_H^\circ$. 

%---------------------------------------------------
\subsection{Duality}
%---------------------------------------------------

\sss
\label{sss:duals-RepH-HHcirc}
Recall that
in the symmetric monoidal category $\Rep^{\heartsuit,\mathrm{fd}}(H)$, each object $V$ is dualizable, with dual given by the contragredient representation $V^*$, see~\S\ref{sss:rigidity-Rep}.
Since monoidal functors send duals to duals (see~\cite[Exercise~2.10.6]{egno}), we deduce that in $\sH_H^\circ$ every object is left and right dualizable; more specifically, for $V \in \Rep^{\heartsuit,\fd}(H)$ the object $\ind_{\hc}^H(V)$ is both left and right dualizable in $\hc_H^\heartsuit$, with
\[
(\ind_{\hc}^H(V))^\vee \simeq \ind_{\hc}^H(V^*) \simeq {}^\vee (\ind_{\hc}^H(V)).
\]
Using the identifications discussed in~\S\ref{sss:diagonally-induced-hc-bimodules}, the evaluation morphisms
\[
V \otimes_k V^* \otimes_k U(\fh) \to U(\fh), \quad V^* \otimes_k V \otimes_k U(\fh) \to U(\fh)
\]
are induced by the evaluation morphism for $V$, and the coevaluation morphisms
\[
k \to V^* \otimes_k V \otimes_k U(\fh), \quad k \to V \otimes_k V^* \otimes_k U(\fh)
\]
are induced by the coevaluation morphism for $V$. In particular, $\sH_H^\circ$ is a rigid monoidal category in the sense of~\cite[Definition~2.10.11]{egno}.

\sss
We now study dualizability of more general objects of $\hc_H^{\heartsuit}$. Before this, recall that in the monoidal category of $U(\fh)$-bimodules, an object $M$ is left dualizable if and only if it is finitely generated projective as a right $U(\fh)$-module, see e.g.~\cite[Exercise~2.10.16]{egno}.\footnote{Note that this statement has a typo: the words left/right should be switched on one side of the claimed equivalence.} Explicitly, if $M$ satisfies this property then the left dual ${}^\vee M$ is the vector space $\Hom_r(M,U(\fh))$ of morphisms of right $U(\fh)$-modules from $M$ to $U(\fh)$, endowed with the bimodule structure determined by the formula
\[
(x \cdot f \cdot y)(m) = x \cdot f(y \cdot m)
\]
for $x,y \in U(\fh)$, $f \in \Hom_r(M,U(\fh))$ and $m \in M$. The evaluation morphism
\[
\Hom_r(M,U(\fh)) \otimes_{U(\fh)} M \to U(\fh)
\]
is defined by $f \otimes m \mapsto f(m)$. Under our assumption, the natural morphism
\[
M \otimes_{U(\fh)} \Hom_r(M,U(\fh)) \to \Hom_r(M,M)
\]
(where the right-hand side is the space of endomorphisms of right $U(\fh)$-modules of $M$) is an isomorphism, and the coevaluation morphism $U(\fh) \to M \otimes_{U(\fh)} \Hom_r(M,U(\fh))$ is defined, via this identification, by the assignment $x \mapsto (m \mapsto xm)$.

Similarly, a $U(\fh)$-bimodule $M$ is right dualizable if and only if it is finitely generated projective as a left $U(\fh)$-module, and in this case its right dual is the vector space $\Hom_l(M,U(\fh))$ of morphisms of left $U(\fh)$-modules from $M$ to $U(\fh)$, endowed with the natural $U(\fh)$-bimodule structure.

\begin{Lem}
An object $M \in \hc_H^{\heartsuit}$ is left, resp.~right, dualizable if and only if it is left, resp.~right, dualizable as a $U(\fh)$-bimodule. In this case its left, resp.~right, dual in $\hc_H^{\heartsuit}$ is its left, resp.~right, dual as a $U(\fh)$-bimodule, endowed with the natural $H$-action.
\end{Lem}

\begin{proof}
To fix notation we consider left dualizability. Since the forgetful functor from $\hc_H^{\heartsuit}$ to $U(\fh)$-bimodules is monoidal, it sends left dualizable objects to left dualizable objects. Conversely, if $M$ is left dualizable as a $U(\fh)$-bimodule, then as explained above its left dual is $\Hom_r(M,U(\fh))$. Since $M$ is finitely generated as a right $U(\fh)$-module, this space admits a natural structure of $H$-module (see e.g.~\cite[Lemma~3.8(ii)]{br-Hecke}) which makes it a Harish-Chandra bimodule, and the evaluation and coevaluation morphisms described above are clearly $H$-equivariant. Hence $M$ is left dualizable in $\hc_H^{\heartsuit}$, which finishes the proof.
\end{proof}

\sss
Let $n=\dim(\fh)$ and $K_\fh = \wedge^n(\fh^*) \otimes_k U(\fh)$, viewed as a $U(\fh)$-bimodule as in~\S\ref{sss:diagonally-induced-hc-bimodules}. (Here $\wedge^n(\fh^*)$ is $1$-dimensional; the corresponding character of $\fh$ is the opposite of the ``modular character'' given by $x \mapsto \mathrm{tr}(\mathrm{ad}(x))$.) Note that this bimodule is invertible; its inverse will be denoted $K_\fh^{-1}$.
Recall that for any $U(\fh)$-bimodule $M$ and any $i \in \Z$ there exists a canonical isomorphism
\begin{equation}
\label{eqn:isom-HH-bimodules}
\Ext^i_{U(\fh) \otimes_k U(\fh)^\rev}(U(\fh), M) \simeq \mathrm{Tor}^{U(\fh) \otimes_k U(\fh)^\rev}_{n-i}(U(\fh), K_\fh \otimes_{U(\fh)} M).
\end{equation}
In fact this identification is the special case for enveloping algebras of van den Bergh duality for Hochschild homology and cohomology, see e.g.~\cite[\S 6.3]{bz}; explicitly it can be obtained by computing both sides using a Chevalley--Eilenberg complex, and using the canonical identification $(\wedge^i \fh)^* \simeq \wedge^{n-i} \fh \otimes_k \wedge^n(\fh^*)$. Now assume that $M$ is left dualizable, and that ${}^\vee M$ is finitely generated projective as a right $U(\fh)$-module. Then ${}^\vee M$ is left dualizable, and using the fact that this bimodule is automatically projective also as a left $U(\fh)$-module, one sees that the functor ${}^\vee M \otimes_{U(\fh)} (-)$ is left adjoint to $M \otimes_{U(\fh)} (-)$ on the \emph{derived} category of $U(\fh)$-bimodules, and similarly the functor $(-) \otimes_{U(\fh)} {}^{\vee \vee} M$ is right adjoint to $(-) \otimes_{U(\fh)} {}^\vee M$ on the derived category of $U(\fh)$-bimodules. Using these properties and applying the isomorphism~\eqref{eqn:isom-HH-bimodules} for $i=n$ and the bimodules $(M \otimes_{U(\fh)} K_\fh^{-1}) \otimes_k U(\fh)$ and $K_\fh^{-1} \otimes_k {}^{\vee\vee} M$ (with, in each case, left action on the left factor and right action on the right factor), we obtain a canonical isomorphism of $U(\fh)$-bimodules
\[
{}^{\vee\vee} M \simeq K_\fh \otimes_{U(\fh)} M \otimes_{U(\fh)} K_\fh^{-1}.
\]

In case $M$ is a (finitely generated) Harish-Chandra bimodule, then this identification is an isomorphism of Harish-Chandra bimodules, where now $K_\fh$ is seen as the Harish-Chandra bimodule $\ind_{\hc}^H(\wedge^n(\fh^*))$.

\sss
\label{sss:pivotality-HH}
The objects of $\sH_H$ are finitely generated and free both as left and as right $U(\fh)$-modules. The comments above therefore imply that these objects are left and right dualizable in $\hc_H^{\heartsuit, \mathrm{fg}}$. Moreover, the functors $(-)^\vee$ and ${}^\vee (-)$ both stabilize $\sH_H$, so that this monoidal category is rigid, and for any $M \in \sH_H$ we have a canonical isomorphism
\[
{}^{\vee\vee} M \simeq K_\fh \otimes_{U(\fh)} M \otimes_{U(\fh)} K_\fh^{-1}.
\]
In particular, $\sH_H$ is pivotal (in the sense of~\cite[Definition~4.7.8]{egno}) in case the $H$-module $\wedge^n(\fh^*)$ is trivial. (Note that any choice of isomorphism of $H$-modules $\wedge^n(\fh^*) \simeq k$ determines isomorphisms $K_\fh \simeq U(\fh)$ and $K_\fh^{-1} \simeq U(\fh)$, which provides a pivotal structure which is independent of the initial choice of isomorphism.)

%%---------------------------------------------------
%%---------------------------------------------------

%---------------------------------------------------
\subsection{Study of \texorpdfstring{$H$}{H}-equivariant \texorpdfstring{$U(\fh)$}{Uh}-modules}
%---------------------------------------------------

\sss
In this subsection we prove a technical statement about resolutions of $H$-equivariant $U(\fh)$-modules that will be required below.
We will say that an $H$-equivariant $U(\fh)$-module is \emph{$H$-equivariantly free} if it admits a finite filtration whose subquotients are of the form $U(\fh) \otimes_k V$ for $V \in \Rep^{\heartsuit,\fd}(H)$. (In particular, such a module is automatically finitely generated over $U(\fh)$.) 

\begin{Prop}
\label{prop:resolution-hc}
Let $M$ be an $H$-equivariant $U(\fh)$-module which is finitely generated as a $U(\fh)$-module, and let $n = \dim (\fh)$. There exists an exact sequence of $H$-equivariant $U(\fh)$-modules
\[
0 \to P^n \to P^{n-1} \to \cdots \to P^0 \to M \to 0
\]
where each $P^i$ is an $H$-equivariantly free $U(\fh)$-module.
\end{Prop}

\sss
In order to prove Proposition~\ref{prop:resolution-hc} we will introduce filtrations, which will allow us to take advantage of the nice behaviour of modules over polynomial rings. More specifically,
recall that $U(\fh)$ admits a $\Z$-filtration whose associated graded ring is the symmetric algebra $S(\fh)$, where the grading on the latter is such that elements of $\fh$ are in degree $1$.  If $M$ is a filtered ($H$-equivariant) $U(\fh)$-module, then the associated graded $\gr(M)$ is a graded ($H$-equivariant) $S(\fh)$-module. (When considering a filtered module $M$, we will always assume that the filtration satisfies $M_i = 0$ for $i \ll 0$ and $M = \bigcup_{i \in \Z} M_i$.)

We will say that a graded $H$-equivariant $S(\fh)$-module is \emph{$H$-equivariantly free} if it admits a finite filtration whose subquotients are of the form $S(\fh) \otimes_k V\langle i\rangle$ with $V \in \Rep^{\heartsuit,\fd}(H)$ and $i \in \Z$, where $\langle i\rangle$ denotes a shift-of-grading. (We will use the convention that $(M \langle i \rangle)_m=M_{m-i}$, and use the same notation for shift of filtrations.)
As usual we will say that
a (nonequivariant) graded $S(\fh)$-module is \emph{free} if it is a direct sum of modules of the form $S(\fh)\langle i\rangle$.

\sss
We start by explaining the relation between the conditions of being equivariantly free and free for $H$-equivariant graded $S(\fh)$-modules.

\begin{Lem}
\label{lem:sh-free-criterion}
Let $M$ be a 
graded $H$-equivariant $S(\fh)$-module which is finitely generated as an $S(\fh)$-module. Then $M$ is $H$-equivariantly free if and only if it is free as a nonequivariant graded $S(\fh)$-module.
\end{Lem}

\begin{proof}
Of course, if $M$ is $H$-equivariantly free it is free as a nonequivariant graded $S(\fh)$-module.
We will prove the reverse implication by induction on the rank of $M$. We therefore assume that $M$ is free as a nonequivariant graded $S(\fh)$-module, and write
$M = \bigoplus_{i \in \Z} M_i$ for the decomposition given by the $\Z$-grading.
Since $M$ is finitely generated, there is a smallest integer $d$ such that $M_d \ne 0$; moreover, $M_d$ is finite-dimensional, and stable under the action of $H$, and thus defines an object of $\Rep^{\heartsuit,\fd}(H)$. Consider the ($H$-equivariant) map
\[
S(\fh) \otimes_k M_d \langle d \rangle \to M
\]
given by the action of $S(\fh)$ on $M$.
Comparison with the nonequivariant setting shows that this map is injective, and that its cokernel is free as a nonequivariant graded module, of lower rank than $M$.  The lemma follows.
\end{proof}

\sss
We next study the relation between the conditions of being equivariantly free over $U(\fh)$ or $S(\fh)$.

\begin{Lem}
\label{lem:uh-free-graded}
Let $M$ be a filtered $H$-equivariant $U(\fh)$-module.  If $\gr(M)$ is $H$-equivariantly free as a graded $H$-equivariant $S(\fh)$-module, then $M$ is $H$-equivariantly free.
\end{Lem}

\begin{proof}
We proceed by induction on the rank of $\gr(M)$ as a (nonequivariant) free $S(\fh)$-module. Write $M = \bigcup_{i \in \Z} M_i$ for the given filtration on $M$, and note that $M$ is finitely generated over $U(\fh)$ since $\gr(M)$ is finitely generated over $S(\fh)$. Let $d$ be the smallest integer such that $M_d \ne 0$, and consider the filtered map
\begin{equation}
\label{eqn:uh-free1}
U(\fh) \otimes_k M_d \langle d \rangle \to M
\end{equation}
induced by the action of $U(\fh)$ on $M$.
Applying $\gr$ to this construction yields a map
\begin{equation}
\label{eqn:uh-free2}
S(\fh) \otimes_k \gr(M)_d \langle d \rangle \to \gr(M).
\end{equation}
By the argument in the proof of Lemma~\ref{lem:sh-free-criterion}, this map is injective.  By, say,~\cite[Corollary~7.6.14]{mr},~\eqref{eqn:uh-free1} is therefore injective and strict with respect to the filtrations.

Let $M'$ be the cokernel of~\eqref{eqn:uh-free1}, equipped with the filtration induced by that of $M$.  The quotient map $M \to M'$ is then strict with respect to the filtrations.  By~\cite[Proposition~7.6.13 and Corollary~7.6.14]{mr}, the induced map $\gr(M) \to \gr(M')$ is surjective, and identifies with the cokernel of~\eqref{eqn:uh-free2}. In particular, as in the proof of Lemma~\ref{lem:sh-free-criterion} again, $\gr (M')$ is free as a nonequivariant graded $S(\fh)$-module, of smaller rank than $\gr (M)$. By Lemma~\ref{lem:sh-free-criterion}, $\gr(M')$ is therefore $H$-equivariantly free as a graded $H$-equivariant $S(\fh)$-module. By induction, we deduce that $M'$ is $H$-equivariantly free, and then that so is $M$.
\end{proof}

\sss
The next lemma will be the key step for the proof of Proposition~\ref{prop:resolution-hc}.

\begin{Lem}
\label{lem:lift-resoln}
Let $M$ be a filtered $H$-equivariant $U(\fh)$-module, and assume that $\gr(M)$ is finitely generated over $S(\fh)$.  There exists a short exact sequence of filtered $H$-equivariant $U(\fh)$-modules
\[
Q \hookrightarrow P \twoheadrightarrow M
\]
such that:
\begin{enumerate}
\item
the induced sequence $\gr(Q) \to \gr(P) \to \gr(M)$ is a short exact sequence;
\item
$\gr(P)$ and $\gr(Q)$ are finitely generated over $S(\fh)$;
\item
$P$ is $H$-equivariantly free as an $H$-equivariant $U(\fh)$-module.
\end{enumerate}
\end{Lem}

\begin{proof}
Write the filtration on $M$ as $M = \bigcup_{i \in \Z} M_i$.  Since $\gr(M)$ is finitely generated, so is $M$; hence each $M_i$ is finite-dimensional, and there exists an integer $N$ such that $\gr (M_N) = \bigoplus_{i \le N} \gr(M)_i$ generates $\gr(M)$ as an $S(\fh)$-module.  Consider the filtered map
\[
\phi: P:= U(\fh) \otimes_k M_N \to M
\]
induced by the action of $U(\fh)$ on $M$.
Then the map
\[
\gr (\phi) : S(\fh) \otimes_k \gr (M_N) \to \gr (M)
\]
is surjective, so $\phi$ is surjective and strict by~\cite[Corollary~7.6.14]{mr}. It is clear that $P$ is $H$-equivariantly free as an $H$-equivariant $U(\fh)$-module. Denoting by $Q$ the kernel of this surjection, and equipping $Q$ with the filtration induced by that of $P$, the embedding $Q \hookrightarrow P$ is strict. Using~\cite[Proposition~7.6.13 and Corollary~7.6.14]{mr} we deduce that $\gr(Q) \to \gr(P) \to \gr(M)$ is a short exact sequence, which finishes the proof.
\end{proof}

\sss
\label{sss:proof-prop-resolution-hc}
Finally, we are in a position to give the proof of Proposition~\ref{prop:resolution-hc}.

\begin{proof}[Proof of Proposition~\ref{prop:resolution-hc}]
Since $M$ is finitely generated, we can (and shall) equip it with a filtration compatible with the $H$-action and such that $\gr (M)$ is a finitely generated $S(\fh)$-module. 
Using Lemma~\ref{lem:lift-resoln} repeatedly, we obtain 
a resolution as $H$-equivariant $U(\fh)$-module
\[
P^{n-1} \xrightarrow{d^{n-1}} \cdots \xrightarrow{d^1} P^0 \xrightarrow{d^0} M \to 0
\]
where each $P^i$ is filtered with finitely generated associated graded, each $d^i$ is filtered, the associated graded is again a resolution, and moreover each $P^i$ is $H$-equivariantly free as an $H$-equivariant $U(\fh)$-module.

For the last step, we set $P^n := \ker (d^{n-1})$, and equip it with the filtration induced by that of $P^{n-1}$. 
Then as above the sequence
\[
0 \to \gr (P^n) \to \gr (P^{n-1}) \to \cdots \to \gr (P^0) \to \gr(M) \to 0
\]
is an exact sequence of graded $H$-equivariant $S(\fh)$-modules. Since $S(\fh)$ is isomorphic to a polynomial ring in $n$ variables, the graded version of Hilbert's syzygy theorem implies that $\gr (P^n)$ is free as a nonequivariant graded $S(\fh)$-module. By Lemma~\ref{lem:sh-free-criterion}, $\gr (P^n)$ is therefore $H$-equivariantly free, and then $P^n$ is $H$-equivariantly free as an $H$-equivariant $U(\fh)$-module by Lemma~\ref{lem:uh-free-graded}, which finishes the proof.
\end{proof}

%---------------------------------------------------
\subsection{The monoidal presentable stable \texorpdfstring{$\infty$}{infinity}-category of Harish-Chandra bimodules}
%---------------------------------------------------

\sss 
Applying the construction of~\S\ref{sss:algebra-structure-Chb} to the additive monoidal category $\sH_H$ we obtain the
monoidal $\infty$-category $\Chb(\sH_H)$, in which every object is left and right dualizable, and which comes equipped with 
a monoidal functor
\[
\ind_{\hc}^H : \Chb(\Rep^{\heartsuit,\fd}(H)) \rightarrow \Chb(\sH_H),
\]
and an involutive equivalence
\[
\inv_*: \Chb(\sH_H) \simto \Chb(\sH_H)^{\rev}.
\]

\sss 
Let $\on{Acy} \subset \mathrm{Ho}(\Chb(\sH_H))$ be the full subcategory consisting of acyclic complexes. (Here, ``acyclic'' means with trivial cohomology when seen as a complex of objects in the abelian category $\hc_H^\heartsuit$.)
Using the fact that objects of $\sH_H$ are flat both as left and right $U(\fh)$-modules, one sees
that $\on{Acy}$ is a two-sided ideal in $\mathrm{Ho}(\Chb(\sH_H))$.
As a consequence, following the discussion in~\S\ref{sss:algebra-structure-Verdier-quotient}, if we define
\[
\hc_{H}^\comp
\]
as 
the Verdier quotient of $\Chb(\sH_H)$ by the full subcategory spanned by objects which belong to $\on{Acy}$,
we obtain a monoidal $\infty$-category in which every object is left and right dualizable.
Moreover, the composite monoidal functor
\[
\Chb(\Rep^{\heartsuit, \fd}(H)) \rightarrow \Chb(\sH_H) \rightarrow \hc_{H}^\comp
\]
factors through 
a monoidal functor 
\begin{equation}
\label{eqn:indc}
\Rep(H)^\comp \rightarrow \hc_{H}^\comp
\end{equation}
(where we use the comments in~\S\ref{sss:Repc-DbRepfg}).
We similarly obtain an involutive equivalence 
\[
\inv_*: \hc_{H}^\comp \simeq (\hc_{H}^{\comp})^{\rev}
\]
induced by $\inv_*$.

\sss
\label{sss:def-hc} 
With these preparations, we define the $\infty$-category of Harish-Chandra bimodules by setting
\[
\hc_H := \Ind(\hc_H^\comp).
\]
As we will justify in~\S\ref{sss:hcc-idempotent-complete} below, $\hc_H^\comp$ is idempotent-complete, so
$\hc_H^\comp$ identifies with the $\infty$-category of compact objects in $\hc_H$, see~\cite[Chap.~1, Lemma~7.2.4]{gr}, which justifies our choice of notation. By the considerations in~\S\ref{sss:algebra-structure-Verdier-quotient}, this $\infty$-category admits a canonical structure of algebra object in $\DGCat$.

By~\cite[Chap.~1, Lemma~9.1.5]{gr},
$\hc_{H}$ is a rigid monoidal $\infty$-category in the sense of~\cite[Chap.~1, Definition~9.1.2]{gr}; in fact it is a rigid $k$-linear monoidal $\infty$-category in the sense of~\cite[Example~7.88]{zhu}. This $\infty$-category comes equipped with an algebra morphism
\begin{equation}
\label{eqn:indhc}
\ind_{\hc}^H : \Rep(H) \rightarrow \hc_H
\end{equation}
(obtained by Ind-completing the monoidal functor~\eqref{eqn:indc})
and an involutive identification 
\begin{equation}
\label{eqn:inv-hc}
\inv_*: \hc_{H} \simeq (\hc_{H})^\rev.
\end{equation}

\sss
\label{sss:twist-pivotality}
The discussion in~\S\ref{sss:pivotality-HH} shows that the automorphism $\varphi_{\hc_H}$ of~\cite[Chap.~1, Lemma~9.2.8]{gr} is given by $\varphi_{\hc_H}(M) = K_\fh \otimes_{U(\fh)} M \otimes_{U(\fh)} K_{\fh}^{-1}$. In case the $H$-module $\wedge^n(\fh^*)$ is trivial, $\hc_H$ is pivotal, and this automorphism is the identity. In general it is nontrivial, but as it is given by conjugation by an invertible object, twisting a $\hc_H$-module by this automorphism will not change the module up to isomorphism. In particular, the twist appearing in~\cite[Chap.~1, Proposition~9.5.3]{gr} can be ignored for this rigid monoidal $\infty$-category.

%---------------------------------------------------
\subsection{Comparison with the derived \texorpdfstring{$\infty$-}{infinity-}category of the abelian category of Harish-Chandra bimodules and applications}
%---------------------------------------------------

\sss
As explained in~\S\ref{sss:HC-equiv-modules-abelian}, $\hc_H^\heartsuit$ is a Grothendieck abelian category, so 
we can consider the derived $\infty$-category $D(\hc_H^\heartsuit)$, which is a presentable stable $\infty$-category, see~\S\ref{sss:derived-cat}. We also consider the bounded derived $\infty$-category $\Db(\hc_H^{\heartsuit, \mathrm{fg}})$ of $\hc_H^{\heartsuit, \mathrm{fg}}$, see~\S\ref{sss:derived-cat}. As in Remark~\ref{rmk:DAcoh-DbAnoeth} we have a canonical fully faithful functor $\Db(\hc_H^{\heartsuit, \mathrm{fg}}) \to D(\hc_H^{\heartsuit})$.

From the construction we also have a canonical functor
\begin{equation}
\label{eqn:functor-hcc-Dhc}
\hc_H^\comp \to \Db(\hc_H^{\heartsuit,\mathrm{fg}}).
\end{equation}

\begin{Lem}
\label{lem:HCc-DbHC}
The functor~\eqref{eqn:functor-hcc-Dhc} is an equivalence of $\infty$-categories.
\end{Lem}

\begin{proof}
Using the equivalence~\eqref{eqn:HC-equiv-modules-abelian}, one can interpret Proposition~\ref{prop:resolution-hc} as saying that any object of $\hc_H^{\heartsuit,\mathrm{fg}}$ admits a bounded resolution by objects of $\sH_H$. It follows that for any bounded complex $X$ of objects of $\hc_H^{\heartsuit,\mathrm{fg}}$ there exists a bounded complex $Y$ of objects of $\sH_H$ and a quasi-isomorphism $Y \to X$. Then the claim follows from~\cite[\href{https://stacks.math.columbia.edu/tag/0GSL}{Tag 0GSL}]{stacks-project}.
\end{proof}

\sss
\label{sss:hcc-idempotent-complete}
Note that
since $\mathrm{Ho}(\Db(\hc_H^{\heartsuit,\mathrm{fg}}))$ is idempotent-complete by~\cite{lc}, and since idempotent-completeness of stable $\infty$-categories can be checked at the level of homotopy categories (see~\cite[Proposition~2.2.3]{jasso}), Lemma~\ref{lem:HCc-DbHC} implies in particular that $\hc_H^\comp$ is idempotent-complete.

\sss
\label{sss:hc-reductive-gp}
Let us
assume that $H$ is a split reductive group scheme over $k$. Recall the considerations of~\S\ref{sss:Repc-tilting}, and assume furthermore that one of the following conditions holds:
\begin{itemize}
\item
$\mathrm{char}(k)$ is good for $H$, $H$ has simply connected derived subgroup, and $\fh$ admits a nondegenerate $H$-invariant bilinear form;
\item
$\mathrm{char}(k)$ is very good for $H$. 
\end{itemize}
Then the proof of~\cite[Lemma~2.9]{br} shows that $U(\fh)$ admits a good filtration as an $H$-module, in the sense of~\cite[Part II, \S 4.16]{jantzen}. (More specifically, the proof applies in the first case. In the second case, one can choose a finite central isogeny $H' \to H$ such that $H'$ has a simply connected derived subgroup and the induced morphism $\fh' \to \fh$ is an isomorphism; then $H'$ satisfies the first set of assumptions, and one deduces the claim for $H$.) Using this property and arguing as in~\cite[\S 2.5]{br} one sees that,
if we denote by $\sT_H$ the full subcategory of $\sH_H^\circ$ 
spanned by objects of the form $\ind_{\hc}^H(V)$ where $V$ is a finite-dimensional tilting $H$-module,
the natural functor
\[
\Chb(\sT_H) \rightarrow \Db(\hc_H^{\heartsuit,\mathrm{fg}})
\]
is fully faithful. By the results of~\S\ref{sss:Repc-tilting}, the essential image of this functor 
contains all objects of the form $\ind_{\hc}^H(V)$ with $V \in \Rep^{\heartsuit, \fd}(H)$, so it is essentially surjective by Proposition~\ref{prop:resolution-hc}. In this case we therefore obtain equivalences
\[
\Chb(\sT_H) \simto \hc_H^\comp, \quad
\Ind(\Chb(\sT_H)) \simto \hc_H.
\]
Since a tensor product of tilting modules is tilting (see~\cite[\S E.7]{jantzen}),
the subcategory $\sT_H \subset \sH_H$ is monoidal, so that, by the considerations in~\S\ref{sss:algebra-structure-Chb}, $\Ind(\Chb(\sT_H))$ has a canonical structure of algebra object in $\DGCat$. It is easily seen that the equivalence above is an equivalence of algebra objects.

\sss
\label{sss:hcc-Dbhcfg}
We now come back to the setting where $H$ is a general affine group scheme of finite type over $k$.
The $\infty$-category $\Db(\hc_H^{\heartsuit,\mathrm{fg}})$ admits a canonical t-structure. Transferring this t-structure through the equivalence of Lemma~\ref{lem:HCc-DbHC}, we deduce a t-structure on $\hc_H^\comp$. It is easily seen that the binary product on $\hc_H^\comp$ given by the monoidal structure is right t-exact with respect to this t-structure, in the sense that it sends $(\hc_H^\comp)^{\leq 0} \times (\hc_H^\comp)^{\leq 0}$ into $(\hc_H^\comp)^{\leq 0}$. By the natural procedure (see e.g.~\cite[Chap.~4, Lemma~1.2.4]{gr}), from this t-structure we deduce a t-structure on $\hc_H$ which is compatible with filtered colimits, and which restricts to the given t-structure on $\hc_H^\comp$. Here again, the binary product on $\hc_H$ given by the monoidal structure is right t-exact with respect to this t-structure.

\sss
\label{sss:hc-LMod-D}
Consider now the functor
\[
D(\hc_H^\heartsuit) \to D(\Rep^\heartsuit(H))
\]
induced by the exact functor $\hc_H^\heartsuit \to \Rep^\heartsuit(H)$ given by forgetting the actions of $U(\fh)$.
This functor is continuous, conservative, and it admits a left adjoint (induced by~\eqref{eqn:induction-Rep-hc}).
In view of the equivalence~\eqref{eqn:HC-equiv-modules-abelian}, the corresponding monad on $D(\Rep^\heartsuit(H))$ is given by $U(\fh) \otimes_k (-)$. Hence the Barr--Beck--Lurie theorem (in the form of~\cite[Chap.~1, Proposition~3.7.7]{gr}) implies that we have an equivalence of $\infty$-categories
\[
D(\hc_H^\heartsuit) \simto \LMod_{U(\fh)}(D(\Rep^\heartsuit(H))).
\]
It is clear that this equivalence identifies the natural t-structure on the left-hand side with the t-structure on the right-hand side provided by Lemma~\ref{lem:t-structures} (applied to the natural t-structure on $D(\Rep^\heartsuit(H))$). Hence this equivalence restricts to an equivalence
\begin{equation}
\label{eqn:Dhc-modules-DRep}
D(\hc_H^\heartsuit)^+ \simto \LMod_{U(\fh)}(D(\Rep^\heartsuit(H))^+).
\end{equation}

\sss
\label{sss:HC-ren}
Recall the description of the renormalization of the derived $\infty$-category of a Grothendieck abelian category given in Proposition~\ref{prop:renormalization-DA} and~\eqref{eqn:renormalization-DA-IndDbA}. In view of Lemma~\ref{lem:HCc-DbHC}, $\hc_H$ is obtained from $D(\hc_H^{\heartsuit})$ using this procedure. By Remark~\ref{rmk:renormalization-DA}\eqref{it:renormalization-bounded-below-parts}, it follows that we have a canonical t-exact equivalence
\begin{equation}
\label{eqn:hc+-Dhc}
(\hc_H)^+ \simeq D(\hc_H^{\heartsuit})^+.
\end{equation}
In particular, the heart of the t-structure on $\hc_H$ identifies with $\hc_H^\heartsuit$, which justifies the notation. This analysis also provides an identification of $\hc_H$ with the $\infty$-category of chain complexes of injective objects in $\hc_H^\heartsuit$.

Note that the renormalization procedure only depends on the bounded below part of the t-structure (and its induced t-structure). In view of~\eqref{eqn:Dhc-modules-DRep} and the discussion in~\S\S\ref{sss:QCoh-BH}--\ref{sss:def-RepH}, one can therefore also obtain $\hc_H$ by renormalization from the $\infty$-categories $\LMod_{U(\fh)}(\QCoh(\BB H))$ or $\LMod_{U(\fh)}(\Rep(H))$ and their t-structures given by Lemma~\ref{lem:t-structures}.

\sss
In the following statement we use the notation introduced in~\S\ref{sss:def-AmodH}. 

\begin{Prop} 
\label{prop:hchind} 
The functor $\ind^H_\hc$ of~\eqref{eqn:indhc} admits a conservative lax monoidal right adjoint 
\[
\Rep(H) \leftarrow \hc_H: (\ind^H_\hc)^R.
\]
Moreover, the resulting monad on $\Rep(H)$ yields a canonical identification
\[
\hc_H \simeq U(\fh)\mod_H.
\]
\end{Prop}

\begin{proof} 
The functor $\ind^H_{\hc}$ admits a (not necessarily continuous) right adjoint $(\ind^H_{\hc})^R$ by the adjoint functor theorem. This adjoint is in fact continuous because $\ind^H_{\hc}$ preserves compact objects, see~\cite[Chap.~1, Lemma~7.1.5]{gr}. The essential image of $\ind^H_{\hc}$ generates $\hc_H$ by~\cite[Chap.~1, Proposition~5.4.5]{gr}, and hence $(\ind^H_{\hc})^R$ is conservative by~\cite[Chap.~1, Lemma~5.4.3]{gr}. This functor is also lax monoidal as the right adjoint of a monoidal functor, see~\cite[Chap.~1, Lemma~3.5.3]{gr}.

By the Barr--Beck--Lurie theorem~\cite[Chap.~1, Proposition~3.7.7]{gr}, the $\infty$-category $\hc_H$ identifies with modules for the monad $(\ind^H_{\hc})^R \circ \ind^H_{\hc}$ on $\Rep(H)$. If we regard $\hc_H$ as a right $\Rep(H)$-module under multiplication, i.e.~with underlying binary product
\[
\hc_H \otimes_k \Rep(H) \xrightarrow{\id \otimes_k \ind^H_{\hc}} \hc_H \otimes_k \hc_H \xrightarrow{\on{mult}} \hc_H,
\]
then $\ind^H_{\hc}$ is tautologically $\Rep(H)$-linear. By Corollary~\ref{cor:rigidity-RepH}\eqref{it:adjoints-RepH-linear},
$(\ind^H_{\hc})^R$ is again $\Rep(H)$-linear, and in particular so is the monad $(\ind^H_\hc)^R \circ \ind^H_\hc$.  That is, the $\infty$-category of modules for the monad $(\ind^H_\hc)^R \circ \ind^H_\hc$ is identified with the $\infty$-category of modules for the algebra object $(\ind^H_\hc)^R \circ \ind^H_\hc(k)$. It therefore remains to identify the latter with $U(\fh)$. 

The functor $\ind^H_\hc$ is t-exact, so its right adjoint $(\ind^H_\hc)^R$ is left t-exact. As a consequence, $(\ind^H_\hc)^R \circ \ind^H_\hc(k)$ belongs to $\Rep(H)^{\geqslant 0}$, and a fortiori to $\Rep(H)^+$. Since the monoidal product on $\Rep(H)$ preserves $\Rep(H)^+$, to conclude one can therefore work with the restrictions of all functors to the ``$+$'' parts. Now we have equivalences
\begin{multline*}
(\hc_H)^+ \overset{\eqref{eqn:hc+-Dhc}}{\simeq} D(\hc_H^{\heartsuit})^+ \overset{\eqref{eqn:Dhc-modules-DRep}}{\simeq} \LMod_{U(\fh)}(D(\Rep^\heartsuit(H))^+) \\
\simeq \LMod_{U(\fh)}(\Rep(H)^+) \simeq \LMod_{U(\fh)}(\Rep(H))^+
\end{multline*}
(where the third equivalence uses the discussion in~\S\S\ref{sss:QCoh-BH}--\ref{sss:def-RepH}, so that the claim is clear.
\end{proof}

\begin{Rem}
The functor $(\ind^H_\hc)^R$ in Proposition~\ref{prop:hchind} can be explicitly constructed by an application of the universal property of the Ind-completion to the composition
\[
\hc_H^\comp \subset (\hc_H)^+ \simeq D(\hc_H^\heartsuit)^+ \to D(\Rep^\heartsuit(H))^+ \simeq \Rep(H)^+ \subset \Rep(H)
\]
where the middle arrow is as in~\S\ref{sss:hc-LMod-D}. In view of this description, this functor will sometimes be denoted $\res^H_{\hc}$.
\end{Rem}

%-------------------------------------------
\subsection{Harish-Chandra bimodules for products of groups}
%-------------------------------------------

We finish this section with a study of the behaviour of the construction of Harish-Chandra bimodules with respect to products of groups, whose proof will use the considerations above.

\begin{Lem}
\label{lem:HC-product}
Let $H_1$, $H_2$ be affine $k$-group schemes of finite type,
and assume that the functor
\[
\Rep(H_1) \otimes_k \Rep(H_2) \to \Rep(H_1 \times_k H_2)
\]
of Lemma~\ref{lem:Rep-product} is an equivalence. Then
we have a canonical equivalence of 
algebra objects in $\DGCat$
\[
\hc_{H_1} \otimes_k \hc_{H_2} \simeq \hc_{H_1 \times_k H_2}.
\]
\end{Lem}

\begin{proof}
Tensor product over $k$ induces a canonical functor
\[
\hc_{H_1}^\comp \times \hc_{H_2}^\comp \to \hc_{H_1 \times_k H_2}^\comp.
\]
Taking Ind-completions and observing that the resulting functor preserves colimits in each variable, we deduce a functor
\[
\hc_{H_1} \otimes \hc_{H_2} \to \hc_{H_1 \times_k H_2}
\]
and then, using the considerations of~\S\ref{sss:claim-relative-tensor-product}, a functor
\[
\hc_{H_1} \otimes_k \hc_{H_2} \to \hc_{H_1 \times_k H_2},
\]
which is clearly a morphism of algebra objects. To show that this functor is an equivalence we use the equivalences
\[
\hc_K \simeq U(\fk)\mod_K
\]
for $K$ being $H_1$, $H_2$ and $H_1 \times_k H_2$ (see Proposition~\ref{prop:hchind}), and~\cite[Chap.~1, Proposition~8.5.4]{gr}.
\end{proof}

%%%%%%%%%%%%%%%%%%%%%%%%%%%%%%
\section{Strong actions}
\label{sec:strong-actions}
%%%%%%%%%%%%%%%%%%%%%%%%%%%%%%

%---------------------------------------------------
\subsection{Definition and first examples}
%---------------------------------------------------

\sss 
Having made sense of the algebra object $\hc_H$, we may define left, resp.~right, strong categorical representations of $H$ as left, resp.~right, $\hc_H$-modules in $\DGCat$. We therefore have $\infty$-categories
\[
\hc_{H}\mod, \quad \modr\hc_H
\]
of left and right strong categorical representations of $H$. 

Note that we have an equivalence of $\infty$-categories
\[
\hc_{H}\mod \simto \modr\hc_H
\]
sending a left $\hc_H$-module to the same object of $\DGCat$, endowed with the right module structure obtained by twisting the given left action of $\hc_H$ by $\inv_*$ (see~\eqref{eqn:inv-hc}). We will therefore often state our results below only for left strong categorical representations, which we will simply call strong categorical representations.

\sss
\label{sss:dualizability-strong-reps}
One can consider the (left and right) dualizability of left and right strong categorical representations of $H$. In particular, by the discussion in~\S\ref{sss:rigidity-duals}, a left strong categorical representation $\sC$ is right dualizable if and only if it is dualizable in $\DGCat$, or in $\PrLSt$. Moreover, in this case, taking also into account the discussion in~\S\ref{sss:twist-pivotality}, the right dual of $\sC$ is the right categorical representation given by $\sC^\vee$ (the dual in $\DGCat$, or in $\PrLSt$), endowed with the right action of $\hc_H$ where an object $M$ acts via the natural right action of $K_\fh^{-1} \otimes_{U(\fh)} M \otimes_{U(\fh)} K_\fh$. (This module is isomorphic to $\sC^\vee$ with the natural right action.)

\sss
\label{sss:restriction-induction-weak-strong}
Recall the algebra morphism $\ind_{\hc}^H$, see~\eqref{eqn:indhc}. Using this morphism we can consider the ``restriction'' functor
\[
\hc_H\mod \to \Rep(H)\mod
\]
and its left adjoint
\[
\hc_H \otimes_{\Rep(H)} (-) : \Rep(H)\mod \to \hc_H\mod,
\]
which allow us to pass back and forth between weak and strong categorical representations of $H$.

\sss 
\label{sss:Amod-strong-reps}
The following class of examples is a counterpart for strong categorical representations of the weak categorical representations studied in~\S\ref{sss:def-AmodH}.
If $A$ is an algebra object in $\hc_H$, the $\infty$-category of left $A$-modules 
\[
\LMod_A(\hc_H)
\]
is naturally a right strong categorical representation of $H$, equipped with equivariant adjoint functors 
\[
\hc_H \rightleftarrows \LMod_A(\hc_H).
\]
Similar statements apply for right modules (which naturally form a left strong categorical representation).

Note that the lax monoidal ``forgetful'' functor $\hc_H \to \Rep(H)$ (see Proposition~\ref{prop:hchind}) sends algebra objects to algebra objects; an algebra object $A$ as above therefore determines an algebra object in $\Rep(H)$, denoted similarly.
Consider the adjunctions 
\[
\Rep(H) \rightleftarrows \hc_H \rightleftarrows \LMod_A(\hc_H).
\]
As both right adjoints in this diagram are conservative, we deduce by the Barr--Beck--Lurie theorem~\cite[Chap.~1, Proposition~3.7.7]{gr} an equivalence of $\infty$-categories
\[
\LMod_A(\hc_H) \simeq \LMod_{A}(\Rep(H)).
\]

\sss
\label{sss:Dmod-strong-action}
Let us emphasize a particular case of the discussion in~\S\ref{sss:Amod-strong-reps}. 
Suppose $X$ is a smooth affine algebraic variety over $k$ equipped with an action of $H$, and let $\sD(X)$ be its algebra of global differential operators. Then the action of $H$ on $X$ induces an action on $\sD(X)$ and an algebra morphism $U(\fh) \to \sD(X)$, so that $\sD(X)$ is naturally an algebra object in $\hc_H$, and we have an equivalence
\[
\LMod_{\sD(X)}(\hc_H) \simeq \LMod_{\sD(X)}(\Rep(H)).
\]
Moreover, 
by Lemma~\ref{lem:equiv-modules-Ind},
this $\infty$-category can be obtained by starting with the ``naive'' derived $\infty$-category $\LMod_{\sD(X)}(\QCoh(\BB H))$ of weakly $H$-equivariant D-modules on $X$, equipped with its localization functor
\[
\on{Loc}: \hc_H^\comp \rightarrow \LMod_{\sD(X)}(\QCoh(\BB H)),
\]
and Ind-completing the full stable subcategory generated by the essential image of $\on{Loc}$. 

%------------------------------------------------------------------------------
\subsection{Harish-Chandra modules}
\label{ss:hK-mod}
%------------------------------------------------------------------------------

\sss
We now discuss another source of interesting strong categorical representations. Consider an affine $k$-group scheme of finite type $H$, and a (closed) subgroup scheme $K \subset H$. Denote their respective Lie algebras by $\fh$ and $\fk$. Then we have the abelian category $(\fh,K)\mod^\heartsuit$ of $(\fh,K)$-modules, or strongly $K$-equivariant $U(\fh)$-modules, i.e.~$K$-equivariant $U(\fh)$-modules such that the differential of the action of $K$ coincides with the restriction of the action of $U(\fh)$ along the embedding $U(\fk) \hookrightarrow U(\fh)$. We have a canonical functor
\begin{equation}
\label{eqn:ind-(h,K)-mod}
\ind_{K}^{(\fh,K)} : \Rep^\heartsuit(K) \to (\fh,K)\mod^\heartsuit, \qquad \ind_{K}^{(\fh,K)}(V) = U(\fh) \otimes_{U(\fk)} V,
\end{equation}
which is left adjoint to the natural forgetful functor $(\fh,K)\mod^\heartsuit \to \Rep(K)^\heartsuit$. The image under this functor of a generator of $\Rep^\heartsuit(K)$ provides a generator of $(\fh,K)\mod^\heartsuit$, and this category is therefore a Grothendieck abelian category.

\sss
\label{sss:action-hcbim-hcmod}
Of course the Harish-Chandra bimodules for $H$ introduced in~\S\ref{sss:hc-bimodules-def} are a special case of this construction, for the antidiagonal embedding of $H$ in $H \times_k H^{\rev}$. (In this case we have $\ind_{\Delta H}^{(\fh \oplus \fh^\rev, \Delta H)} = \ind^H_\hc$.)

\sss
For general $H$, $K$, we have a natural left action of the monoidal abelian category $\hc_H^\heartsuit$ on the abelian category $(\fh,K)\mod^\heartsuit$, induced by the action of $U(\fh)$-bimodules on $U(\fh)$-modules. Note that for $V_1 \in \Rep^{\heartsuit}(H)$ and $V_2 \in \Rep^\heartsuit(K)$ we have a canonical isomorphism
\begin{equation}
\label{eqn:tensor-induced-hc}
\ind_{\hc}^H(V_1) \otimes_{U(\fh)} \ind_{K}^{(\fh,K)}(V_2) \simeq \ind_{K}^{(\fh,K)}(\res^H_K(V_1) \otimes_k V_2).
\end{equation}
Explicitly, identifying as in~\S\ref{sss:diagonally-induced-hc-bimodules} $\ind_{\hc}^H(V_1)$ with $V_1 \otimes_k U(\fh)$, the left-hand side identifies with $V_1 \otimes_k (U(\fh) \otimes_{U(\fk)} V_2)$ with the diagonal action of $U(\fh)$ (on the first two factors), which identifies with $U(\fh) \otimes_{U(\fk)} (\res^H_K(V_1) \otimes_k V_2)$ via the map $v \otimes (x \otimes w) \mapsto x_{(1)} \otimes S(x_{(2)}) \cdot v \otimes w$ (for $v \in V_1$, $x \in U(\fh)$ and $w \in V_2$) in Sweedler's notation.

\sss
We will denote by $(\fh,K)\mod^{\heartsuit,\mathrm{fg}}$ the full abelian subcategory of $(\fh,K)\mod^\heartsuit$ spanned by objects which are finitely generated as $U(\fh)$-modules. It is clear that this subcategory is stable under the action of $\hc_H^{\heartsuit, \mathrm{fg}}$, and that the functor~\eqref{eqn:ind-(h,K)-mod} restricts to a functor
\[
\Rep^{\heartsuit,\fd}(K) \to (\fh,K)\mod^{\heartsuit,\mathrm{fg}}.
\]

\sss
We set
\[
(\fh,K)\mod^\comp := \Db((\fh,K)\mod^{\heartsuit,\mathrm{fg}}).
\]
Then $(\fh,K)\mod^\comp$ is a small stable $\infty$-category, which is idempotent complete by the same considerations as in~\S\ref{sss:hcc-idempotent-complete}, and which is equipped with a canonical t-structure with heart $(\fh,K)\mod^{\heartsuit,\mathrm{fg}}$, and a t-exact functor
\[
\ind_K^{(\fh,K)} : \Rep(K)^\comp \to (\fh,K)\mod^\comp
\]
induced by the exact functor~\eqref{eqn:ind-(h,K)-mod}.

\sss
\label{sss:def-(h,K)mod}
We next set
\[
(\fh,K)\mod := \Ind((\fh,K)\mod^\comp).
\]
This $\infty$-category is naturally an object in $\DGCat$. The subcategory of compact objects in $(\fh,K)\mod$ identifies with $(\fh,K)\mod^\comp$, justifying the notation. We also have a canonical t-structure compatible with filtered colimits on $(\fh,K)\mod$, which restricts to the t-structure considered above on $(\fh,K)\mod^\comp$, and a t-exact functor
\[
\ind_K^{(\fh,K)} : \Rep(K) \to (\fh,K)\mod.
\]

Using the universal property of the Ind-completion we obtain a canonical t-exact functor
\begin{equation}
\label{eqn:functor-hKmod-D}
(\fh,K)\mod \to D((\fh,K)\mod^{\heartsuit}).
\end{equation}

\sss
\label{sss:hKmod-renormalization}
As in~\S\ref{sss:HC-ren}, using Proposition~\ref{prop:renormalization-DA} and~\eqref{eqn:renormalization-DA-IndDbA} one sees that the $\infty$-category $(\fh,K)\mod$ can be obtained from the derived $\infty$-category $D((\fh,K)\mod^{\heartsuit})$ by renormalization. As a consequence, this $\infty$-category identifies with the $\infty$-category of chain complexes of injective objects in $(\fh,K)\mod^{\heartsuit}$, and moreover the functor~\eqref{eqn:functor-hKmod-D} restricts to a t-exact equivalence
\[
\bigl( (\fh,K)\mod \bigr)^+ \simto D((\fh,K)\mod^{\heartsuit})^+.
\]
In particular, the heart of the t-structure on $(\fh,K)\mod$ identifies with $(\fh,K)\mod^\heartsuit$.

\sss
The action of the monoidal category $\sH_H$ on $(\fh,K)\mod^{\heartsuit, \mathrm{fg}}$ (see~\S\ref{sss:action-hcbim-hcmod}) induces an action of $\Chb(\sH_H)$ on $\Chb((\fh,K)\mod^{\heartsuit, \mathrm{fg}})$. Since any object in $\sH_H$ is flat as a right $U(\fh)$-module, tensoring with any complex in $\Chb(\sH_H)$ sends acyclic complexes in $\Chb((\fh,K)\mod^{\heartsuit, \mathrm{fg}})$ to acyclic complexes. It is also easily seen that the tensor product of an acyclic complex in $\Chb(\sH_H)$ with any complex in $\Chb((\fh,K)\mod^{\heartsuit, \mathrm{fg}})$ is acyclic. In view of the considerations in~\S\ref{sss:algebra-structure-Verdier-quotient}, we therefore obtain on $(\fh,K)\mod$ a structure of $\hc_H$-module, i.e.~of strong categorical representation of $H$. The action of $\hc_H$ on $(\fh,K)\mod$ is right t-exact in the sense that the binary operation $\hc_H \times (\fh,K)\mod \to (\fh,K)\mod$ sends $(\hc_H)^{\leq 0} \times ((\fh,K)\mod)^{\leq 0}$ into $((\fh,K)\mod)^{\leq 0}$.

\sss
One can also describe the $\infty$-category $(\fh,K)\mod^\comp$ in a way which is more parallel to the construction of $\hc_H^\comp$, as we now explain.
We will denote by $\sH_{H,K}$ the full subcategory of $(\fh,K)\mod^{\heartsuit,\mathrm{fg}}$ spanned by objects which are finite extensions of objects of the form $\ind_{K}^{(\fh,K)}(V)$ with $V \in \Rep^{\heartsuit,\fd}(K)$. It follows from the formula~\eqref{eqn:tensor-induced-hc} that $\sH_{H,K}$ is a module for the
monoidal additive category $\sH_H$. 

\sss
\label{sss:duality-H-HK}
Consider the opposite group $H^\rev$, and its subgroup $K^\rev$.
We have an involutive equivalence of additive categories
\begin{equation}
\label{eqn:duality-h-K}
\DD_{H,K} : \sH_{H,K} \simto (\sH_{H^\rev,K^\rev})^{\op}
\end{equation}
constructed as follows. Consider the abelian category $\fh\mod_K^\heartsuit$ of $K$-equivariant $U(\fh)$-modules, and the full subcategory $\fh\mod_K^{\heartsuit,\mathrm{fg}}$ of $K$-equivariant $U(\fh)$-modules which are finitely generated as $U(\fh)$-modules. Then $(\fh,K)\mod^\heartsuit$ is a full subcategory of $\fh\mod_K^\heartsuit$ (not stable under extensions), and $(\fh,K)\mod^{\heartsuit, \mathrm{fg}}$ is a full subcategory of $\fh\mod_K^{\heartsuit,\mathrm{fg}}$.

For any 
$M \in \fh\mod_K^{\heartsuit,\mathrm{fg}}$, 
the space $\Hom_{U(\fh)}(M,U(\fh))$ admits a canonical action of $K$ (induced by the actions on $M$ and $U(\fh)$), which we can twist via $k \mapsto k^{-1}$ to obtain an action of $K^\rev$, and a canonical right action of $U(\fh)$ (induced by right multiplication on $U(\fh)$). Taken together, these actions define a structure of $K^\rev$-equivariant $U(\fh^\rev)$-module. Deriving this functor we obtain, for any $n \in \Z_{\geq 0}$, a functor
\[
\Ext^n_{U(\fh)}(-, U(\fh)) : \fh\mod_K^{\heartsuit,\mathrm{fg}} \to (\fh^\rev\mod_{K^\rev}^\heartsuit)^\op.
\]
Using a Chevalley--Eilenberg complex one checks that for $V \in \Rep^{\heartsuit,\fd}(K)$ we have
\begin{equation}
\label{eqn:Ext-Ind-hc}
\Ext^n_{U(\fh)}(\ind_{K}^{(\fh,K)}(V), U(\fh)) =
\begin{cases}
(V^* \otimes_k \det(\fk)^*) \otimes_{U(\fk)} U(\fh) & \text{if $n=\dim(\fk)$;}\\
0 & \text{otherwise}
\end{cases}
\end{equation}
(where $\det(\fk)$ is the top exterior power of $\fk$, and $V^* \otimes_k \det(\fk)^*$ is viewed as a $K^\rev$-module in the natural way).
This formula implies that the functor $\Ext^{\dim(\fk)}_{U(\fh)}(-, U(\fh))$ restricts to a functor from $\sH_{H,K}$ to $(\sH_{H^\rev,K^\rev})^\op$, which we choose as the functor~\eqref{eqn:duality-h-K}.

It is not difficult to check that for $M \in \sH_H$ and $N \in \sH_{H,K}$ we have a canonical isomorphism
\begin{equation}
\label{eqn:duality-HK-action}
\DD_{H,K}(M \otimes_{U(\fh)} N) \simeq \DD_{H,K}(N) \otimes_{U(\fh)} M^\vee.
\end{equation}

\begin{Rem}
There is a canonical equivalence $\fh\mod_K^\heartsuit \simeq \fh^\rev\mod_{K^\rev}^\heartsuit$ induced by inversion on $H$, which induces an equivalence $\sH_{H,K} \simeq \sH_{H^\rev, K^\rev}$, so that $\DD_{H,K}$ can be seen as a contravariant autoequivalence of $\sH_{H,K}$. The description using $H^\rev$ seems more canonical however.
\end{Rem}

\sss
The following statement is a counterpart for Harish-Chandra modules of Proposition~\ref{prop:resolution-hc}.

\begin{Prop}
\label{prop:resolution-hc-mod}
Let $M \in (\fh,K)\mod^{\heartsuit,\mathrm{fg}}$, and let $n = \dim (\fh/\fk)$. There exists an exact sequence
\[
0 \to P^n \to P^{n-1} \to \cdots \to P^0 \to M \to 0
\]
in $(\fh,K)\mod^{\heartsuit,\mathrm{fg}}$
where each $P^i$ belongs to $\sH_{H,K}$.
\end{Prop}

\begin{proof}
This statement can be obtained by adapting the proof of Proposition~\ref{prop:resolution-hc}. Namely, we consider filtered $(\fh,K)$-modules (where we require the natural compatibility of the filtration with the action of $U(\fh)$, and that $K$ preserves the filtration). The action of $S(\fh)$ on the associated graded of such a module factors through an action of $S(\fh/\fk)$. We replace the condition of being $H$-equivariantly free for an $H$-equivariant $U(\fh)$-module, resp.~graded $S(\fh)$-module, by the condition of having a filtration by modules of the form $\mathrm{ind}_{K}^{(\fh,K)}(V)$ with $V \in \Rep^{\heartsuit,\fd}(K)$, resp.~of the form $S(\fh/\fk) \otimes_k V \langle n \rangle$ with $V \in \Rep^{\heartsuit,\fd}(K)$ and $n \in \Z$. Then we have natural analogues in this context of Lemmas~\ref{lem:sh-free-criterion},~\ref{lem:uh-free-graded} and~\ref{lem:lift-resoln}, which allow us to copy the proof of Proposition~\ref{prop:resolution-hc}.
\end{proof}

\sss
\label{sss:(h,K)modc-HHK}
Using Proposition~\ref{prop:resolution-hc-mod}, one sees as in the proof of Lemma~\ref{lem:HCc-DbHC} that there exists a canonical equivalence from the Verdier quotient of $\Chb(\sH_{H,K})$ by the full subcategory spanned by acyclic complexes to $(\fh,K)\mod^\comp$. Since $\sH_{H,K}$ is a module for $\sH_H$, one can also observe the action of $\hc_H$ on $(\fh,K)\mod$ using this perspective.

\sss
There exists an involutive equivalence of $\infty$-categories
\begin{equation}
\label{eqn:duality-h-K-comp}
\DD_{(\fh,K)} : (\fh,K)\mod^\comp \simto ((\fh^\rev,K^\rev)\mod^\comp)^\op
\end{equation}
constructed as follows. First, the functor $\DD_{H,K}$ constructed in~\S\ref{sss:duality-H-HK}
induces an equivalence $\Chb(\sH_{H,K}) \simto \Chb(\sH_{H^\rev,K^\rev})^\op$. We observe that this functor sends acyclic complexes to acyclic complexes, so it factors through the desired equivalence~\eqref{eqn:duality-h-K-comp}. In fact this follows from the observation that given an exact sequence $0 \to M_1 \to M_2 \to M_3 \to 0$, if $\Ext^n_{U(\fh)}(M_i, U(\fh))=0$ for $i \in \{1,2\}$ and $n < \dim(\fk)$, resp.~for $i \in \{2,3\}$ and $n > \dim(\fk)$, then the same vanishing holds for $i=3$, resp.~for $i=1$. Hence given an acyclic complex of objects in $\sH_{H,K}$, each kernel $M$ of each differential satisfies $\Ext^n_{U(\fh)}(M, U(\fh))=0$ unless $n = \dim(\fk)$, which implies that the complex obtained by applying $\DD_{H,K}$ to each term is acyclic.

From the formula~\eqref{eqn:duality-HK-action} we obtain that for $M \in \hc_H^\comp$ and $N \in (\fh,K)\mod^\comp$ we have
\[
\DD_{(\fh,K)}(M \otimes_{U(\fh)} N) \simeq \DD_{(\fh,K)}(N) \otimes_{U(\fh)} M^\vee.
\]

\sss
\label{sss:h-mod}
As a particular case of the discussion above one can take $K=\{1\}$ the trivial subgroup; in this case we will write $\fh\mod$ for $(\fh,\{1\})\mod$. Note that since $U(\fh)$ has finite global dimension, the $\infty$-category $\fh\mod$ is simply the $\infty$-category of (complexes of) $U(\fh)$-modules, which identifies with $\LMod_{U(\fh)}(\Vect_k)$. 
In terms of the ``induction'' functor from weak to strong categorical representations, see~\S\ref{sss:restriction-induction-weak-strong}, we have
\begin{equation}
\label{eqn:hmod-induction-Vect}
\hc_H \otimes_{\Rep(H)} \Vect_k = \fh\mod.
\end{equation}
In fact, this identification follows from Proposition~\ref{prop:hchind} and~\cite[Chap.~1, Corollary~8.5.7]{gr}.

\sss
Another interesting special case is when $K=H$. In this case we have $(\fh,H)\mod=\Rep(H)$.

\sss
\label{sss:dualizability-gKmod}
By definition $(\fh,K)\mod$ is the Ind-completion of a small stable $\infty$-category. Hence it is dualizable in $\PrLSt$ and in $\DGCat$, with dual
\begin{equation}
\label{eqn:dual-(h,K)-mod}
((\fh,K)\mod)^\vee = \Ind(((\fh,K)\mod^{\comp})^{\op}),
\end{equation}
see~\cite[Chap.~1, Proposition~7.3.2 and Proposition~9.5.3]{gr}. Using the equivalence~\eqref{eqn:duality-h-K-comp} one can also describe the dual of $(\fh,K)\mod$ as $(\fh^\rev,K^\rev)\mod$.

The left action of $\hc_H$ on $(\fh,K)\mod$ induces a right action on $((\fh,K)\mod)^\vee$, see~\cite[Chap.~1, \S 9.4.1]{gr}; under the identification~\eqref{eqn:dual-(h,K)-mod} this action is induced by the right action of $(\hc_H)^\comp$ on $((\fh,K)\mod^{\comp})^{\op}$ given by $N \cdot M = ({}^\vee M) \otimes_{U(\fh)} N$ for $M \in (\hc_H)^\comp$ and $N \in (\fh,K)\mod^{\comp}$. Using the equivalence~\eqref{eqn:duality-h-K-comp}, the right action of $\hc_H$ on $(\fh^\rev,K^\rev)\mod$ identified with $((\fh,K)\mod)^\vee$ is therefore given by the obvious action where $\fh^\rev$-modules are identified with right $U(\fh)$-modules.

Finally, one can study the dualizability of $(\fh,K)\mod$ as a left strong categorical representation. By the comments in~\S\ref{sss:dualizability-strong-reps}, $(\fh,K)\mod$ is right dualizable as a left strong categorical representation, and its right dual is the right strong categorical representation given by $(\fh^\rev,K^\rev)\mod$ with the natural right action of $\hc_H$.

%-------------------------------------------------------
\subsection{Functorialities for \texorpdfstring{$\infty$}{infinity}-categories of Harish-Chandra modules}
\label{ss:hKmod-functorialities}
%-------------------------------------------------------

\sss
We will now study functorialities for Harish-Chandra modules. First, consider an affine group scheme $H_2$ of finite type over $k$, and subgroups $K \subset H_1 \subset H_2$. Denote by $\fk$, $\fh_1$, $\fh_2$ the Lie algebras of $K$, $H_1$ and $H_2$ respectively. Then we have an exact restriction functor
\[
\res^{(\fh_2,K)}_{(\fh_1,K)} : (\fh_2,K)\mod^\heartsuit \to (\fh_1,K)\mod^\heartsuit
\]
and its exact left adjoint
\[
\ind^{(\fh_2,K)}_{(\fh_1,K)} : (\fh_1,K)\mod^\heartsuit \to (\fh_2,K)\mod^\heartsuit
\]
given by $\ind^{(\fh_2,K)}_{(\fh_1,K)}(M) = U(\fh_2) \otimes_{U(\fh_1)} M$.

The functor $\ind^{(\fh_2,K)}_{(\fh_1,K)}$ restricts to an exact functor $(\fh_1,K)\mod^{\heartsuit,\mathrm{fg}} \to (\fh_2,K)\mod^{\heartsuit,\mathrm{fg}}$ which induces a functor $(\fh_1,K)\mod^\comp \to (\fh_2,K)\mod^\comp$, and then a functor
\[
\ind^{(\fh_2,K)}_{(\fh_1,K)} : (\fh_1,K)\mod \to (\fh_2,K)\mod.
\]
Since this functor sends compact objects to compact objects, its right adjoint $\res^{(\fh_2,K)}_{(\fh_1,K)}$ is continuous. In fact this right adjoint can be explicitly constructed by restricting the composition
\[
D((\fh_2,K)\mod^\heartsuit)^+ \to D((\fh_1,K)\mod^\heartsuit)^+ \simeq ((\fh_1,K)\mod)^+ \subset (\fh_1,K)\mod
\]
(where the first arrow is induced by the functor $\res^{(\fh_2,K)}_{(\fh_1,K)}$ considered above, and the equivalence is explained in~\S\ref{sss:hKmod-renormalization}) to $\Db((\fh_1,K)\mod^{\heartsuit,\mathrm{fg}})$, and then using the universal property of the Ind-completion.

In particular, when $K=H_1$ we have $\ind_{(\fk,K)}^{(\fh_2,K)} = \ind_{K}^{(\fh_2,K)}$ where the right-hand side is as in~\S\ref{sss:def-(h,K)mod}. Accordingly, in this case we will write $\res_{K}^{(\fh_2,K)}$ for $\res_{(\fk,K)}^{(\fh_2,K)}$.

\sss
\label{sss:coind-hcmod}
We now consider an affine group scheme $H$ of finite type over $k$, and two subgroups $K_1 \subset K_2 \subset H$. Denote the corresponding Lie algebras by $\fk_1$, $\fk_2$ and $\fh$. We then have an obvious exact restriction functor
\[
\res^{(\fh,K_2)}_{(\fh,K_1)} : (\fh,K_2)\mod^\heartsuit \to (\fh,K_1)\mod^\heartsuit
\]
which restricts to a functor $(\fh,K_2)\mod^{\heartsuit,\mathrm{fg}} \to (\fh,K_1)\mod^{\heartsuit,\mathrm{fg}}$. Taking the induced functor on bounded derived categories and then Ind-completing, we deduce a continuous functor
\[
\res^{(\fh,K_2)}_{(\fh,K_1)} : (\fh,K_2)\mod \to (\fh,K_1)\mod.
\]
By construction this functor sends compact objects to compact objects, so its right adjoint
\[
\coind^{(\fh,K_2)}_{(\fh,K_1)} : (\fh,K_1)\mod \to (\fh,K_2)\mod
\]
is continuous. It is clear that the functor $\res^{(\fh,K_2)}_{(\fh,K_1)}$ is $\hc_H$-linear; by rigidity, it follows that $\coind^{(\fh,K_2)}_{(\fh,K_1)}$ is also $\hc_H$-linear. Since $\res^{(\fh,K_2)}_{(\fh,K_1)}$ is t-exact, $\coind^{(\fh,K_2)}_{(\fh,K_1)}$ is left t-exact.

\sss
\label{sss:transitivity-coind-hcmod}
If we are given subgroups $K_1 \subset K_2 \subset K_3 \subset H$ then we have
\[
\res^{(\fh,K_3)}_{(\fh,K_1)} = \res^{(\fh,K_2)}_{(\fh,K_1)} \circ \res^{(\fh,K_3)}_{(\fh,K_2)},
\]
and hence
\[
\coind^{(\fh,K_3)}_{(\fh,K_1)} =  \coind^{(\fh,K_3)}_{(\fh,K_2)} \circ \coind^{(\fh,K_2)}_{(\fh,K_1)}.
\]

\sss
\label{sss:coind-res-diagram}
If $K_1$, $K_2$, $H$ are as in~\S\ref{sss:coind-hcmod}, we have a commutative diagram
\[
\begin{tikzcd}[column sep=2cm]
(\fh,K_1)\mod \ar[r, "\coind^{(\fh,K_2)}_{(\fh,K_1)}"] \ar[d, "\res^{(\fh,K_1)}_{(\fk_2,K_1)}"'] & (\fh,K_2)\mod \ar[d, "\res^{(\fh,K_2)}_{K_2}"] \\
(\fk_2, K_1)\mod \ar[r, "\coind_{(\fk_2,K_1)}^{(\fk_2,K_2)}"] & \Rep(K_2)
\end{tikzcd}
\]
obtained by passage to right adjoints from the commutative diagram
\[
\begin{tikzcd}[column sep=2cm]
(\fh,K_1)\mod & (\fh,K_2)\mod \ar[l, "\res^{(\fh,K_2)}_{(\fh,K_1)}"'] \\
(\fk_2, K_1)\mod \ar[u, "\ind^{(\fh,K_1)}_{(\fk_2,K_1)}"] & \Rep(K_2). \ar[u, "\ind^{(\fh,K_2)}_{K_2}"'] \ar[l, "\res_{(\fk_2,K_1)}^{(\fk_2,K_2)}"']
\end{tikzcd}
\]

\sss
\label{sss:coind-hcmod-same-Liealg-1}
Let us consider the construction of~\S\ref{sss:coind-hcmod} in the special case when $\fk_1=\fk_2=:\fk$ (but $K_1$ might be different from $K_2$). In this case we have a commutative diagram
\[
\begin{tikzcd}[column sep=2cm]
(\fh,K_2)\mod \ar[r, "\res^{(\fh,K_2)}_{(\fh,K_1)}"] & (\fh,K_1)\mod \\
\Rep(K_2) \ar[r, "\res^{K_2}_{K_1}"] \ar[u, "\ind_{K_2}^{(\fh,K_2)}"] & \Rep(K_1). \ar[u, "\ind_{K_1}^{(\fh,K_1)}"'] 
\end{tikzcd}
\]
Passing to right adjoints we deduce a commutative diagram
\begin{equation}
\label{eqn:diag-coind-res-hcmod}
\begin{tikzcd}[column sep=2cm]
(\fh,K_2)\mod \ar[d, "\res_{K_2}^{(\fh,K_2)}"'] & (\fh,K_1)\mod \ar[l, "\coind^{(\fh,K_2)}_{(\fh,K_1)}"'] \ar[d, "\res_{K_1}^{(\fh,K_1)}"] \\
\Rep(K_2) & \Rep(K_1). \ar[l, "\coind^{K_2}_{K_1}"'] 
\end{tikzcd}
\end{equation}
In this diagram the horizontal arrows are left t-exact (see~\S\ref{sss:Res-Coind} and~\S\ref{sss:coind-hcmod}), and the bounded below part of each of these $\infty$-categories identifies with the bounded below derived $\infty$-category of its heart, see~\S\S\ref{sss:QCoh-BH}--\ref{sss:def-RepH} and~\S\ref{sss:hKmod-renormalization}. In particular, the restriction of $\res_{K_2}^{(\fh,K_2)}$ to $((\fh,K_2)\mod)^+$ is conservative. Since $\coind^{K_2}_{K_1}$ has bounded cohomological dimension (see~\cite[Part~I, Proposition~5.12(b)]{jantzen}), we deduce that $\coind^{(\fh,K_2)}_{(\fh,K_1)}$ preserves the bounded part of the t-structure.

\sss
\label{sss:coind-hcmod-same-Liealg-2}
Continue with the setting of~\S\ref{sss:coind-hcmod-same-Liealg-1}, and consider the diagram
\[
\begin{tikzcd}[column sep=2cm]
((\fh,K_2)\mod)^+ & ((\fh,K_1)\mod)^+ \ar[l, "\coind^{(\fh,K_2)}_{(\fh,K_1)}"'] \\
(\Rep(K_2))^+ \ar[u, "\ind_{K_2}^{(\fh,K_2)}"] & (\Rep(K_1))^+. \ar[l, "\coind^{K_2}_{K_1}"'] \ar[u, "\ind_{K_1}^{(\fh,K_1)}"']
\end{tikzcd}
\]
We have a ``Beck--Chevalley map''
\[
\ind_{K_2}^{(\fh,K_2)} \circ \coind^{K_2}_{K_1} \to \coind^{(\fh,K_2)}_{(\fh,K_1)} \circ \ind_{K_1}^{(\fh,K_1)}.
\]
Using the commutativity of~\eqref{eqn:diag-coind-res-hcmod} and the tensor identity for coinduction of representations (see~\cite[Part~I, Proposition~3.6]{jantzen}) one sees that this morphism becomes an equivalence after application of the conservative functor $\res_{K_2}^{(\fh,K_2)}$. Hence the map itself is an equivalence.

Assume now, in addition, that the quotient $K_2/K_1$ is a projective scheme. Then the functor $\coind^{K_2}_{K_1}$ sends compact objects to compact objects, see~\cite[Part~I, Proposition~5.12(c)]{jantzen}. Since $((\fh,K_1)\mod)^\comp$ is generated, as a small stable $\infty$-category, by the image of $(\Rep(K_1))^\comp$ under $\ind_{K_1}^{(\fh,K_1)}$, see~\S\ref{sss:(h,K)modc-HHK}, we deduce that $\coind^{(\fh,K_2)}_{(\fh,K_1)}$ sends compact objects to compact objects; in other words, its right adjoint is continuous.

\sss
We now drop the assumption that $\fk_1=\fk_2$, and come back to the general setting of~\S\ref{sss:coind-hcmod}.

\begin{Lem}
\label{lem:coind-compact-objects}
Assume that $k$ is a perfect field of characteristic $p>0$, that $K_1$ and $K_2$ are smooth, and that the quotient $K_2/K_1$ is projective. Then the functor
\[
\coind^{(\fh,K_2)}_{(\fh,K_1)} : (\fh,K_1)\mod \to (\fh,K_2)\mod
\]
sends compact objects to compact objects; in other words, its right adjoint is continuous.
\end{Lem}

\begin{proof}
Denote by $K_{12}$ the subgroup of $K_2$ generated by $K_1$ and the Frobenius kernel of $K_2$. Then we have inclusions $K_1 \subset K_{12} \subset K_2$, and hence
$\coind^{(\fh,K_2)}_{(\fh,K_1)} = \coind^{(\fh,K_2)}_{(\fh,K_{12})} \circ \coind^{(\fh,K_{12})}_{(\fh,K_1)}$, see~\S\ref{sss:transitivity-coind-hcmod}.
Here $K_2$ and $K_{12}$ have the same Lie algebra, and the quotient $K_2/K_{12} \simeq (K_2/K_1)^{(1)}$ is projective. Hence, by~\S\ref{sss:coind-hcmod-same-Liealg-2}, $\coind^{(\fh,K_2)}_{(\fh,K_{12})}$ sends compact objects to compact objects, which reduces the proof of the lemma to proving that $\coind^{(\fh,K_{12})}_{(\fh,K_1)}$ sends compact objects to compact objects.

Recall the Frobenius center of $U(\fh)$, which identifies canonically with $\OO(\fh^{*(1)})$, and recall that $U(\fh)$ is finite and free over this central subalgebra. For any $M \in (\fh,K_1)\mod^\heartsuit$, the action of $\OO(\fh^{*(1)})$ on $M$ factors through an action of $\OO((\fh/\fk_1)^{*(1)})$; we deduce a canonical t-exact functor
\[
(\fh,K_1)\mod \to \IndCoh((\fh/\fk_1)^{*(1)} / K_1)
\]
(where $K_1$ acts on $(\fh/\fk_1)^{*(1)}$ via the Frobenius morphism).
This functor has a left adjoint
\[
\IndCoh((\fh/\fk_1)^{*(1)} / K_1) \to (\fh,K_1)\mod
\]
induced by the exact functor sending a $K_1$-equivariant coherent sheaf on $(\fh/\fk_1)^{*(1)}$, considered as a $K_1$-equivariant $\OO((\fh/\fk_1)^{*(1)})$-module $M$, to $U(\fh) \otimes_{U(\fk_1) \cdot \OO(\fh^{*(1)})} M$. (Here $U(\fk_1) \cdot \OO(\fh^{*(1)})$ is the subalgebra of $U(\fh)$ generated by $U(\fk_1)$ and $\OO(\fh^{*(1)})$; note that $U(\fh)$ is free over this subalgebra by the Poincar\'e--Birkhoff--Witt theorem, which guarantees exactness of the functor. The $U(\fk_1)$-action on $M$ is obtained by differentiating the $K_1$-action.)

Similarly we have a natural exact functor
\[
(\fh,K_{12})\mod \to \IndCoh((\fh/\fk_2)^{*(1)} / K_1),
\]
which has a left adjoint
\[
\IndCoh((\fh/\fk_2)^{*(1)} / K_1) \to (\fh,K_{12})\mod
\]
induced by the restriction of the functor considered above to $K_1$-equivariant coherent sheaves on $(\fh/\fk_2)^{*(1)}$. (Here we use the identification between representations of the Frobenius kernel of $K_2$ and modules over the restricted enveloping algebra of $\fk_2$, see e.g.~\cite[Part~I, \S 9.6]{jantzen}, to extend the action of $K_1$ to an action of $K_{12}$.)

We have a commutative diagram
\[
\begin{tikzcd}[column sep=2cm]
(\fh,K_{12})\mod \ar[r, "\res^{(\fh,K_{12})}_{(\fh,K_1)}"] & (\fh,K_1)\mod \\
\IndCoh((\fh/\fk_2)^{*(1)} / K_1) \ar[r, "i_*^{\IndCoh}"] \ar[u] & \IndCoh((\fh/\fk_1)^{*(1)} / K_1) \ar[u]
\end{tikzcd}
\]
where the vertical arrows are the functors considered above, and $i$ is induced by the closed embedding $(\fh/\fk_2)^{*(1)} \hookrightarrow (\fh/\fk_1)^{*(1)}$. Passing to right adjoints, we deduce a commutative diagram
\[
\begin{tikzcd}[column sep=2cm]
(\fh,K_{12})\mod \ar[d] & (\fh,K_1)\mod \ar[l, "\coind^{(\fh,K_{12})}_{(\fh,K_1)}"'] \ar[d] \\
\IndCoh((\fh/\fk_2)^{*(1)} / K_1) & \IndCoh((\fh/\fk_1)^{*(1)} / K_1) \ar[l, "i_{\IndCoh}^!"']
\end{tikzcd}
\]
where again the vertical arrows are as above. Here the functor $i_{\IndCoh}^!$ sends compact objects to compact objects, see~\cite[Proposition~9.40(1)]{zhu}. Using the same arguments as in~\S\ref{sss:coind-hcmod-same-Liealg-2} we deduce that $\coind^{(\fh,K_{12})}_{(\fh,K_1)}$ sends compact objects to compact objects, which finishes the proof.
\end{proof}

\sss
We now change setting, and consider two affine $k$-group schemes of finite type $H_1$ and $H_2$, with respective Lie algebras $\fh_1$ and $\fh_2$, and subgroups $K_1 \subset H_1$, $K_2 \subset H_2$.

\begin{Lem}
\label{lem:product-hcmod}
There exists a canonical fully faithful functor
\[
\bigl( (\fh_1,K_1)\mod \bigr) \otimes_k \bigl( (\fh_2,K_2)\mod \bigr) \to \bigl( (\fh_1 \oplus \fh_2,K_1 \times_k K_2)\mod \bigr).
\]
If the functor
$\Rep(K_1) \otimes_k \Rep(K_2) \to \Rep(K_1 \times_k K_2)$
of Lemma~\ref{lem:Rep-product} is an equivalence, then this functor is an equivalence. 
If the functor $\Rep(H_1) \otimes_k \Rep(H_2) \to \Rep(H_1 \times_k H_2)$
of Lemma~\ref{lem:Rep-product} is also an equivalence, this equivalence is compatible with the actions of $\hc_{H_1 \times_k H_2}$ on both sides, where the action on the left-hand side is obtained using Lemma~\ref{lem:HC-product}.
\end{Lem}

\begin{proof}
The functor is induced by the assignment sending a pair $(M,N)$ to the tensor product $M \otimes_k N$, with the obvious actions of $\fh_1 \oplus \fh_2$ and $K_1 \times_k K_2$.
We have a commutative diagram
\[
\begin{tikzcd}
\bigl( (\fh_1,K_1)\mod \bigr) \otimes_k \bigl( (\fh_2,K_2)\mod \bigr) \ar[r] \ar[d, shift left=0.5ex] & \bigl( (\fh_1 \oplus \fh_2,K_1 \times_k K_2)\mod \bigr) \ar[d, shift left=0.5ex] \\
\Rep(K_1) \otimes_k \Rep(K_2) \ar[r] \ar[u, shift left=0.5ex] & \Rep(K_1 \times_k K_2) \ar[u, shift left=0.5ex]
\end{tikzcd}
\]
where the vertical arrows are induced by corresponding induction and restriction functors, and in each case the upward arrow is left adjoint to the downward arrow. Since $\bigl( (\fh_1,K_1)\mod \bigr) \otimes_k \bigl( (\fh_2,K_2)\mod \bigr)$ is compactly generated by objects in the image of $\Rep(K_1) \otimes_k \Rep(K_2)$ (see~\cite[Chap.~1, \S 10.5.7]{gr}), and since the lower horizontal arrow is fully faithful by Lemma~\ref{lem:Rep-product}, we deduce the full faithfulness of the upper horizontal arrow.

In case the lower arrow is an equivalence, this diagram shows that the natural compact generators of $\bigl( (\fh_1 \oplus \fh_2,K_1 \times_k K_2)\mod \bigr)$ belong to the image of our functor, which implies that the functor is also essentially surjective.
\end{proof}

\sss
We will now use the considerations above to study properness (in the sense of~\S\ref{sss:smooth-proper-modules}) of $(\fh,K)\mod$ as a (right dualizable) strong representation, see~\S\ref{sss:dualizability-gKmod}. We must therefore study the counit morphism
\begin{equation}
\label{eqn:counit-hcmod-1}
\left( (\fh,K)\mod \right) \otimes_k \left( (\fh^\rev,K^\rev)\mod \right) \to \hc_H.
\end{equation}
We will assume that the canonical functors $\Rep(K) \otimes_k \Rep(K) \to \Rep(K \times_k K)$ and $\Rep(H) \otimes_k \Rep(H) \to \Rep(H \times_k H)$ are equivalences, and compare the counit above with the composition
\begin{multline}
\label{eqn:counit-hcmod-2}
\left( (\fh,K)\mod \right) \otimes_k \left( (\fh^\rev,K^\rev)\mod \right) \xrightarrow[\sim]{\text{Lem.~\ref{lem:product-hcmod}}} (\fh \oplus \fh^\rev,K \times_k K^\rev)\mod \\
\xrightarrow{\res^{(\fh \oplus \fh^\rev,K \times K^\rev)}_{(\fh \oplus \fh^\rev,\Delta K)}} (\fh \oplus \fh^\rev, \Delta K)\mod \xrightarrow{\coind_{(\fh \oplus \fh^\rev,\Delta K)}^{(\fh \oplus \fh^\rev,\Delta H)}} (\fh \oplus \fh^\rev, \Delta H)\mod \xrightarrow[\sim]{\text{\S\ref{sss:action-hcbim-hcmod}}} \hc_H.
\end{multline}
(Here, $\Delta K$ and $\Delta H$ are the antidiagonal copies of $K$ and $H$ in $H \times_k H^\rev$.)
In the following statement we use the notation $\boxtimes_k$ as in~\cite[Chap.~1, \S 10.4.1]{gr}.
We also consider $\det(\fh)$ as a right $U(\fh)$-module by twisting the natural left action by the antipode.

\begin{Lem}
\label{lem:counit-hcmod}
Let $V_1 \in \Rep(K)$ and $V_2 \in \Rep(K^\rev)$. Then the image of $\ind_K^{(\fh,K)}(V_1) \boxtimes_k \ind_{K^\rev}^{(\fh^\rev,K^\rev)}(V_2)$ under~\eqref{eqn:counit-hcmod-1} identifies with the image of $\ind_K^{(\fh,K)}(V_1) \boxtimes_k \left( \ind_{K^\rev}^{(\fh^\rev,K^\rev)}(V_2) \otimes_k \det(\fh) \right)[\dim(\fh/\fk)]$ under~\eqref{eqn:counit-hcmod-2}.
\end{Lem}

\sss
Before we give the proof of this lemma we establish a preparatory result.

\begin{Lem}
\label{lem:identification-Maps-det}
For any $M \in \Rep(K)$, there exists a canonical identification in $\Vect_k$
\[
\Map_{\Rep(K)}(\det(\fh/\fk)[\dim(\fh/\fk)], M) \simeq \Map_{(\fh,K)\mod}(k, \ind_K^{(\fh,K)}(M)).
\]
\end{Lem}

\begin{proof}
Recall that we have a resolution
\[
U(\fh) \otimes_{U(\fk)} \wedge^{-\bullet}(\fh/\fk)
\]
of the trivial $(\fh,K)$-module, where the description of the differential is similar to that in the Chevalley--Eilenberg complex. (This resolution is discussed e.g.~in~\cite[Chap.~VII, \S 8]{knapp} over the complex numbers and for specific choices of $H$ and $K$; the description given in~\cite[Eqns.~(7.53c) and (7.55)]{knapp} works in our present generality.) Using the stupid truncation for this complex, we obtain a canonical morphism
\begin{equation}
\label{eqn:morph-resolution-trivial-hKmod}
k \to \ind_K^{(\fh,K)}(\det(\fh/\fk))[\dim(\fh/\fk)]
\end{equation}
in $(\fh,K)\mod$.

Using this morphism one can construct a functorial morphism
\[
\Map_{\Rep(K)}(\det(\fh/\fk)[\dim(\fh/\fk)], M) \to \Map_{(\fh,K)\mod}(k, \ind_K^{(\fh,K)}(M))
\]
as the composition of the morphism induced by $\ind_K^{(\fh,K)} : \Rep(K) \to (\fh,K)\mod$ with the morphism
\[
\Map_{(\fh,K)\mod}(\ind_K^{(\fh,K)}(\det(\fh/\fk))[\dim(\fh/\fk)], \ind_K^{(\fh,K)}(M)) \to \Map_{(\fh,K)\mod}(k, \ind_K^{(\fh,K)}(M))
\]
induced by~\eqref{eqn:morph-resolution-trivial-hKmod}. To conclude the proof it remains to prove that this morphism is an isomorphism. 

For this one can assume that $M \in \Rep^{\heartsuit,\fd}(K)$, and use the resolution above to compute the complex $\Map_{(\fh,K)\mod}(k, \ind_K^{(\fh,K)}(M))$; we obtain that this complex is the colimit of a (finite) system of complexes whose $n$-th term is
\begin{multline*}
\Map_{(\fh,K)\mod}(U(\fh) \otimes_{U(\fk)} \wedge^{n}(\fh/\fk), \ind_K^{(\fh,K)}(M)) \simeq
\Map_{\Rep(K)}(\wedge^{n}(\fh/\fk), \ind_K^{(\fh,K)}(M)) \\
\simeq \Map_{\Rep(K)}(k, (\wedge^{n}(\fh/\fk))^* \otimes_k (U(\fh) \otimes_{U(\fk)} M)).
\end{multline*}
Next we use the canonical isomorphism $(\wedge^{n}(\fh/\fk))^* \simeq \wedge^{\dim(\fh/\fk)-n}(\fh/\fk) \otimes_k \det(\fh/\fk)^*$ to identify this system with a system whose $n$-th term is
\[
\Map_{\Rep(K)}(\det(\fh/\fk), \wedge^{\dim(\fh/\fk)-n}(\fh/\fk) \otimes_k (U(\fh) \otimes_{U(\fk)} M)).
\]
Finally we consider the filtration on this system induced by the Poincar\'e--Birkhoff--Witt filtration on $U(\fh)$; in the associated graded we recognize the Koszul complex of the vector space $\fh/\fk$, which is a resolution of the trivial module and allows us to check that our map is indeed an isomorphism.
\end{proof}

\sss
We can now explain the proof of Lemma~\ref{lem:counit-hcmod}.

\begin{proof}[Proof of Lemma~\ref{lem:counit-hcmod}]
By definition, under the identification~\eqref{eqn:dual-(h,K)-mod},
the counit
\[
\left( (\fh,K)\mod \right) \otimes_k \left( (\fh,K)\mod \right)^\vee \to \hc_H
\]
is defined using the Ind-extension of the following operation. Given $M \in (\fh,K)\mod^\comp$ and $N \in (\fh,K)\mod^{\comp,\op}$ one considers the object of $(\hc_H)^\vee = \Funct_k(\hc_H,\Vect_k)$ defined by $P \mapsto \Map_{(\fh,K)\mod}(N, P \otimes_{U(\fh)} M)$. Then, one uses the fact that we have an identification
\[
\hc_H \simto (\hc_H)^\vee
\]
sending $Q$ to the functor $P \mapsto \Map_{\hc_H}(U(\fh), P \otimes_{U(\fh)} Q)$ to see this object as an object of $\hc_H$. 

In our case we need to analyze this procedure in case
\[
M = \ind_K^{(\fh,K)}(V_1) \quad \text{and} \quad N = (\DD_{(\fh,K)})^{-1}(\ind_{K^\rev}^{(\fh^\rev,K^\rev)}(V_2)) \simeq \ind_K^{(\fh,K)}((V_2)^* \otimes_k \det(\fk)^*).
\]
(See~\eqref{eqn:Ext-Ind-hc} for the isomorphism in the right-hand side.)
So we need to understand the object
\[
\Map_{(\fh,K)\mod} \left( \ind_K^{(\fh,K)}((V_2)^* \otimes_k \det(\fk)^*), P \otimes_{U(\fh)} (U(\fh) \otimes_{U(\fk)} V_1) \right).
\]
By adjunction, this complex identifies with
\begin{equation}
\label{eqn:computation-Maps-counit}
\Map_{\Rep(K)} \left( (V_2)^* \otimes_k \det(\fk)^*, P \otimes_{U(\fk)} V_1 \right) \simeq \Map_{\Rep(K)} \left( k, (P \otimes_{U(\fk)} V_1) \otimes_k V_2 \otimes_k \det(\fk) \right)
\end{equation}
where the action on $(P \otimes_{U(\fk)} V_1) \otimes_k V_2 \otimes_k \det(\fk)$ is diagonal.

On the other hand, we want to consider
\[
\Map_{\hc_H} \left( U(\fh), P \otimes_{U(\fh)} \coind_{(\fh \oplus \fh^\rev,\Delta K)}^{(\fh \oplus \fh^\rev,\Delta H)}((U(\fh) \otimes_{U(\fk)} V_1) \otimes_k (V_2 \otimes_{U(\fk)} U(\fh) \otimes_k \det(\fh)))[\dim(\fh/\fk)] \right).
\]
Since the functor $\res_{(\fh \oplus \fh^\rev,\Delta K)}^{(\fh \oplus \fh^\rev,\Delta H)}$ is $\hc_H$-linear for the natural left actions, so is its right adjoint, which allows us to identify the complex above with 
\[
\Map_{\hc_H} \left( U(\fh), \coind_{(\fh \oplus \fh^\rev,\Delta K)}^{(\fh \oplus \fh^\rev,\Delta H)}((P \otimes_{U(\fk)} V_1) \otimes_k (V_2 \otimes_{U(\fk)} U(\fh) \otimes_k \det(\fh))  ) [\dim(\fh/\fk)] \right).
\]
Now we use the fact that $U(\fh) = \ind_{\Delta H}^{(\fh \oplus \fh^\rev, \Delta H)}(k)$ and adjunction to identify this complex with
\[
\Map_{\Rep(H)} \left( k, \res_{\Delta H}^{(\fh \oplus \fh^\rev, \Delta H)} \circ \coind_{(\fh \oplus \fh^\rev,\Delta K)}^{(\fh \oplus \fh^\rev,\Delta H)}((P \otimes_{U(\fk)} V_1) \otimes_k (V_2 \otimes_{U(\fk)} U(\fh) \otimes_k \det(\fh))) [\dim(\fh/\fk)] \right).
\]
Next, we use the commutative diagram from~\S\ref{sss:coind-res-diagram} to identify the complex with
\[
\Map_{\Rep(H)} \left( k, \coind_{(\fh,\Delta K)}^{H}((P \otimes_{U(\fk)} V_1) \otimes_k (U(\fh) \otimes_{U(\fk)} V_2) \otimes_k \det(\fh))[\dim(\fh/\fk)] \right)
\]
where now $V_2$ is seen as a left $U(\fk)$-module using the antipode and the actions on the tensor product $(P \otimes_{U(\fk)} V_1) \otimes_k (U(\fh) \otimes_{U(\fk)} V_2) \otimes_k \det(\fh)$ are diagonal.
By adjunction this complex identifies with
\[
\Map_{(\fh,K)\mod} \left( k, (P \otimes_{U(\fk)} V_1) \otimes_k (U(\fh) \otimes_{U(\fk)} V_2) \otimes_k \det(\fh)[\dim(\fh/\fk)] \right).
\]
Using the tensor identity we identify this complex with
\[
\Map_{(\fh,K)\mod} \left( k, U(\fh) \otimes_{U(\fk)}((P \otimes_{U(\fk)} V_1) \otimes_k V_2 \otimes_k \det(\fh))[\dim(\fh/\fk)] \right).
\]
Finally, by Lemma~\ref{lem:identification-Maps-det} this complex identifies with
\[
\Map_{\Rep(K)} \left( \det(\fh/\fk)[\dim(\fh/\fk)], (P \otimes_{U(\fk)} V_1) \otimes_k V_2\otimes_k \det(\fh) [\dim(\fh/\fk)] \right),
\]
which indeed identifies with the right-hand side of~\eqref{eqn:computation-Maps-counit}.
\end{proof}

\sss
\label{sss:hK-mod-proper}
By construction, $(\fh,K)\mod$ and $(\fh^\rev,K^\rev)\mod$ are compactly generated. By~\cite[Chap.~1, \S 10.5.7]{gr}, it follows that $( (\fh,K)\mod) \otimes_k ( (\fh^\rev,K^\rev)\mod )$ is compactly generated, with compact objects the full idempotent complete stable subcategory generated by the objects of the form $\ind_K^{(\fh,K)}(V_1) \boxtimes_k \ind_{K^\rev}^{(\fh^\rev,K^\rev)}(V_2)$ with $V_1,V_2 \in \Rep(K)^{\comp}$. Using Lemma~\ref{lem:coind-compact-objects} it follows that if $k$ is a perfect field of characteristic $p>0$, $K$ and $H$ are smooth, and $H/K$ is projective, the functor~\eqref{eqn:counit-hcmod-1} sends compact objects to compact objects, so that $(\fh,K)\mod$ is proper as a left $\hc_H$-module. (An example of this setting is obtained by taking $H$ to be a split reductive group scheme over $k$ and $K$ a parabolic subgroup. We do not expect such modules to be smooth.)

%-------------------------------------------------------
\subsection{Tensor products and functors}
\label{sss:tens-funct-strong}
%-------------------------------------------------------

\sss
\label{sss:assumption-tensor-Funct}
In this subsection we study relative tensor products and functor categories for strong representations. We fix an affine group scheme $H$ of finite type over $k$, such that the functor
$\Rep(H) \otimes_k \Rep(H) \to \Rep(H \times_k H)$ of
Lemma~\ref{lem:Rep-product} is an equivalence. By Lemma~\ref{lem:HC-product}, this implies that we also have $\hc_{H \times_k H} \simeq \hc_H \otimes_k \hc_H$.

\sss 
The results below will be deduced from an analogue in the present setting of Proposition~\ref{prop:Rep-H2-modules-OH}.
Consider the monad 
\[
m^H_\hc : \hc_H \otimes_k \hc_H \rightleftarrows \hc_H: (m^H_\hc)^R,
\]
where $m_\hc^H$ denotes the multiplication map. Here, $(m^H_\hc)^R$ is the right adjoint of $m^H_\hc$, which is $(\hc_H \otimes_k (\hc_H)^{\rev})$-linear since $\hc_H$ is rigid, see~\S\ref{sss:def-hc}.

On the other hand we consider the algebra $\sD(H)$ of differential operators on $H$. Since $H$ has two actions of $H$ (induced by multiplication on the right and on the left), $\sD(H)$ is endowed with two actions of $H$ and two algebra morphisms from $U(\fh)$, hence (using left and right multiplication on itself) four actions of $U(\fh)$, and these structures can be combined to give two structures of Harish-Chandra bimodules (see~\S\ref{sss:Dmod-strong-action}). More specifically, we let
$H^\rev \times_k H$ act on $H$ by $(g,h) \cdot k = g^{-1}kh^{-1}$ for $g,h,k \in H$, and therefore view $\sD(H)$ as an algebra object in $\hc_{H^\rev \times_k H}$. Our assumption implies that $\hc_{H^\rev \times_k H} \simeq \hc_{H^\rev} \otimes_k \hc_{H}$, and we have a tautological identification $\hc_{H^\rev} = (\hc_{H})^\rev$. (Here we consider the left action of $U(\fh^\rev)$ as a right action of $U(\fh)$, etc.) We can therefore consider $\sD(H)$ as an algebra object in $(\hc_H)^\rev \otimes_k \hc_H$.

\begin{Prop}
\label{prop:monadhc}
Under the assumption of~\S\ref{sss:assumption-tensor-Funct}, the monad 
\[
m_\hc^H : \hc_H \otimes_k \hc_H \rightleftarrows \hc_H: (m^H_\hc)^R
\]
yields an identification 
\[
\hc_H \simeq \LMod_{\sD(H)}((\hc_H)^\rev \otimes_k \hc_H).
\]
\end{Prop}

\begin{proof} 
By the discussion in~\S\ref{sss:rigidity}, it suffices to construct
an isomorphism of algebra objects $(m^H_\hc)^R( U(\fh)) \simeq \sD(H)$.
We have a commutative diagram
\[
\begin{tikzcd}[column sep=3cm]
\Rep(H) \otimes_k \Rep(H) \ar[r, "\on{ind}^H_{\hc} \otimes_k \on{ind}^H_{\hc}"] \ar[d, "m^H_{\Rep}"'] & \hc_H \otimes_k \hc_H \ar[d, "m^H_{\hc}"] \\ 
 \Rep(H) \ar[r, "\on{ind}^H_{\hc}"] & \hc_H
\end{tikzcd}
\]
expressing the monoidality of the functor $\on{ind}^H_{\hc}$, where we denote by $m^H_{\Rep}$ the binary product for $\Rep(H)$.
Passing to right adjoints we deduce a commutative diagram 
\begin{equation}
\label{eqn:cddiaghc}
\begin{tikzcd}[column sep=3cm]
\hc_H \ar[r, "(\on{ind}^H_\hc)^R"] \ar[d, "(m^H_{\hc})^R"'] & \Rep(H) \ar[d, "(m^H_{\Rep})^R"] \\ 
\hc_H \otimes_k \hc_H \ar[r, "(\on{ind}^H_\hc)^R \otimes_k (\on{ind}^H_\hc)^R"] & \Rep(H) \otimes_k \Rep(H). 
\end{tikzcd}
\end{equation}
Here in the right-hand column, as in the proof of Proposition~\ref{prop:Rep-H2-modules-OH}, $m^H_{\Rep}$ identifies with restriction along the diagonal embedding $\Delta : H \hookrightarrow H \times_k H$, and $(m^H_{\Rep})^R$ identifies with its right adjoint $\coind_{\Delta H}^{ H \times_k H}$, see~\S\ref{sss:Res-Coind}.
Using this diagram, we obtain an identification
\[
((\on{ind}^H_\hc)^R \otimes_k (\on{ind}^H_\hc)^R) \circ (m^H_{\hc})^R(U(\fh)) \simeq \coind_{\Delta H}^{ H \times_k H} (U(\fh)).
\]

Now we have a canonical isomorphism
\[
\sD(H) \simto \coind_{\Delta H}^{H \times_k H} (U(\fh))
\]
in $\Rep(H) \otimes_k \Rep(H) \simeq \Rep(H \times_k H)$.
Indeed, this morphism is obtained by adjunction from the natural maps 
\[
\sD(H) \rightarrow \sD(H) \otimes_{\OO(H)} k_e \xleftarrow{\sim} 
U(\fh),
\]
where $k_e$ denotes the skyscraper $\OO(H)$-module at the identity $e \in H$, and the map
\[
U(\fh)
\rightarrow \sD(H) \otimes_{\OO(H)} k_e
\]
is induced by 
left multiplication of $H$ on itself.

It remains to identify the algebra structure on $\sD(H)$ with its usual one. However, as in~\S\ref{sss:Amod-strong-reps} we have an equivalence
\[
\hc_H \simeq \LMod_{\sD(H)}((\hc_H)^\rev \otimes_k \hc_H) \simeq \LMod_{{\sD(H)}}(\Rep(H) \otimes_k \Rep(H)),
\]
so that the upper arrow in~\eqref{eqn:cddiaghc} defines a functor
\[
\LMod_{{\sD(H)}}(\Rep(H) \otimes_k \Rep(H)) \to \Mod_{\OO(H)}(\Rep(H) \otimes_k \Rep(H)).
\]
The image of $\sD(H)$ under this map provides an $\OO(H)$-module structure on $\sD(H)$ which commutes with right multiplication, hence is induced by
a map of algebra objects
\[
\OO(H) \rightarrow \sD(H)
\]
in $\Rep(H) \otimes_k \Rep(H)$.
We also have the unit map $U(\fh^\rev \oplus \fh) \rightarrow \sD(H)$, and the usual algebra structure on $\sD(H)$ is the unique algebra structure on it in $\hc_H \otimes_k \hc_H$ compatible with the above two maps. Indeed, this reduces to the analogous assertion in $\hc_{H \times_k H}^\heartsuit$ for t-structure reasons. There it follows e.g. by choosing a side, say left translations, using that the multiplication map $U(\fh) \otimes_k \OO(H) \rightarrow \sD(H)$ is an isomorphism of vector spaces, and that the commutators of elements of $U(\fh)$ and $\OO(H)$ are dictated by the structure as a Harish-Chandra bimodule. 
\end{proof}

\sss
The following statements follow, in view of the discussion in~\S\ref{sss:rigidity-tens-Hom}.

\begin{Prop} 
\label{prop:tensor-Funct-strong-actions}
Under the assumption of~\S\ref{sss:assumption-tensor-Funct},
if $\sM$, resp.~$\sN$, is a right, resp.~left, strong categorical representation of $H$, we have an identification
\[
\sM \otimes_{\hc_H} \sN \simeq \LMod_{\sD(H)}(\sM \otimes_k \sN).
\]
If $\sM$ and $\sN$ are left strong categorical representations of $H$, we have an identification
\[
\Funct_{\hc_H}(\sM,\sN) \simeq \RMod_{\sD(H)}( \Funct_k(\sM,\sN)).
\]
In case $\sM$ is dualizable, we have a canonical identification
\[
\Funct_{\hc_H}(\sM, \sN) \simeq \RMod_{\sD(H)}( \sM^\vee \otimes_k \sN).
\]
\end{Prop}

%------------------------------------------------------
\subsection{Equivariantization and deequivariantization}
\label{ss:equiv-deequiv-strong}
%------------------------------------------------------

\sss
We continue with our affine group scheme $H$ of finite type over $k$, and assume that the functor
$\Rep(H) \otimes_k \Rep(H) \to \Rep(H \times_k H)$ of
Lemma~\ref{lem:Rep-product} is an equivalence.
Consider the $\hc_H$-module $\fh\mod$, see~\S\ref{sss:h-mod}, and denote by $\Dmod(H)$ the $\infty$-category of (left) crystalline D-modules on the smooth affine scheme $H$, i.e.~of (left) modules for the algebra $\sD(H)$.

The following statement is a counterpart in our setting of a standard result when $\mathrm{char}(k)=0$, see e.g.~\cite[Proposition~2.5.3]{tao}.

\begin{Lem}
\label{lem:Dmod-Funct-hc}
There is a canonical equivalence
\[
\Dmod(H) \simeq \Funct_{\hc_H}(\fh\mod, \fh\mod).
\]
\end{Lem}

\begin{proof}
By Proposition~\ref{prop:tensor-Funct-strong-actions} we have
\[
\Funct_{\hc_H}(\fh\mod, \fh\mod) \simeq \RMod_{\sD(H)}(\Funct_k(\fh\mod, \fh\mod)).
\]
Now, by the results in~\cite[Chap.~1, \S 8.5--8.6]{gr} we have $\Funct_k(\fh\mod, \fh\mod) \simeq (\fh \oplus \fh^\rev)\mod$. Here $\sD(H)$ is seen as an algebra object in $U(\fh)$-bimodules, via left and right multiplication of $H$ on itself. By the same arguments as in~\S\ref{sss:Amod-strong-reps} we have an identification
\[
\RMod_{\sD(H)}((\fh \oplus \fh^\rev)\mod) \simeq \RMod_{\sD(H)}(\Vect_k).
\]
Finally, we have a canonical identification $\sD(H) \simeq \sD(H)^\rev$ (given by the identity on $\OO(H)$ and $-\id$ on $\fh$ seen as left invariant vector fields), which allows us to identify the right-hand side with $\Dmod(H)$.
\end{proof}

\sss
By the discussion in~\S\ref{sss:Funct}, the equivalence of Lemma~\ref{lem:Dmod-Funct-hc} endows $\Dmod(H)$ with the structure of an algebra object in $\DGCat$, which allows us to consider the $\infty$-category $\Dmod(H)\mod$. We expect that the associated binary operation
\[
\Dmod(H \times_k H) \simeq \Dmod(H) \otimes_k \Dmod(H) \to \Dmod(H)
\]
is given by the usual bimodule defining pushforward for $D$-modules associated with the multiplication morphism $H \times_k H \to H$, but we will not consider this question, in particular for lack of an adequate reference for a sufficiently robust theory of $D$-modules on schemes in positive characteristic.

The identification~\eqref{eqn:hmod-induction-Vect} also provides an algebra morphism
\[
\QCoh(H) \overset{\eqref{eqn:End-Vect-Rep}}{=} \Funct_{\Rep(H)}(\Vect_k, \Vect_k) \to \Funct_{\hc_H}(\fh\mod, \fh\mod) = \Dmod(H).
\]
Concretely, this functor is given by the tensor product operation $\sD(H) \otimes_{\OO(H)} (-)$.

\sss
By the definitions, $\fh\mod$ has commuting actions of $\hc_H$ and $\Dmod(H)$. 
Consider the associated adjunction
\[
(-) \otimes_{\Dmod(H)} \fh\mod : \modr( \Dmod(H) ) \rightleftarrows \hc_H\mod : \Funct_{\hc_H}(\fh\mod, -).
\]
By rigidity of $\hc_H$ we have an identification $\Funct_{\hc_H}(\fh\mod, -) \simeq (\fh\mod)^\vee \otimes_{\hc_H} (-)$, so that this functor is continuous and $\DGCat$-linear.

The same arguments as for Proposition~\ref{prop:QCoh-Rep-ff} allow us to prove the following counterpart for strong representations.

\begin{Prop}
\label{prop:Dmod-HC-ff}
The functor
\[
(-) \otimes_{\Dmod(H)} \fh\mod : \modr(\Dmod(H)) \to \hc_H\mod
\]
is fully faithful; in other words it
exhibits $\modr(\Dmod(H))$ as a full subcategory of $\hc_H\mod$.
\end{Prop}

\begin{Rem}
Recall the functors of Proposition~\ref{prop:QCoh-Rep-ff} and~\S\ref{sss:restriction-induction-weak-strong}.
In view of the identification~\eqref{eqn:hmod-induction-Vect}, we have a commutative diagram
\[
\begin{tikzcd}[column sep=3cm]
\modr(\Dmod(H)) \ar[r, "(-) \otimes_{\Dmod(H)} \fh\mod"] & \hc_H\mod \\
\modr\QCoh(H) \ar[u, "(-) \otimes_{\QCoh(H)} \Dmod(H)"] \ar[r, "(-) \otimes_{\QCoh(H)} \Vect_k"] & \Rep(H)\mod, \ar[u, "\hc_H \otimes_{\Rep(H)} (-)"']
\end{tikzcd}
\]
and hence also a similar commutative diagram for the right adjoints of these functors.
\end{Rem}

%--------------------------------------------------
\subsection{Categorical traces}
\label{sss:traces-strong}
%--------------------------------------------------

\sss
We consider once again an affine group scheme $H$ of finite type over $k$, and assume that the functor
$\Rep(H) \otimes_k \Rep(H) \to \Rep(H \times_k H)$ of
Lemma~\ref{lem:Rep-product} is an equivalence.
In this subsection we establish counterparts for strong categorical representations of the results of~\S\ref{ss:categorical-trace-weak}. These counterparts will be slightly different, because it is \emph{not} the case that a group homomorphism induces a corresponding algebra map between the associated $\infty$-categories of Harish-Chandra bimodules.

\sss
Given $k$-group schemes of finite type $H_1$, $H_2$ and a morphism $\varphi : H_1 \to H_2$, as noted above there is no naturally induced algebra morphism $\hc_{H_2} \to \hc_{H_1}$. However there exists a natural $(\hc_{H_1}, \hc_{H_2})$-bimodule, namely the $\infty$-category $(\fh_1 \oplus \h_2^\rev, H_1^\varphi)\mod$ of Harish-Chandra modules associated with $H_1 \times H_2^\rev$ and its subgroup $H_1^\varphi$ given by the image of the map $H_1 \to H_1 \times H_2^\rev$ defined by $h \mapsto (h,\varphi(h)^{-1})$.

In particular, if $H_1=H_2=H$ we obtain in this way an $\hc_H$-bimodule, and following~\cite[\S 7.3.1]{zhu} we can consider its Hochschild homology
\[
\Tr(\hc_H,\varphi) = \hc_H \otimes_{\hc_H \otimes_k (\hc_H)^\rev} \left( (\fh \oplus \fh^\rev, H^\varphi)\mod \right).
\]
Using Proposition~\ref{prop:monadhc}, one sees that
\[
\Tr(\hc_H,\varphi) \simeq \LMod_{\sD(H)}((\fh \oplus \fh^\rev, H^\varphi)\mod).
\]

\sss
Consider the composition of conservative functors
\[
\LMod_{\sD(H)}((\fh \oplus \fh^\rev, H^\varphi)\mod) \to (\fh \oplus \fh^\rev, H^\varphi)\mod \xrightarrow{\res_H^{(\fh \oplus \fh^\rev, H^\varphi)}} \Rep(H).
\]
Applying the Barr--Beck--Lurie theorem~\cite[Chap.~1, Proposition~3.7.7]{gr} we deduce an equivalence
\[
\Tr(\hc_H,\varphi) \simeq \LMod_{\overline{\sD(H)}^\varphi}(\Rep(H))
\]
where $\overline{\sD(H)}^\varphi$ is the monad $M \mapsto \sD(H) \otimes_{U(\fh)} M$ on $\Rep(H)$, and we still use the notation $\LMod$ for modules over a monad. (Here the $U(\fh)$-action on $M$ is obtained by differentiation from the $H$-action, and the morphism $U(\fh) \to \sD(H)$ is induced by the action of $H$ on itself given by $h \cdot k = \varphi(h)kh^{-1}$.)

The natural morphism $\OO(H) \to \sD(H)$ induces a morphism of monads from the monad associated with the algebra object $\OO(H)$ (where $H$ acts on itself as above) to $\overline{\sD(H)}^\varphi$; we deduce a canonical map
\[
\LMod_{\overline{\sD(H)}^\varphi}(\Rep(H)) \to \LMod_{\OO(H)}(\Rep(H)).
\]
As explained in~\S\ref{sss:trace-endomorphism}, the codomain of this map identifies with $\Tr(\Rep(H),\varphi)$; we have therefore constructed a canonical conservative functor
\begin{equation}
\label{eqn:morphism-traces}
\Tr(\hc_H,\varphi) \to \Tr(\Rep(H),\varphi).
\end{equation}

\sss
\label{sss:trace-hc-id}
Consider the setting and assumptions of~\S\ref{sss:trace-Rep-id}, and the constructions above with $\varphi=\id$. In this case, one can think of $\Tr(\hc_H,\id)$ as a (renormalized) $\infty$-category of $D$-modules on the adjoint quotient $H/H$. In fact, by a counterpart of Lemma~\ref{lem:t-structures} for monads, this $\infty$-category admits a t-structure, whose heart consists of $H$-modules $M$ equipped with an action map $\sD(H) \otimes_{U(\fh)} M \to M$ (where here we consider usual tensor products of modules over an algebra), i.e.~strongly $H$-equivariant $D$-modules on $H$. In view of Proposition~\ref{prop:trace-Rep-id}, the map~\eqref{eqn:morphism-traces} provides a ``forgetful'' functor
\[
\Tr(\hc_H,\id) \to \IndCoh(H/H).
\]

\sss
Consider now the setting of~\S\ref{sss:trace-Rep-Frob}, so $\varphi$ is the Frobenius morphism $\Frob$. In this case the morphism $U(\fh) \to \sD(H)$ coincides with the morphism induced by the right regular action of $H$ on itself (because the differential of $\Frob$ is trivial), so it induces an isomorphism $\OO(H) \otimes_k U(\fh) \simto \sD(H)$. As a consequence the monad $\overline{\sD(H)}^\Frob$ identifies with the monad associated with the algebra $\OO(H)$, so that~\eqref{eqn:morphism-traces} is an equivalence.
The answer we get for the trace is therefore that given by Proposition~\ref{prop:trace-Rep-Frob}.

%%%%%%%%%%%%%%%%%%%%%%%%%%%%%%%%%%%%
%%%%%%%%%%%%%%%%%%%%%%%%%%%%%%%%%%%%

\end{document}